\UseRawInputEncoding
\documentclass[11pt, reqno]{amsart}
\usepackage{amsfonts,latexsym,enumerate}
\usepackage{amsmath}
\usepackage{amscd} 
\usepackage{float,amsmath,amssymb,mathrsfs,bm,multirow,graphics}
\usepackage[dvips]{graphicx}
\usepackage[percent]{overpic}
\usepackage{amsaddr}
\usepackage[numbers,sort&compress]{natbib}
\usepackage{xcolor}
\usepackage{todonotes}

\newcommand{\R}{{\Bbb R}}

\newcommand{\C}{{\Bbb C}}
\newcommand{\D}{{\Bbb D}}

\newcommand{\tr}{\text{\upshape tr\,}}
\newcommand{\res}{\text{\upshape Res\,}}
\newcommand{\diag}{\text{\upshape diag\,}}
\newcommand{\re}{\text{\upshape Re\,}}
\newcommand{\im}{\text{\upshape Im\,}}
\newcommand{\ntlim}{\lim^\angle}
\DeclareMathOperator{\sech}{sech}

\newcommand{\sol}{\text{\upshape sol}}

\allowdisplaybreaks

\newtheorem{theorem}{Theorem}[section]

\newtheorem{lemma}[theorem]{Lemma}

\newtheorem{remark}[theorem]{Remark}

\newtheorem{RHproblem}[theorem]{RH problem}
\newtheorem{figuretext}{Figure}

\numberwithin{equation}{section}

\usepackage[colorlinks=true]{hyperref}
\hypersetup{urlcolor=blue, citecolor=red, linkcolor=blue}

\usepackage{tikz}
\usepackage{pict2e}
\usetikzlibrary{arrows,calc,chains, positioning, shapes.geometric,shapes.symbols,decorations.markings,arrows.meta}
\usetikzlibrary{patterns,angles,quotes}
\usetikzlibrary{decorations.markings}
\tikzset{middlearrow/.style={
			decoration={markings,
				mark= at position 0.6 with {\arrow{#1}} ,
			},
			postaction={decorate}
		}
	}
\tikzset{->-/.style={decoration={
				markings,
				mark=at position #1 with {\arrow{latex}}},postaction={decorate}}}
	
\tikzset{-<-/.style={decoration={
				markings,
				mark=at position #1 with {\arrowreversed{latex}}},postaction={decorate}}}
				
				\tikzset{
	master/.style={
		execute at end picture={
			\coordinate (lower right) at (current bounding box.south east);
			\coordinate (upper left) at (current bounding box.north west);
		}
	},
	slave/.style={
		execute at end picture={
			\pgfresetboundingbox
			\path (upper left) rectangle (lower right);
		}
	}
}
\tikzset{cross/.style={cross out, draw, 
         minimum size=2*(#1-\pgflinewidth), 
         inner sep=0pt, outer sep=0pt}}

\def\XXint#1#2#3{{\setbox0=\hbox{$#1{#2#3}{\int}$ }
\vcenter{\hbox{$#2#3$ }}\kern-.59\wd0}}

\usepackage[french,english]{babel}
\makeatletter

\newcommand{\abstractin}[1]{%
  \otherlanguage{#1}%
  \item[\hskip\labelsep\scshape\abstractname.]%
}
\makeatother

\title[The soliton resolution conjecture for the Boussinesq equation]
{The soliton resolution conjecture \\ for the Boussinesq equation}

\author{C. Charlier$^{1}$ and J. Lenells$^{2}$}

\address{$^{1}$Centre for Mathematical Sciences, Lund University,
22100 Lund, Sweden. \\
$^{2}$Department of Mathematics, KTH Royal Institute of Technology, \\
10044 Stockholm, Sweden.}
\email{christophe.charlier@math.lu.se}
\email{jlenells@kth.se}

\begin{document}

\begin{abstract}
We analyze the Boussinesq equation on the line with Schwartz initial data belonging to the physically relevant class of global solutions. In a recent paper, we determined ten main asymptotic sectors describing the large $(x,t)$-behavior of the solution, and for each of these sectors we provided the leading order asymptotics in the case when no solitons are
present. In this paper, we give a formula valid in the asymptotic sector $x/t \in (1,M]$, where $M$ is a large positive constant, in the case when solitons are present. Combined with earlier results, this validates the soliton resolution conjecture for the Boussinesq equation everywhere in the $(x,t)$-plane except in a number of small transition zones.

\vspace{0.3cm}\abstractin{french}
Nous analysons l'\'{e}quation de Boussinesq sur la ligne avec des donn\'{e}es initiales de Schwartz appartenant \`{a} la classe physiquement pertinente de solutions globales. Dans un article r\'{e}cent, nous avons d\'{e}termin\'{e} dix principaux secteurs asymptotiques d\'{e}crivant le comportement de la solution pour $(x,t)$ grand, et pour chacun de ces secteurs, nous avons fourni les termes asymptotiques dominants dans le cas o\`{u} aucun soliton n'est pr\'{e}sent. Dans cet article, nous donnons une formule valable dans le secteur asymptotique $x/t \in (1,M]$, o\`{u} $M$ est une grande constante positive, dans le cas o\`{u} des solitons sont pr\'{e}sents. Combin\'{e} avec des r\'{e}sultats ant\'{e}rieurs, cela valide la conjecture de r\'{e}solution des solitons pour l'\'{e}quation de Boussinesq partout dans le plan $(x,t)$, \`{a} l'exception d'un certain nombre de petites zones de transition.
\end{abstract}

\maketitle




\section{Introduction}
This paper is concerned with the asymptotic behavior as $t \to \infty$ of solutions of the Boussinesq \cite{B1872} equation 
\begin{align}\label{boussinesq}
u_{tt} = u_{xx} + (u^2)_{xx} + u_{xxxx},
\end{align}
which models two-way propagation of dispersive waves of small amplitude in shallow water \cite{J1997}. A Lax pair for (\ref{boussinesq}) was constructed in \cite{Z1974}, soliton solutions were studied in \cite{H1973, B1976, BZ2002}, and the existence of global solutions of certain initial-boundary value problems was investigated in \cite{KL1977, LS1985, Y2002}. Determining the long-time asymptotics for the solution of \eqref{boussinesq} was listed as an important open problem by Deift in 2006 \cite{D2008}; we recently addressed this problem in \cite{CLmain}.

Equation (\ref{boussinesq}) is linearly unstable due to exponentially growing Fourier components of large frequency. As a consequence, it is ill-posed and is also referred to as the ``bad" Boussinesq equation. 
The ``good'' Boussinesq equation is the equation obtained by changing the sign of the $u_{xxxx}$-term in (\ref{boussinesq}); this change of signs makes the equation locally well-posed and papers dedicated to the study of the ``good'' Boussinesq equation include \cite{BFH2019, BS1988, CLgoodboussinesq, CLWasymptotics, F2009, F2011, HM2015, KT2010, X2006, L1993, TM1991, LS1995, L1997, WZ2022}.

In \cite{CLmain}, we developed an inverse scattering approach to the initial value problem for \eqref{boussinesq}. Along with other results, we established in \cite{CLmain} uniqueness and existence of global solutions to \eqref{boussinesq} for a broad class of initial data relevant for water waves. For these global solutions, we identified ten main asymptotic sectors describing the large $(x,t)$-behavior of the solution (see Figure \ref{sectors.pdf}), and for each of these sectors we gave the leading order asymptotics in \cite[Theorem 2.14]{CLmain} in the case when no solitons are present. For conciseness, we omitted the proofs of some of these asymptotic formulas in \cite{CLmain}. 

In this paper, we establish an asymptotic formula for $u(x,t)$ valid in the sector $x/t \in (1,M]$, where $M$ is a large positive constant, in the case when solitons are present.
Following \cite{CLmain}, we refer to this sector as Sector II, see Figure \ref{sectors.pdf}. 
If no solitons are present, our formula reduces to the formula announced in \cite{CLmain}, thereby also providing a proof of this formula.

\begin{figure}
\bigskip\begin{center}
\begin{overpic}[width=.6\textwidth]{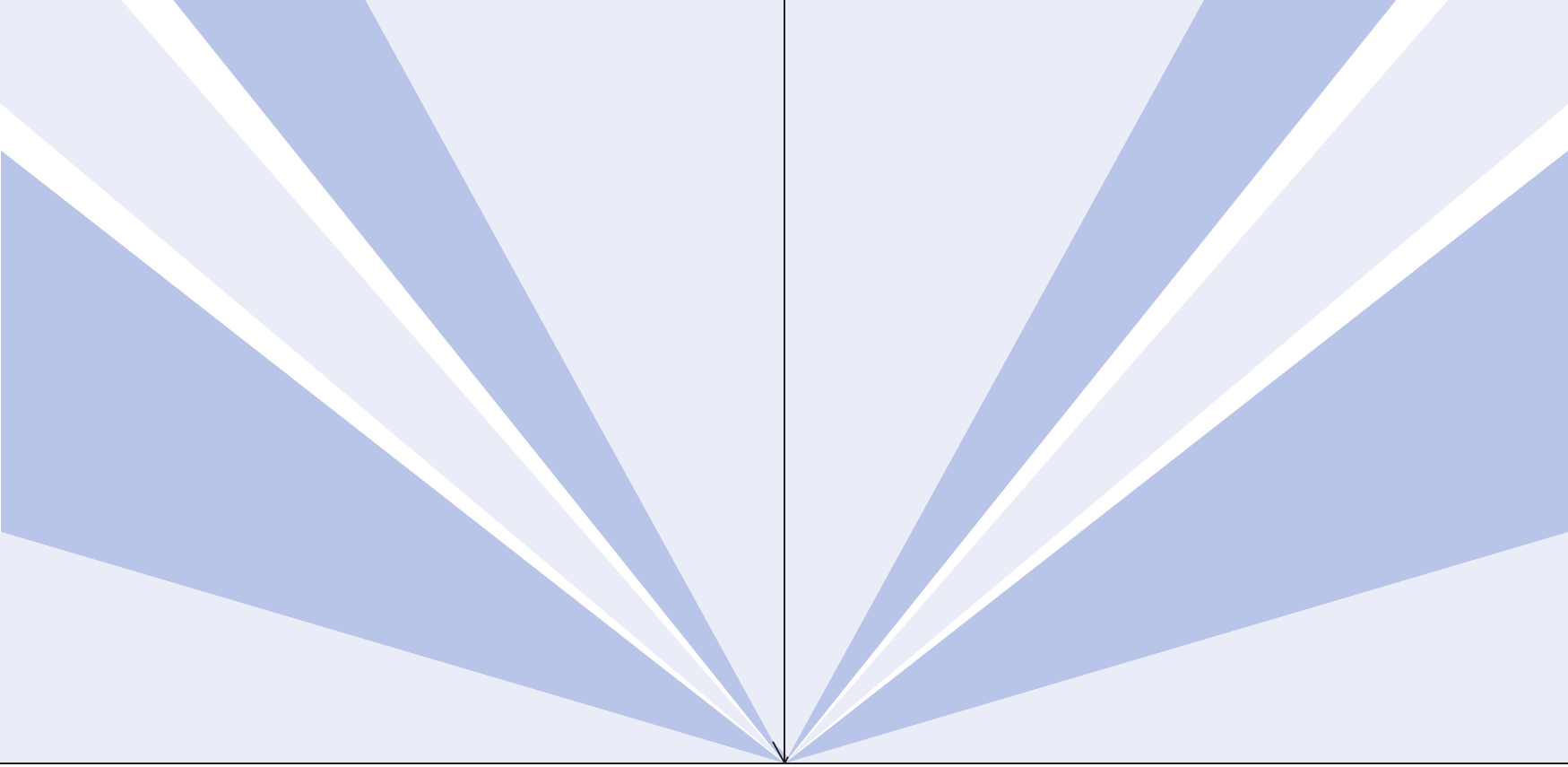}
      \put(102,-.3){\footnotesize $x$}
      \put(49.7,50.8){\footnotesize $t$}
      \put(90,4){\footnotesize I}
      \put(90,21){\footnotesize II}
      \put(86,36){\footnotesize III}
      \put(76,41){\footnotesize IV}
      \put(60.5,41){\footnotesize V}
      \put(37,41){\footnotesize VI}
      \put(19,41){\footnotesize VII}
      \put(9,36){\footnotesize VIII}
      \put(8,21){\footnotesize IX}
      \put(8,4){\footnotesize X}
    \end{overpic}
     \begin{figuretext}\label{sectors.pdf}
       The ten asymptotic sectors for (\ref{boussinesq}) in the $xt$-plane.
     \end{figuretext}
     \end{center}
\end{figure}

\subsection{The soliton resolution conjecture}
The soliton resolution conjecture for an integrable equation asserts, roughly speaking, that a solution of the equation with generic initial data eventually behaves as a finite number of solitons of order $1$ superimposed on a decaying radiative background. Earlier works on soliton resolution for equations with $2 \times 2$ Lax pairs include \cite{BJM2018, CL2021, CLL2000, CF2022, CGRT2023, HL2018, JLPS2018, JLPS2020, LTYF2022, LTY2022, PS2019, YF2021}. As mentioned above, the Boussinesq equation (\ref{boussinesq}) has the peculiar feature of being linearly unstable because large frequency Fourier modes grow (or decay) exponentially in time. 
The inverse scattering transform can be viewed as a nonlinear Fourier transform, and as such it allows us to identify the unstable high-frequency modes in a nonlinearly precise way \cite{CLmain}. The upshot is that if the initial data are such that unstable high-frequency modes are present, the solution will contain components that grow exponentially in time and the soliton resolution conjecture will automatically fail. 
This observation is consistent with the interpretation of the Boussinesq equation as a model for water waves, which relies on the absence of high-frequency modes.
Consequently, the natural formulation of the soliton resolution conjecture for (\ref{boussinesq}) restricts attention to initial data which contain no unstable high-frequency modes (in a nonlinear sense made precise in Assumption $(\ref{nounstablemodesassumption})$ below). Combined with the results of \cite{CLsectorI, CLsectorIV, CLsectorV}, our results here validate this soliton resolution conjecture for the Boussinesq equation everywhere in the $(x,t)$-plane except in a number of transition zones. In what follows, we explain this claim in further detail.

The one-soliton solutions of (\ref{boussinesq}) are given by 
\begin{align}\label{onesoliton}
u(x,t) = A \sech^2(\sqrt{A/6}(x - x_0 -ct)),
\end{align}
where $A \geq 0$ is the amplitude, $c = \pm \sqrt{1 + 2A/3}$ is the velocity, and $x_0 \in \R$ is a free parameter related to translation invariance \cite{H1973}. We observe that $c$ can be both positive and negative, corresponding to the fact that (\ref{boussinesq}) describes waves traveling in both directions. We also observe that $|c| > 1$, meaning that all one-solitons travel with speeds greater than $1$. One can show that all breather solitons of (\ref{boussinesq}) also propagate with speeds greater than $1$. Thus, by choosing $M$ large enough, we may assume that any finite number of right-moving solitons are contained in the sector $x/t \in (1,M]$, i.e., in Sector II. Similarly, we may assume that any finite number of left-moving solitons are contained in Sector IX, which is characterized by $x/t \in [-M,-1)$, see Figure \ref{sectors.pdf}. 
The soliton resolution conjecture for the Boussinesq equation then asserts that for generic initial data with no unstable high-frequency modes, the corresponding solution $u(x,t)$ of (\ref{boussinesq}) behaves as follows:
In Sectors II and IX, the large $t$ behavior of $u$ is described by a soliton component of order $1$ plus a decaying radiation component, whereas everywhere else, $u(x,t)$ asymptotes to a decaying radiative wave. Since (\ref{boussinesq}) is invariant under the reflection symmetry  $x \mapsto -x$, it is enough to consider positive values of $x$ when verifying this conjecture.

In \cite{CLsectorIV}, we showed that the long-time behavior of $u(x,t)$ in Sector IV, which is characterized by $x/t \in (1/\sqrt{3}, 1)$, is a radiative decaying wave described by a sum of two modulated sine-waves of order $O(t^{-1/2})$; we also observed that the only effect of solitons on the asymptotics in Sector IV is to introduce phase shifts in these modulated sine-waves.
In \cite{CLsectorV}, we obtained a similar asymptotic formula in Sector V. For brevity, we did not include solitons in \cite{CLsectorV}, but the same arguments as in \cite{CLsectorIV} show that the only effect of the solitons in Sector V is again to introduce phase shifts in the radiative terms.
In \cite{CLsectorI}, we obtained the asymptotics of $u(x,t)$ in Sector I. The formula in \cite{CLsectorI} was stated under the assumption of no solitons, but the arguments in the present paper show that they remain valid unchanged, provided that $M$ is large enough (so that all right-moving solitons are contained in Sector II). In the present paper, we obtain explicitly the leading order soliton component of order $1$ in Sector II, and show that the subleading term consists of a radiation component that decays like $t^{-1/2}$. Taken together, these results validate the soliton resolution conjecture for (\ref{boussinesq}) everywhere in the $(x,t)$-plane with the exception of a few narrow transition zones; more precisely, for any  $\epsilon > 0$, they validate it for all $\zeta := x/t \in \R$ such that $\zeta$ stays a distance $\epsilon$ away from the points $0, \pm 1/\sqrt{3}$, and $\pm 1$. We strongly expect that the soliton resolution conjecture holds also for $\zeta$ in the transition zones around $0, \pm 1/\sqrt{3}$, $\pm 1$, and that this can be proved using the framework of \cite{CLmain}; however, a treatment of these transition zones is beyond the scope of the present paper.

\subsection{Organization of the paper}
The main results are stated in Section \ref{mainsec}. The row-vector Riemann--Hilbert (RH) problem that underlies our inverse scattering approach to (\ref{boussinesq}) is recalled in Section \ref{overviewsec}. Sections \ref{nton1sec}--\ref{n3ton4sec} implement four transformations of this RH problem which bring it to a form suitable for the evaluation of its large $t$ behavior. In Section \ref{localparametrixsec}, we construct a local parametrix and in Section \ref{n3tonhatsec} we transform the RH problem one last time to arrive at a small-norm problem. The asymptotic behavior of $u(x,t)$ in Sector II is obtained in Section \ref{usec}. In Section \ref{puresolitonsec}, we solve the RH problem corresponding to pure multi-soliton solutions of \eqref{boussinesq}.
Another RH problem relevant for the asymptotics is analyzed in Section \ref{Msolsec}. Appendix \ref{modelapp} considers a model RH problem that is used in the construction of the local parametrix.

\section{Main results}\label{mainsec}
In this paper, the initial data 
\begin{align}\label{initial data}
u_{0}(x):=u(x,0) \quad \text{and} \quad u_{1}(x):=u_{t}(x,0)
\end{align}
are real-valued functions in the Schwartz class $\mathcal{S}(\mathbb{R})$ such that $\int_{\mathbb{R}}u_{1}(x)dx = 0$. The last assumption guarantees that the total mass $\int_{\mathbb{R}}u(x,t)dx$ is independent of $t$. It is shown in \cite{CLmain} that the solution to the initial value problem for \eqref{boussinesq} can be expressed in terms of the solution $n$ of a row-vector RH problem whose definition involves two spectral functions $r_{1}(k)$ and $r_{2}(k)$. The functions $r_{1}(k)$ and $r_{2}(k)$ are defined via a nonlinear Fourier transform of the initial data, which we briefly recall here for convenience.

\subsection{Definition of $r_{1}$ and $r_{2}$.} Let $\{l_j(k), z_j(k)\}_{j=1}^3$ be given by
\begin{align}\label{lmexpressions intro}
& l_{j}(k) = i \frac{\omega^{j}k + (\omega^{j}k)^{-1}}{2\sqrt{3}}, \qquad z_{j}(k) = i \frac{(\omega^{j}k)^{2} + (\omega^{j}k)^{-2}}{4\sqrt{3}}, \qquad k \in \C\setminus \{0\},
\end{align}
where $\omega = e^{\frac{2\pi i}{3}}$. Let $v_{0}(x) := \int_{-\infty}^{x}u_{1}(x')dx'$, and define $P(k)$ and $\mathsf{U}(x,k)$ by
\begin{align*}
P(k) = \begin{pmatrix}
1 & 1 & 1  \\
l_{1}(k) & l_{2}(k) & l_{3}(k) \\
l_{1}(k)^{2} & l_{2}(k)^{2} & l_{3}(k)^{2}
\end{pmatrix}, \quad \mathsf{U}(x,k) = P(k)^{-1} \begin{pmatrix}
0 & 0 & 0 \\
0 & 0 & 0 \\
-\frac{u_{0x}}{4}-\frac{iv_{0}}{4\sqrt{3}} & -\frac{u_{0}}{2} & 0
\end{pmatrix} P(k).
\end{align*} 
Let $\mathcal{L} = \diag(l_1 , l_2 , l_3)$, and let $X(x,k), X^A(x,k), Y(x,k), Y^A(x,k)$ be the unique solutions to the Volterra integral equations
\begin{align*}  
 & X(x,k) = I - \int_x^{\infty} e^{(x-x')\mathcal{L}(k)} (\mathsf{U}X)(x',k) e^{-(x-x')\mathcal{L}(k)} dx',
	\\
 & X^A(x,k) = I + \int_x^{\infty} e^{-(x-x')\mathcal{L}(k)} (\mathsf{U}^T X^A)(x',k)e^{(x-x')\mathcal{L}(k)} dx',	\\
& Y(x,k)  =  I  +  \int_{-\infty}^x  e^{(x-x')\mathcal{L}(k)} (\mathsf{U}Y)(x',k) e^{-(x-x')\mathcal{L}(k)}dx',
	\\ \nonumber
& Y^A(x,k)  =  I  -  \int_{-\infty}^x  e^{-(x-x')\mathcal{L}(k)} (\mathsf{U}^T Y^A)(x',k) e^{(x-x')\mathcal{L}(k)} dx',
\end{align*}
where $\mathsf{U}^T$ denotes the transpose of $\mathsf{U}$. Let $s(k)$ and $s^A(k)$ be given by 
\begin{align*}
& s(k) = I - \int_\R e^{-x \mathcal{L}(k)}(\mathsf{U}X)(x,k)e^{x \mathcal{L}(k)}dx, & & s^A(k) = I + \int_\R e^{x \mathcal{L}(k)}(\mathsf{U}^T X^A)(x,k)e^{-x \mathcal{L}(k)}dx.
\end{align*}
Let $\D = \{k \in \C \, | \, |k| < 1\}$ be the open unit disk and let $\Gamma = \cup_{j=1}^9 \Gamma_j$ be the contour displayed in Figure \ref{fig: Dn}. Let $\hat{\Gamma}_{j} = \Gamma_{j} \cup \partial \D$ denote the union of $\Gamma_j$ and the unit circle $\partial \D$. The functions $\{r_j(k)\}_1^2$ are defined by
\begin{align}\label{r1r2def}
\begin{cases}
r_1(k) = \frac{(s(k))_{12}}{(s(k))_{11}}, & k \in \hat{\Gamma}_{1}\setminus (\mathcal{Q} \cup \{0\}),
	\\ 
r_2(k) = \frac{(s^A(k))_{12}}{(s^A(k))_{11}}, \quad & k \in \hat{\Gamma}_{4}\setminus (\mathcal{Q} \cup \{0\}),
\end{cases}
\end{align}	
where the set $\mathcal{Q} := \{\kappa_{j}\}_{j=1}^{6}$ consists of the six points $\kappa_{j} := e^{\frac{\pi i(j-1)}{3}}$, $j=1,\ldots,6$.

\subsection{Solitons}
In \cite{CLscatteringsolitons}, we generalized the approach of \cite{CLmain} to allow for solutions with solitons. The solitons manifest themselves in the formulation of the RH problem as a set of poles $\mathsf{Z}$ and a set of corresponding residue constants $\{c_{k_0}\}_{k_0 \in \mathsf{Z}} \subset \C$ which are defined as follows. 

The functions $s_{11}$ and $s^A_{11}$ appearing in (\ref{r1r2def}) are analytic in the interiors of the sets $\bar{D}_1 \cup \bar{D}_2$ and $\bar{D}_4 \cup \bar{D}_5$, respectively, where the open sets $\{D_{j}\}_{j=1}^{6}$ are as indicated in Figure \ref{fig: Dn}, see \cite{CLmain}. The solitons are generated by the possible zeros of $s_{11}(k)$ and $s^A_{11}(k)$.
We will restrict ourselves to the generic case when there are no zeros on the contour $\Gamma$. 
In view of the symmetries
 $s_{11}(k) = s_{11}(\omega/k)$ and $s_{11}^A(k) = \overline{s_{11}(\bar{k}^{-1})}$ which were established in \cite{CLmain}, it is then sufficient to consider the zeros of $s_{11}$ in $D_2$. 
The set $\mathsf{Z}$ is defined as the zero-set of $s_{11}(k)$. Since $s_{11}(k) \to 1$ as $k\to \infty$, $\mathsf{Z}$ is necessarily a finite set.
We will consider the generic case when the zeros of $s_{11}$ are simple and $s_{22}^A \neq 0$ at every point of $\mathsf{Z} \setminus \R$.
 
We divide the open set $D_{2}$ into three parts: $D_{2} = D_{\mathrm{reg}} \sqcup D_{\mathrm{sing}} \sqcup (D_{2}\cap \R)$, where $D_{2}\cap \R = (-1,0) \cup (1,\infty)$ and
\begin{align*}
& D_{\mathrm{reg}} := D_{2} \cap \big( \{k \,|\, |k|>1, \im k >0\} \cup \{k \,|\, |k|<1, \im k <0\} \big), \\
& D_{\mathrm{sing}} := D_{2} \cap \big( \{k \,|\, |k|>1, \im k <0\} \cup \{k \,|\, |k|<1, \im k >0\} \big).
\end{align*}
If $u_0$ and $u_1$ are compactly supported, the residue constants $\{c_{k_0}\}_{k_0 \in \mathsf{Z}}$ are defined by
 \begin{align}\label{ck0compactdef}
& c_{k_{0}} := \begin{cases} 
- \frac{s_{13}(k_0)}{\dot{s}_{11}(k_0)}, & k_{0} \in \mathsf{Z}\setminus \mathbb{R}, \\
- \frac{s_{12}(k_0)}{\dot{s}_{11}(k_0)}, & k_{0} \in \mathsf{Z}\cap \mathbb{R}.
\end{cases}
\end{align}
If $u_0$ and $v_0$ are not compactly supported, then the expressions in (\ref{ck0compactdef}) are in general not well-defined and we instead use the following more complicated definition of the residue constants: Define the vector-valued function $w(x,k)$ by
\begin{align}\label{wdef}
w = \begin{pmatrix}
Y_{21}^AX_{32}^A - Y_{31}^AX_{22}^A  \\
Y_{31}^AX_{12}^A - Y_{11}^AX_{32}^A  \\
 Y_{11}^AX_{22}^A - Y_{21}^AX_{12}^A 
\end{pmatrix}.
\end{align}
If $k_{0} \in \mathsf{Z}\setminus \mathbb{R}$, then $c_{k_0} \in \C$ is defined as the unique constant such that
\begin{subequations}\label{ck0def}
\begin{align}\label{ck0def1}
\frac{w(x,k_{0})}{\dot{s}_{11}(k_0)} = c_{k_0} e^{(l_1(k_0)-l_3(k_0))x} [X(x,k_0)]_1 \qquad \text{for all $x \in \R$},
\end{align}
where we write $[A]_j$ for the $j$th column of an $n \times m$ matrix $A$. If $k_{0} \in \mathsf{Z} \cap \mathbb{R}$, then $c_{k_0} \in \C$ is defined as the unique constant such that
\begin{align}\label{ck0def2}
\frac{[Y(x,k_0)]_2}{\dot{s}_{22}^A(k_0)} = c_{k_0} e^{(l_1(k_0)-l_2(k_0))x} [X(x,k_0)]_1\qquad \text{for all $x \in \R$}.
\end{align}
\end{subequations}
It is proved in \cite{CLscatteringsolitons} that the $c_{k_0}$ are well-defined by the relations in (\ref{ck0def}).

Simple zeros of $s_{11}(k)$ in $\mathsf{Z} \setminus \R$ with positive and negative real parts give rise to right- and left-moving breather solitons, respectively \cite{CLscatteringsolitons}. 
Similarly, positive and negative simple zeros of $s_{11}(k)$ in $\mathsf{Z} \cap \R$ give rise to right- and left-moving bell-shaped solitons, respectively. 
It is shown in \cite{CLscatteringsolitons} that zeros in $D_{\mathrm{sing}}$ generate solitons with singularities. It is also shown in \cite{CLscatteringsolitons} that a soliton corresponding to $k_0 \in \mathsf{Z} \cap \R$ is non-singular and real-valued if and only if the associated residue constant $c_{k_0}$ satisfies $i(\omega^{2}k_{0}^{2} - \omega)c_{k_{0}} \geq 0$.
Since we are only interested in singularity-free real-valued solutions, we will assume that $\mathsf{Z} \subset D_{\mathrm{reg}} \cup (-1,0) \cup (1,\infty)$ and that $i(\omega^{2}k_{0}^{2} - \omega)c_{k_{0}} \geq 0$ for each $k_0 \in \mathsf{Z} \cap \R$. 


\subsection{Assumptions}\label{assumptionssubsec}
Our results will hold under the following assumptions:

\begin{enumerate}[$(i)$]
\item\label{solitonassumption} Finite number of non-singular solitons: assume that $s_{11}(k)$ has a simple zero at each point in $\mathsf{Z}$, where $\mathsf{Z} \subset D_{\mathrm{reg}} \cup (-1,0) \cup (1,\infty)$ is a finite set, and that $s_{11}(k)$ has no other zeros in $\bar{D}_2 \cup \partial \D \setminus (\mathcal{Q} \cup \{0\})$. If $k_0 \in \mathsf{Z} \setminus \R$, then suppose that $s_{22}^A(k_0) \neq 0$. If $k_0 \in \mathsf{Z} \cap \R$, then suppose that $i(\omega^{2}k_{0}^{2} - \omega)c_{k_{0}} \geq 0$. 

\item Generic behavior of $s$ and $s^{A}$  near $k= \pm 1$: we assume for $k_{\star} =1$ and $k_{\star}=-1$ that
\begin{align*}
& \lim_{k \to k_{\star}} (k-k_{\star}) s(k)_{11} \neq 0, & & \hspace{-0.1cm} \lim_{k \to k_{\star}} (k-k_{\star}) s(k)_{13} \neq 0, & & \hspace{-0.1cm} \lim_{k \to k_{\star}} s(k)_{31} \neq 0, & & \hspace{-0.1cm} \lim_{k \to k_{\star}} s(k)_{33} \neq 0, \\
& \lim_{k \to k_{\star}} (k-k_{\star}) s^{A}(k)_{11} \neq 0, & & \hspace{-0.1cm} \lim_{k \to k_{\star}} (k-k_{\star}) s^{A}(k)_{31} \neq 0, & & \hspace{-0.1cm} \lim_{k \to k_{\star}} s^{A}(k)_{13} \neq 0, & & \hspace{-0.1cm} \lim_{k \to k_{\star}} s^{A}(k)_{33} \neq 0.
\end{align*}

\item\label{nounstablemodesassumption}
 No unstable high-frequency modes: we assume that $r_{1}(k)=0$ for all $k \in [0,i]$, where $[0,i]$ is the vertical segment from $0$ to $i$.
\end{enumerate}

\begin{figure}
\begin{center}
\begin{tikzpicture}[scale=0.9]
\node at (0,0) {};
\draw[black,line width=0.45 mm,->-=0.4,->-=0.85] (0,0)--(30:4);
\draw[black,line width=0.45 mm,->-=0.4,->-=0.85] (0,0)--(90:4);
\draw[black,line width=0.45 mm,->-=0.4,->-=0.85] (0,0)--(150:4);
\draw[black,line width=0.45 mm,->-=0.4,->-=0.85] (0,0)--(-30:4);
\draw[black,line width=0.45 mm,->-=0.4,->-=0.85] (0,0)--(-90:4);
\draw[black,line width=0.45 mm,->-=0.4,->-=0.85] (0,0)--(-150:4);

\draw[black,line width=0.45 mm] ([shift=(-180:2.5cm)]0,0) arc (-180:180:2.5cm);
\draw[black,arrows={-Triangle[length=0.2cm,width=0.18cm]}]
($(3:2.5)$) --  ++(90:0.001);
\draw[black,arrows={-Triangle[length=0.2cm,width=0.18cm]}]
($(57:2.5)$) --  ++(-30:0.001);
\draw[black,arrows={-Triangle[length=0.2cm,width=0.18cm]}]
($(123:2.5)$) --  ++(210:0.001);
\draw[black,arrows={-Triangle[length=0.2cm,width=0.18cm]}]
($(177:2.5)$) --  ++(90:0.001);
\draw[black,arrows={-Triangle[length=0.2cm,width=0.18cm]}]
($(243:2.5)$) --  ++(330:0.001);
\draw[black,arrows={-Triangle[length=0.2cm,width=0.18cm]}]
($(297:2.5)$) --  ++(210:0.001);

\draw[black,line width=0.15 mm] ([shift=(-30:0.55cm)]0,0) arc (-30:30:0.55cm);

\node at (0.8,0) {$\tiny \frac{\pi}{3}$};

\node at (-1:2.85) {\footnotesize $\Gamma_8$};
\node at (60:2.85) {\footnotesize $\Gamma_9$};
\node at (116:2.8) {\footnotesize $\Gamma_7$};
\node at (181:2.83) {\footnotesize $\Gamma_8$};
\node at (234:2.71) {\footnotesize $\Gamma_9$};
\node at (300:2.83) {\footnotesize $\Gamma_7$};

\node at (77:1.45) {\footnotesize $\Gamma_1$};
\node at (160:1.45) {\footnotesize $\Gamma_2$};
\node at (-163:1.45) {\footnotesize $\Gamma_3$};
\node at (-77:1.45) {\footnotesize $\Gamma_4$};
\node at (-42:1.45) {\footnotesize $\Gamma_5$};
\node at (43:1.45) {\footnotesize $\Gamma_6$};

\node at (84:3.3) {\footnotesize $\Gamma_4$};
\node at (155:3.3) {\footnotesize $\Gamma_5$};
\node at (-156:3.3) {\footnotesize $\Gamma_6$};
\node at (-84:3.3) {\footnotesize $\Gamma_1$};
\node at (-35:3.3) {\footnotesize $\Gamma_2$};
\node at (35:3.3) {\footnotesize $\Gamma_3$};

\end{tikzpicture}
\hspace{0.5cm}
\begin{tikzpicture}[scale=0.9]
\node at (0,0) {};
\draw[black,line width=0.45 mm] (0,0)--(30:4);
\draw[black,line width=0.45 mm] (0,0)--(90:4);
\draw[black,line width=0.45 mm] (0,0)--(150:4);
\draw[black,line width=0.45 mm] (0,0)--(-30:4);
\draw[black,line width=0.45 mm] (0,0)--(-90:4);
\draw[black,line width=0.45 mm] (0,0)--(-150:4);

\draw[black,line width=0.45 mm] ([shift=(-180:2.5cm)]0,0) arc (-180:180:2.5cm);
\draw[black,line width=0.15 mm] ([shift=(-30:0.55cm)]0,0) arc (-30:30:0.55cm);

\node at (120:1.6) {\footnotesize{$D_{1}$}};
\node at (-60:3.7) {\footnotesize{$D_{1}$}};

\node at (180:1.6) {\footnotesize{$D_{2}$}};
\node at (0:3.7) {\footnotesize{$D_{2}$}};

\node at (240:1.6) {\footnotesize{$D_{3}$}};
\node at (60:3.7) {\footnotesize{$D_{3}$}};

\node at (-60:1.6) {\footnotesize{$D_{4}$}};
\node at (120:3.7) {\footnotesize{$D_{4}$}};

\node at (0:1.6) {\footnotesize{$D_{5}$}};
\node at (180:3.7) {\footnotesize{$D_{5}$}};

\node at (60:1.6) {\footnotesize{$D_{6}$}};
\node at (-120:3.7) {\footnotesize{$D_{6}$}};

\node at (0.8,0) {$\tiny \frac{\pi}{3}$};

\draw[fill] (0:2.5) circle (0.1);
\draw[fill] (60:2.5) circle (0.1);
\draw[fill] (120:2.5) circle (0.1);
\draw[fill] (180:2.5) circle (0.1);
\draw[fill] (240:2.5) circle (0.1);
\draw[fill] (300:2.5) circle (0.1);

\node at (0:2.9) {\footnotesize{$\kappa_1$}};
\node at (59.5:2.84) {\footnotesize{$\kappa_2$}};
\node at (120:2.85) {\footnotesize{$\kappa_3$}};
\node at (180:2.9) {\footnotesize{$\kappa_4$}};
\node at (240:2.85) {\footnotesize{$\kappa_5$}};
\node at (300:2.85) {\footnotesize{$\kappa_6$}};

\draw[dashed] (-4.3,-3.8)--(-4.3,3.8);

\end{tikzpicture}
\end{center}
\begin{figuretext}\label{fig: Dn}
The contour $\Gamma = \cup_{j=1}^9 \Gamma_j$ in the complex $k$-plane (left) and the open sets $D_{n}$, $n=1,\ldots,6$, together with the sixth roots of unity $\kappa_j$, $j = 1, \dots, 6$ (right).
\end{figuretext}
\end{figure}

Regarding Assumption $(iii)$, we note that equation (\ref{boussinesq}) possesses solutions that blow up at any given positive time \cite{CLmain}. Assumption $(iii)$ guarantees that the solution of the initial value problem \eqref{boussinesq}--\eqref{initial data} does not carry unstable nonlinear high-frequency modes and therefore exists globally. As discussed in the introduction, this assumption is physically motivated because the derivation of (\ref{boussinesq}) requires the wavelength to be long compared to the depth of the water; the equation therefore only remains valid as a model for water waves as long as the high-frequency modes are absent (or at least heavily suppressed). If Assumptions $(i)$--$(iii)$ hold true, then the solution $u(x,t)$ of the initial value problem \eqref{boussinesq}--\eqref{initial data} exists globally \cite{CLscatteringsolitons}. By \cite[Theorems 2.6 and 2.11]{CLscatteringsolitons}, there is a wide class of initial data satisfying $(i)$--$(iii)$.

\subsection{Statement of the main result}
Let $\zeta := x/t$ and suppose that $\zeta > 1$. For $1 \leq j<i \leq 3$, let $\Phi_{ij}(\zeta, k)$ be given by
\begin{align}\label{def of Phi ij}
\Phi_{ij}(\zeta,k) = (l_{i}(k)-l_{j}(k))\zeta + (z_{i}(k)-z_{j}(k)), \qquad k \in \mathbb{C}\setminus \{0\}.
\end{align}
The saddle points $\{k_{j}=k_{j}(\zeta)\}_{j=1}^{4}$ of $k\mapsto \Phi_{21}(\zeta,k)$ are given by
\begin{subequations}\label{def of kj}
\begin{align}
& k_{1} = \frac{1}{4}\bigg( \zeta - \sqrt{8+\zeta^{2}} + i \sqrt{2}\sqrt{4-\zeta^{2}+\zeta\sqrt{8+\zeta^{2}}} \bigg), && k_{2} = \bar{k}_1,  \\
& k_{3} = \frac{1}{4}\bigg( \zeta + \sqrt{8+\zeta^{2}} + \sqrt{2}\sqrt{-4+\zeta^{2}+\zeta\sqrt{8+\zeta^{2}}} \bigg), &&  k_{4}=k_{3}^{-1}, 
\end{align}
\end{subequations} 
and satisfy $|k_{1}|=1$, $k_{3}\in (1,+\infty)$, and $\arg k_{1} \in (\frac{\pi}{2},\frac{2\pi}{3})$. Define $z_{\star}$ by
\begin{align*}
& z_{\star} = z_{\star}(\zeta) := \sqrt{2}e^{\frac{\pi i}{4}} \sqrt{\frac{4-3k_{1} \zeta - k_{1}^{3} \zeta}{4k_{1}^{4}}},
\end{align*}
where the branch of the square root is such that $-ik_{1}z_{\star}>0$. 
Let $\Gamma(k)$ be the Gamma function, and define
\begin{align}\label{nudef}
& \nu = \nu(k_1) := - \frac{1}{2\pi}\ln(1+r_{1}(k_{1})r_{2}(k_{1})) \leq 0.
\end{align}
Define also
\begin{align}
\arg d_{0} = &\; \arg d_{0}(\zeta,t) := \nu  \ln \bigg| \frac{(\frac{1}{\omega^{2}k_{1}}-k_{1})(\frac{1}{\omega k_{1}}-k_{1})}{3(\frac{1}{k_{1}}-k_{1})^{2}z_{\star}^{2}} \bigg| - \nu\ln t \nonumber \\
& + \frac{1}{2\pi} \int_{i}^{k_{1}} \ln \bigg| \frac{(k_{1}-s)^{2}(\frac{1}{\omega^{2}k_{1}}-s)(\frac{1}{\omega k_{1}}-s)}{(\frac{1}{k_{1}}-s)^{2}(\omega k_{1}-s)(\omega^{2} k_{1}-s)} \bigg| d \ln(1+r_{1}(s)r_{2}(s)), \label{argd0}
\end{align}
where the path of the integral $\int_{i}^{k_{1}}$ starts at $i$, follows the unit circle in the counterclockwise direction, and ends at $k_{1}$.

Our main result establishes the behavior of $u(x,t)$ as $(x,t)\to \infty$ in Sector II. Its formulation involves the the unique solution $M_{\mathrm{sol}}(x,t,k)$ of RH problem \ref{RHS}, which is stated and discussed in Section \ref{Msolsubsec} below.

\begin{theorem}[Asymptotics in Sector II with solitons]\label{asymptoticsth}
Let $u_0,u_1 \in \mathcal{S}(\R)$ be real-valued initial data such that $\int_{\mathbb{R}}u_{1} dx =0$ and such that Assumptions $(i)$--$(iii)$ are fulfilled. Let $\mathcal{I}$ be a fixed compact subset of $(1,\infty)$. 
The global solution $u(x,t)$ of the initial value problem for (\ref{boussinesq}) with initial data $u_0, u_1$ enjoys the following asymptotics as $t \to \infty$:
\begin{align}\label{uasymptotics}
& u(x,t) = u_{\mathrm{sol}}(x,t) + \frac{u_{\mathrm{rad}}(x,t)}{\sqrt{t}} + O\bigg(\frac{\ln t}{t}\bigg), 
\end{align}
uniformly for $\zeta:= \frac{x}{t} \in \mathcal{I}$, where the soliton term $u_{\mathrm{sol}}$ is given by 
\begin{align}\label{usoldef}
u_{\mathrm{sol}}(x,t) := -i\sqrt{3}\frac{\partial}{\partial x} \lim_{k\to\infty}k\Big(\big((1,1,1)M_{\mathrm{sol}}(x,t,k)\big)_{3}-1\Big),
\end{align}
and the coefficient $u_{\mathrm{rad}}$ of the radiation term is given by
\begin{align}\nonumber
u_{\mathrm{rad}}(x,t) := &\; \frac{  \sqrt{6\pi}  e^{-\frac{\pi \nu}{2}}\im(k_1) }{-ik_{1}z_{\star}} 
(1,1,1)M_{\mathrm{sol}}(x,t,k_{1})
\begin{pmatrix}
0 & \hspace{-0.15cm}\frac{e^{\frac{\pi i}{4}} e^{-t\Phi_{21}(\zeta,k_{1})}}{e^{i \arg d_0} r_{2}(k_{1}) \Gamma(i\nu)} & \hspace{-0.15cm}0 \\
\frac{e^{-\frac{\pi i}{4}} e^{t \Phi_{21}(\zeta,k_{1})}}{e^{-i \arg d_0} r_{1}(k_{1}) \Gamma(-i\nu)} & \hspace{-0.15cm}0 & \hspace{-0.15cm}0 \\
0 & \hspace{-0.15cm}0 & \hspace{-0.15cm}0
\end{pmatrix}
	\\ \label{uraddef}
& \times M_{\mathrm{sol}}(x,t,k_{1})^{-1}\begin{pmatrix}
\omega^{2} i k_{1}^{-1} - \omega i k_{1}  \\ \omega ik_{1}^{-1} - \omega^{2}ik_{1} \\ ik_{1}^{-1}-ik_{1}
\end{pmatrix}.
\end{align}
\end{theorem}

The proof of Theorem \ref{asymptoticsth} is given in Sections \ref{overviewsec}--\ref{usec}, with the final steps presented in Section \ref{proofth1subsec}.

The asymptotic formula of Theorem \ref{asymptoticsth} is very general. However, the presence of the solution $M_{\mathrm{sol}}(x,t,k)$ of RH problem \ref{RHS} makes it somewhat challenging to interpret. In what follows, we therefore present more explicit versions of the asymptotic formula (\ref{uasymptotics}) which make the resolution into individual solitons more evident. We also give an explicit expression for the leading term $u_{\mathrm{sol}}(x,t)$.

\subsection{Formula valid away from the asymptotic solitons}
Away from the directions of the asymptotic solitons, the asymptotic formula (\ref{uasymptotics}) can be simplified as follows. 
Let $S_0 \subset (1,+\infty)$ be the (finite) set of all $\zeta \in (1, +\infty)$ for which there exists a $k_{0}$ such that either $k_{0} \in (\mathsf{Z}\setminus \R) \cap \{k\,|\, \re \Phi_{31}(\zeta,k) = 0\}$ or $k_{0} \in \mathsf{Z}\cap \R \cap \{k\,|\, \re \Phi_{21}(\zeta,k) = 0\}$.
Then $S_0$ is the set of velocities of all asymptotic solitons propagating to the right, see (\ref{S0explicit}). For each $\epsilon > 0$, let $S_\epsilon$ be the $\epsilon$-neighborhood of $S_0$ given by 
\begin{align}\label{Sdeltadef}
S_\epsilon = \bigcup_{\zeta_0 \in S_0}\{\zeta \in (1,+\infty) \, |\, |\zeta - \zeta_0| \leq \epsilon\}.
\end{align}
If $\zeta \in (1,+\infty) \setminus S_\epsilon$, then there is no asymptotic soliton traveling with velocity $\zeta$, so we expect the soliton term $u_{\mathrm{sol}}(x,t)$ in (\ref{uasymptotics}) to be exponentially small and the radiation term $u_{\mathrm{rad}}(x,t)/\sqrt{t}$ to simplify. Our next theorem, whose statement involves the function
\begin{align}\label{def of mathcalP}
\mathcal{P}(\zeta,k)  =&\; \prod_{\substack{k_{0}\in \mathsf{Z}\setminus \R \\ \re \Phi_{31}(\zeta,k_{0})<0}} \frac{(k-k_{0})(k-k_{0}^{-1})(k-\omega^{2} \bar{k}_{0})(k-\omega \bar{k}_{0}^{-1})}{(k-\bar{k}_{0})(k-\bar{k}_{0}^{-1})(k-\omega k_{0})(k-\omega^{2} k_{0}^{-1})} \nonumber \\
& \times  \prod_{\substack{k_{0}\in \mathsf{Z}\cap \R \\ \re \Phi_{21}(\zeta,k_{0})<0}} \frac{(k-\omega^{2}k_{0})(k-\omega k_{0}^{-1})}{(k-\omega k_{0})(k-\omega^{2}k_{0}^{-1})}, 
\end{align}
shows that this is indeed the case.

\begin{theorem}[Asymptotics in Sector II away from solitons]\label{asymptoticsth2}
Under the assumptions of Theorem \ref{asymptoticsth}, we have, for each $\epsilon > 0$,
\begin{align}\label{uasymptotics2}
& u(x,t) = \frac{A(\zeta)}{\sqrt{t}} \cos \alpha(\zeta,t) + O\bigg(\frac{\ln t}{t}\bigg) \qquad \text{as $t \to \infty$},
\end{align}
uniformly for $\zeta = \frac{x}{t} \in \mathcal{I} \setminus S_\epsilon$, where 
\begin{align}
& A(\zeta) := 2\sqrt{3} \frac{\sqrt{-\nu}\sqrt{-1-2\cos (2\arg k_{1})}}{-ik_{1}z_{\star}} \im k_{1}, \nonumber \\
& \alpha(\zeta,t) := \frac{3\pi}{4} + \arg r_{2}(k_{1}) + \arg \Gamma(i \nu) + \arg d_{0} + \arg \frac{\mathcal{P}(\zeta,\omega k_{1})}{\mathcal{P}(\zeta,\omega^{2} k_{1})} + t \, \im \Phi_{21}(\zeta,k_{1}). \nonumber
\end{align}
\end{theorem}

The proof of Theorem \ref{asymptoticsth2} is given in Section \ref{proofth2subsec}.

\subsection{Explicit expression for the leading term}
RH problem \ref{RHS} for $M_{\mathrm{sol}}(x,t,k)$ can be solved explicitly and this leads to an explicit expression for the leading term $u_{\mathrm{sol}}(x,t)$ in Theorem \ref{asymptoticsth}. In general, $u_{\mathrm{sol}}(x,t)$ is not a pure multi-soliton solution of (\ref{boussinesq}). However, it can be described as a modulated multi-soliton, meaning that it is a multi-soliton with parameters depending on the slowly varying variable $\zeta = x/t$. 

To present the expression for $u_{\mathrm{sol}}(x,t)$, we first state the formula for the pure multi-soliton of (\ref{boussinesq}).
Consider the following data: 
\begin{enumerate}
\item two integers $n_b \geq 0$ and $n_s \geq 0$, 

\item a finite set $\{\lambda_j\}_1^{n_b} \subset D_{\mathrm{reg}}$ and associated constants $\{c_{\lambda_j}\}_{j=1}^{n_b} \in \C$, 

\item a finite set $\{k_j\}_1^{n_s} \subset (-1,0) \cup (1,\infty)$ and associated constants $\{c_{k_j}\}_{j=1}^{n_s} \in \C$ satisfying $i(\omega^{2}k_j^{2} - \omega)c_{k_j} \geq 0$ for $j = 1, \dots, n_s$.
\end{enumerate}
We show in Lemma \ref{uslemma} that associated to each such collection of data, there is a real-valued Schwartz class solution $u_s$ of (\ref{boussinesq}) for $(x,t) \in \R \times [0, +\infty)$ given by
\begin{align}\label{usdef}
u_s\Big(x,t; (\lambda_j, c_{\lambda_j})_{j=1}^{n_b}, (k_j, c_{k_j})_{j=1}^{n_s}\Big)
:= 6 \frac{\partial^2}{\partial x^2} \ln \det (I-B),
\end{align}
where the matrix $B = B\big(x,t; (\lambda_j, c_{\lambda_j})_{j=1}^{n_b}, (k_j, c_{k_j})_{j=1}^{n_s}\big)$ is the $(2n_b + n_s) \times (2n_b + n_s)$-matrix given in block form as
\begin{align}\label{Bdef}
B := \begin{pmatrix} 
B_{\lambda \lambda} & B_{\lambda \bar{\lambda}} & B_{\lambda  k}
	\\
B_{\bar{\lambda} \lambda} & B_{\bar{\lambda} \bar{\lambda}} & B_{\bar{\lambda}  k}
	\\
B_{ k \lambda} & B_{ k \bar{\lambda}} & B_{k k}
\end{pmatrix}
\end{align}
with the submatrices
\begin{align*}
& (B_{\lambda \lambda} )_{jl} 
= \frac{\tilde{c}_{\lambda_l} e^{x(l_1(\lambda_j) - l_3(\lambda_l)) + t(z_1(\lambda_j) - z_3(\lambda_l))}}{l_1(\lambda_j) - l_3(\lambda_l)} && \text{for $j, l \in \{1, \dots, n_b\}$,}
	\\
& (B_{\lambda \bar{\lambda}})_{jl} 
= \frac{\bar{\tilde{c}}_{\lambda_l} e^{x(l_1(\lambda_j) - l_2(\bar{\lambda}_l))+ t(z_1(\lambda_j) - z_2(\bar{\lambda}_l))}}{l_1(\lambda_j) - l_2(\bar{\lambda}_l)}  && \text{for $j, l \in \{1, \dots, n_b\}$},
	\\
& (B_{\lambda  k})_{jl} 
= \frac{\tilde{c}_{k_l} e^{x(l_1(\lambda_j) - l_2(k_l))+ t(z_1(\lambda_j) - z_2(k_l))}}{l_1(\lambda_j) - l_2(k_l)}  && \text{for $j \in \{1, \dots, n_b\}$ and $l \in \{1, \dots, n_s\}$,}
	\\
& (B_{\bar{\lambda} \lambda} )_{jl} 
= \frac{\tilde{c}_{\lambda_l} e^{x(l_3(\bar{\lambda}_j) - l_3(\lambda_l)) + t(z_3(\bar{\lambda}_j) - z_3(\lambda_l))}}{l_3(\bar{\lambda}_j) - l_3(\lambda_l)} && \text{for $j, l \in \{1, \dots, n_b\}$},
	\\
& (B_{\bar{\lambda} \bar{\lambda}} )_{jl} 
=  \frac{\bar{\tilde{c}}_{\lambda_l} e^{x(l_3(\bar{\lambda}_j) - l_2(\bar{\lambda}_l)) + t(z_3(\bar{\lambda}_j) - z_2(\bar{\lambda}_l))}}{l_3(\bar{\lambda}_j) - l_2(\bar{\lambda}_l)} && \text{for $j, l \in \{1, \dots, n_b\}$},
	\\
& (B_{\bar{\lambda}  k} )_{jl} 
= \frac{\tilde{c}_{k_l} e^{x(l_3(\bar{\lambda}_j) - l_2(k_l)) + t(z_3(\bar{\lambda}_j) - z_2(k_l))}}{l_3(\bar{\lambda}_j) - l_2(k_l)} && \text{for $j \in \{1, \dots, n_b\}$ and $l \in \{1, \dots, n_s\}$,}
	\\
& (B_{ k \lambda} )_{jl} 
= \frac{\tilde{c}_{\lambda_l} e^{x(l_1(k_j) - l_3(\lambda_l)) + t(z_1(k_j) - z_3(\lambda_l))}}{l_1(k_j) - l_3(\lambda_l)} && \text{for $j \in \{1, \dots, n_s\}$ and $l \in \{1, \dots, n_b\}$,}
	\\
& (B_{ k \bar{\lambda}} )_{jl} 
=  \frac{\bar{\tilde{c}}_{\lambda_l} e^{x(l_1(k_j) - l_2(\bar{\lambda}_l)) + t(z_1(k_j) - z_2(\bar{\lambda}_l))}}{l_1(k_j) - l_2(\bar{\lambda}_l)} && \text{for $j \in \{1, \dots, n_s\}$ and $l \in \{1, \dots, n_b\}$,}
	\\
& (B_{k k} )_{jl} 
=\frac{\tilde{c}_{k_l} e^{x(l_1(k_j) - l_2(k_l)) + t(z_1(k_j) - z_2(k_l))}}{l_1(k_j) - l_2(k_l)} && \text{for $j,l \in \{1, \dots, n_s\}$},
\end{align*}
and
\begin{align*}
& \tilde{c}_{\lambda_l} := \frac{i(\lambda_l^2-1)}{2 \sqrt{3} \lambda_l^2} c_{\lambda_l} \;\; \text{for $l \in \{1, \dots, n_b\}$},
	\qquad
 \tilde{c}_{k_l} := \frac{i(k_l^2-\omega^2)}{2 \sqrt{3} k_l^2} \omega^2 c_{k_l} \;\; \text{for $l \in \{1, \dots, n_s\}$}.
\end{align*}
The function $u_s\big(x,t; (\lambda_j, c_{\lambda_j})_{j=1}^{n_b}, (k_j, c_{k_j})_{j=1}^{n_s}\big)$ in (\ref{usdef}) is a multi-soliton solution of (\ref{boussinesq}) supporting $n_b$ breather solitons and $n_s$ bell-shaped solitons, see Remark \ref{multisolitonremark}.

The leading term $u_{\mathrm{sol}}(x,t)$ in Theorem \ref{asymptoticsth} is given by replacing the constants $c_{\lambda_j}$ and $c_{k_j}$ in (\ref{usdef}) by functions depending on $\zeta = x/t$.
More precisely, we have the following result whose proof is presented in Section \ref{usolexplicitproofsec}.

\begin{theorem}[Explicit expression for $u_{\mathrm{sol}}$]\label{usolth}
Under the assumptions of Theorem \ref{asymptoticsth}, the leading term in (\ref{uasymptotics}) is given by
\begin{align}\label{usolexplicit}
u_{\mathrm{sol}}(x,t) = u_s\bigg(x,t; \bigg(\lambda_j, \frac{c_{\lambda_j}}{\hat{\Delta}_{11}(\zeta, \lambda_j)\hat{\Delta}_{33}^{-1}(\zeta, \lambda_j)}\bigg)_{j=1}^{n_b}, \bigg(k_j, \frac{c_{k_j}}{\hat{\Delta}_{11}(\zeta, k_j)\hat{\Delta}_{22}^{-1}(\zeta, k_j)}\bigg)_{j=1}^{n_s}\bigg),
\end{align}
where $u_s$ is the function in (\ref{usdef}), $\zeta = x/t$, $\{\lambda_j\}_1^{n_b} := \mathsf{Z} \cap D_{\mathrm{reg}}$, $\{k_j\}_1^{n_s} := \mathsf{Z} \cap \R$, the constants $\{c_{\lambda_j}\}_1^{n_b}$ and $\{c_{k_j}\}_1^{n_s}$ are defined by (\ref{ck0def}), and the functions $\hat{\Delta}_{jj}(\zeta, k)$, $j = 1,2,3$, are defined for $\zeta \in \mathcal{I}$ and $k \in \mathbb{C}\setminus \partial \D$ by
\begin{align}\nonumber
& \hat{\Delta}_{11}(\zeta,k) := \hat{\Delta}_{33}(\zeta, \omega k), \quad \hat{\Delta}_{22}(\zeta,k) := \hat{\Delta}_{33}(\zeta,\omega^2 k), \quad
\hat{\Delta}_{33}(\zeta,k) := \frac{\delta(\zeta,\omega k)}{\delta(\zeta,\omega^{2} k)}\frac{\delta(\zeta,\frac{1}{\omega^{2} k})}{\delta(\zeta,\frac{1}{\omega k})},
	\\ \label{introdeltadef}
& \delta(\zeta, k) := \exp \left\{ \frac{-1}{2\pi i} \int_{i}^{k_{1}} \frac{\ln(1 + r_1(s)r_{2}(s))}{s - k} ds \right\}.
\end{align}
In (\ref{introdeltadef}), the principal branch is adopted for the logarithm and the path of integration follows the unit circle in the counterclockwise direction. 
\end{theorem}

We see from (\ref{usolexplicit}) and (\ref{introdeltadef}) that if $r_1 \equiv r_2 \equiv 0$, then $u_{\mathrm{sol}}(x,t)$ is a pure multi-soliton solution of (\ref{boussinesq}). In general, $u_{\mathrm{sol}}(x,t)$ is a modulated version of this multi-soliton, where the constants $c_{\lambda_j}$ and $c_{k_j}$ have been replaced by the slowly varying functions  $\frac{c_{\lambda_j}}{\hat{\Delta}_{11}(\zeta, \lambda_j)\hat{\Delta}_{33}^{-1}(\zeta, \lambda_j)}$ and $\frac{c_{k_j}}{\hat{\Delta}_{11}(\zeta, k_j)\hat{\Delta}_{22}^{-1}(\zeta, k_j)}$, respectively.\footnote{We emphasize that the $x$-derivatives in (\ref{usdef}) should be evaluated before this replacement is made, i.e., the $x$-derivatives in (\ref{usdef}) do not act on the functions $\hat{\Delta}_{11}$, $\hat{\Delta}_{22}$, and $\hat{\Delta}_{33}$.}

\begin{figure}
\bigskip\begin{center}
\hspace{-.4cm}
\begin{overpic}[width=.46\textwidth]{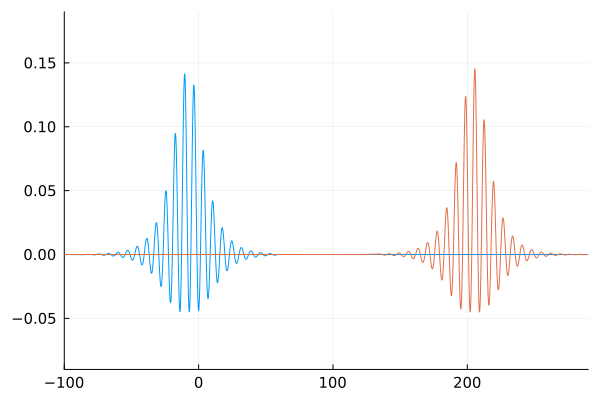}
      \put(8.5,67.5){\footnotesize $u_s$}
      \put(27,58){\footnotesize $t=0$}
      \put(73,58){\footnotesize $t=200$}
      \put(101,4){\footnotesize $x$}
    \end{overpic}
    \hspace{0.5cm}
\begin{overpic}[width=.46\textwidth]{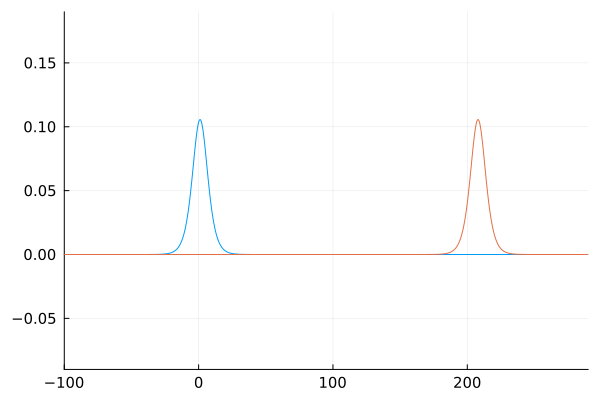}
       \put(8.5,67.5){\footnotesize $u_s$}
      \put(29,50){\footnotesize $t=0$}
      \put(74,50){\footnotesize $t=200$}
      \put(101,4){\footnotesize $x$}
    \end{overpic}
     \begin{figuretext}\label{solitonfig}
       Left: The single breather soliton $u_s(x,t; (\lambda_1, c_{\lambda_1}), \emptyset)$ for $\lambda_1 = 1.2+0.01i$ and $c_{\lambda_1} = 0.1$ at time $t = 0$ (blue) and $t = 200$ (red). 
       
       Right: The one-soliton $u_s(x,t; \emptyset, (k_1, c_{k_1}))$ for $k_1 = 1.3$ and $c_{k_1} = 0.1 e^{-i\arg(i(\omega^2 k_1^2 - \omega))}$ at time $t = 0$ (blue) and $t = 200$ (red). 
\end{figuretext}
     \end{center}
\end{figure}

\begin{remark}\label{multisolitonremark}
To understand the multi-soliton formula (\ref{usdef}) it is useful to keep the following in mind (see \cite[Appendix A]{CLscatteringsolitons}): If $n_b = 1$ and $n_s = 0$, then the function $u_s(x,t; (\lambda_1, c_{\lambda_1}), \emptyset)$ in (\ref{usdef}) reduces to a single breather soliton propagating in the direction $x/t = \zeta_{\lambda_1}$, where $\zeta_{\lambda_1}$ is such that $\re \Phi_{31}(\zeta_{\lambda_1},\lambda_{1}) = 0$ (recall here that $\lambda_{1}\in D_{\mathrm{reg}}$), i.e.
\begin{align}\label{breathervelocity}
\zeta_{\lambda_1} := \frac{(1+|\lambda_1|^2)(\lambda_1^2 + \omega \bar{\lambda}_1^2)}{2|\lambda_1|^2(\omega \lambda_1 + \bar{\lambda}_1)} \in (-\infty, -1) \cup (1, +\infty),
\end{align}
while if $n_b = 0$ and $n_s = 1$, then the function $u_s(x,t; \emptyset, (k_1, c_{k_1}))$ in (\ref{usdef}) reduces to the pure one-soliton (assuming that $c_{k_1} \neq 0$)
\begin{align} \label{usonesoliton}
& u_s(x,t; \emptyset, (k_1, c_{k_1})) = A \sech^2\bigg(\sqrt{\frac{A}{6}}(x - x_0 - \zeta_{k_1} t)\bigg),
	\\  \nonumber
& A := \frac{3}{8}(k_1 - k_1^{-1})^2 \geq 0, \qquad 
x_0 := \frac{2k_1}{k_1^2 -1}\ln\bigg( \frac{i (\omega^2 k_1^2 - \omega) c_{k_1}}{\sqrt{3}k_1(k_1^2 -1)}\bigg) \in \R,
\end{align}
propagating in the direction $x/t = \zeta_{k_1}$, where the principal branch is used for the logarithm and $\zeta_{k_1}$ is such that $\re \Phi_{21}(\zeta_{k_1}, k_{1}) = 0$ (recall here that $k_{1}\in (-1,0)\cup (1,+\infty)$), i.e.
\begin{align}\label{onesolitonvelocity}
\zeta_{k_1} := \frac{k_1 + k_1^{-1}}{2} \in (-\infty, -1) \cup (1, +\infty),
\end{align}
see Figure \ref{solitonfig}. For general values of $n_b \geq 0$ and $n_s \geq 0$, the multi-soliton $u_s$ can be viewed as a nonlinear superposition of $n_b$ breather solitons and $n_s$ one-solitons.
\end{remark}

\subsection{Formula valid near an asymptotic soliton}
Our last theorem provides a simplification of the asymptotic formula (\ref{uasymptotics}) in the direction of an asymptotic soliton. Under the assumptions of Theorem \ref{asymptoticsth}, let $\{\lambda_j\}_1^{n_b} := \mathsf{Z} \cap D_{\mathrm{reg}}$ and $\{k_j\}_1^{n_s} := \mathsf{Z} \cap \R$ denote the zeros of $s_{11}(k)$ in $D_{\mathrm{reg}}$ and $(-1,0) \cup (1,+\infty)$, respectively. For $j = 1, \dots, n_b$, we define 
$$\zeta_{\lambda_j} := \frac{(1+|\lambda_j|^2)(\lambda_j^2 + \omega \bar{\lambda}_j^2)}{2|\lambda_j|^2(\omega \lambda_j + \bar{\lambda}_j)} \in (-\infty, -1) \cup (1, +\infty),$$
and for $j = 1, \dots, n_s$, we define
$$\zeta_{k_j} := \frac{k_j + k_j^{-1}}{2} \in (-\infty, -1) \cup (1, +\infty).$$
Note that $\re \Phi_{31}(\zeta,\lambda_j) = 0$ if and only if $\zeta = \zeta_{\lambda_j}$ and $\re \Phi_{21}(\zeta, k_j) = 0$ if and only if $\zeta = \zeta_{k_j}$, so the set $S_0$ defined above equation (\ref{Sdeltadef}) is given by
\begin{align}\label{S0explicit}
S_0 = \{\zeta_{\lambda_1}, \dots, \zeta_{\lambda_{n_b}}, \zeta_{k_1}, \dots, \zeta_{k_{n_s}}\} \cap (1,+\infty).
\end{align}
   
In light of (\ref{breathervelocity}) and (\ref{onesolitonvelocity}), we expect each zero $\lambda_j$ to generate an asymptotic breather soliton propagating at velocity $x/t = \zeta_{\lambda_j}$, and each zero $k_j$ to generate an asymptotic one-soliton propagating at velocity $x/t = \zeta_{k_j}$.
If some of the velocities $\zeta_{\lambda_j}$ and/or $\zeta_{k_j}$ coincide and are equal to say $\zeta_0$, then we expect the asymptotics in the direction $x/t = \zeta_0$ to be described by a multi-soliton that combines all the breathers and one-solitons traveling at velocity $\zeta_0$.
The following theorem, which complements Theorem \ref{asymptoticsth2}, confirms that this is indeed the case.

\begin{theorem}[Asymptotics in the direction of an asymptotic soliton]\label{asymptoticsth3}
Under the assumptions of Theorem \ref{asymptoticsth}, if $\zeta_0 \in S_0$ and $\epsilon > 0$ is so small that $\zeta_0 - \epsilon > 1$ and $\{\zeta \in \R \,| \, |\zeta - \zeta_0| \leq \epsilon\}$ contains no other point in $S_0$, then, as $t\to \infty$,
\begin{align}\nonumber
 u(x,t) = &\; u_s\bigg(x,t; \bigg(\lambda_j, \frac{c_{\lambda_j}}{\hat{\Delta}_{11}(\zeta, \lambda_j)\hat{\Delta}_{33}^{-1}(\zeta, \lambda_j)}\bigg)_{j \in I_0}, \bigg(k_j, \frac{c_{k_j}}{\hat{\Delta}_{11}(\zeta, k_j)\hat{\Delta}_{22}^{-1}(\zeta, k_j)}\bigg)_{j \in J_0}\bigg)
	\\ \label{uasymptoticsnearsoliton}
& + \frac{u_{\mathrm{rad}}(x,t)}{\sqrt{t}} + O\bigg(\frac{\ln t}{t}\bigg)
\end{align}
uniformly for all $\zeta:= \frac{x}{t}$ such that $|\zeta - \zeta_0| \leq \epsilon$, where $I_0$ consists of all $j \in \{1, \dots, n_b\}$ such that $\zeta_{\lambda_j} = \zeta_0$, $J_0$ consists of all $j \in \{1, \dots, n_s\}$ such that $\zeta_{k_j} = \zeta_0$, and $u_{\mathrm{rad}}(x,t)$ is given by (\ref{uraddef}).
In particular, if the set $I_0$ is empty and $J_0 = \{j\}$ contains one element, then, for any $C_0 > 0$,
\begin{align}\label{usechasymptotics}
 u(x,t) = &\; A_j \sech^2\bigg(\sqrt{\frac{A_j}{6}}(x - x_{0j} - \zeta_{k_j} t)\bigg)
 + \frac{u_{\mathrm{rad}}(x,t)}{\sqrt{t}} + O\bigg(\frac{\ln t}{t}\bigg)
\end{align}
as $t\to \infty$ uniformly for $x \in [\zeta_{k_j} t - C_0, \zeta_{k_j} t + C_0]$, where
\begin{align}\nonumber
& A_j := \frac{3}{8}(k_j - k_j^{-1})^2 \geq 0
\end{align}
and
\begin{align}\nonumber 
x_{0j} := \frac{2k_j}{k_j^2 -1} \ln\bigg( \frac{i (\omega^2 k_j^2 - \omega)}{\sqrt{3}k_j(k_j^2 -1)} \frac{c_{k_j}}{\hat{\Delta}_{11}(\zeta_{k_j}, k_j)\hat{\Delta}_{22}^{-1}(\zeta_{k_j}, k_j)}\bigg) \in \R
\end{align}
with the principal branch used for the logarithm.
\end{theorem}

\begin{remark}
Since the function $k\to \zeta_{k} := \frac{k+k^{-1}}{2}$ is a bijection from $(1,+\infty)$ to $(1,+\infty)$, the set $J_{0}$ in \eqref{uasymptoticsnearsoliton} contains at most one element.
\end{remark}

The proof of Theorem \ref{asymptoticsth3} is presented in Section \ref{proofth3subsec}.

\subsection{The RH problem for $M_{\mathrm{sol}}$}\label{Msolsubsec}
The statement of Theorem \ref{asymptoticsth} involves the $3 \times 3$-matrix valued function $M_{\mathrm{sol}}(x,t,k)$. This function is defined as the unique solution of a RH problem which we now state. 
For each $k_{0}\in \mathsf{Z}\setminus \mathbb{R}$, we define $d_{k_0} \in \mathbb{C}$ by
\begin{align}\label{dk0def}
d_{k_{0}} = \frac{\bar{k}_0^2-1}{\omega^2 (\omega^2 - \bar{k}_0^2)} \bar{c}_{k_0}, \qquad k_{0}\in \mathsf{Z}\setminus \mathbb{R}.
\end{align}
For $\zeta \in \mathcal{I}$ and $k \in \mathbb{C}\setminus \partial \D$, we define $\delta(\zeta,k)$, $\hat{\Delta}_{11}(\zeta,k)$, $\hat{\Delta}_{22}(\zeta,k)$, and $\hat{\Delta}_{33}(\zeta,k)$ as in (\ref{introdeltadef}) and we set
\begin{align}\label{hatDeltadef}
\hat{\Delta}(\zeta,k) := \begin{pmatrix}
\hat{\Delta}_{11}(\zeta, k) & 0 & 0 \\
0 & \hat{\Delta}_{22}(\zeta, k) & 0 \\
0 & 0 & \hat{\Delta}_{33}(\zeta,k) \end{pmatrix}.
\end{align}
We also define the set $\hat{\mathsf{Z}}$ by
\begin{align}\label{def of hatZ}
\hat{\mathsf{Z}} = \mathsf{Z} \cup \mathsf{Z}^{*} \cup \omega \mathsf{Z} \cup \omega \mathsf{Z}^{*} \cup \omega^{2} \mathsf{Z} \cup \omega^{2} \mathsf{Z}^{*} \cup \mathsf{Z}^{-1} \cup \mathsf{Z}^{-*} \cup \omega \mathsf{Z}^{-1} \cup \omega \mathsf{Z}^{-*} \cup \omega^{2} \mathsf{Z}^{-1} \cup \omega^{2} \mathsf{Z}^{-*},
\end{align}
where $\mathsf{Z}^{-1}=\{k_{0}^{-1} \,|\, k_{0}\in \mathsf{Z}\}$, $\mathsf{Z}^{*}=\{\bar{k}_{0} \,|\, k_{0}\in \mathsf{Z}\}$, and $\mathsf{Z}^{-*}=\{\bar{k}_{0}^{-1} \,|\, k_{0}\in \mathsf{Z}\}$.

\begin{RHproblem}[RH problem for $M_{\mathrm{sol}}$]\label{RHS}
Find a $3 \times 3$-matrix valued function $M_{\mathrm{sol}}(x,t,k)$ with the following properties:
\begin{enumerate}[$(a)$]
\item\label{RHSitema} $M_{\mathrm{sol}}(x,t,\cdot) : \C \setminus \hat{\mathsf{Z}} \to \mathbb{C}^{3 \times 3}$ is analytic, where $\hat{\mathsf{Z}}$ is given by \eqref{def of hatZ}.

\item\label{RHSitemb} For $k \in \C \setminus \hat{\mathsf{Z}}$, $M_{\mathrm{sol}}$ obeys the symmetries
\begin{align}\label{Msolsymm}
M_{\mathrm{sol}}(x,t,k) = \mathcal{A} M_{\mathrm{sol}}(x,t,\omega k)\mathcal{A}^{-1} = \mathcal{B} M_{\mathrm{sol}}(x,t,k^{-1}) \mathcal{B},
\end{align}
where
\begin{align}\label{def of Acal and Bcal}
\mathcal{A} := \begin{pmatrix}
0 & 0 & 1 \\
1 & 0 & 0 \\
0 & 1 & 0
\end{pmatrix} \qquad \mbox{ and } \qquad \mathcal{B} := \begin{pmatrix}
0 & 1 & 0 \\
1 & 0 & 0 \\
0 & 0 & 1
\end{pmatrix}.
\end{align}

\item\label{RHSitemc} $M_{\mathrm{sol}}(x,t,k) = I + O(k^{-1})$ as $k \to \infty$.

\item\label{RHSitemd}
At each point $k_0 \in \mathsf{Z}\setminus \R$, the first two columns of $M_{\mathrm{sol}}$ are analytic, while the third column has (at most) a simple pole. Furthermore, 
\begin{align}\label{Mk0notinR}
 \underset{k = k_0}{\res} [M_{\mathrm{sol}}(x,t,k)]_3 = \frac{c_{k_{0}} e^{-\theta_{31}(x,t,k_0)}}{\hat{\Delta}_{11}(\zeta, k_0)\hat{\Delta}_{33}^{-1}(\zeta, k_0)}[M_{\mathrm{sol}}(x,t,k_0)]_1.
\end{align}
At each point $\bar{k}_{0}$ with $k_{0}\in \mathsf{Z}\setminus \R$, the first and third columns of $M_{\mathrm{sol}}$ are analytic, while the second column has (at most) a simple pole. Furthermore, 
\begin{align}\label{Mk0notinR 2}
 \underset{k = \bar{k}_0}{\res} [M_{\mathrm{sol}}(x,t,k)]_2 = \frac{d_{k_{0}} e^{\theta_{32}(x,t,\bar{k}_0)}}{\hat{\Delta}_{33}(\zeta, \bar{k}_0)\hat{\Delta}_{22}^{-1}(\zeta, \bar{k}_0)} [M_{\mathrm{sol}}(x,t,\bar{k}_0)]_3,
\end{align}
where $d_{k_0}$ is given by (\ref{dk0def}).
At each point $k_0 \in \mathsf{Z}\cap \R$, the first and third columns of $M_{\mathrm{sol}}$ are analytic, while the second column has (at most) a simple pole. Furthermore,
\begin{align}\label{Mk0inR}
 \underset{k = k_{0}}{\res} [M_{\mathrm{sol}}(x,t,k)]_2 = \frac{c_{k_{0}} e^{-\theta_{21}(x,t,k_0)}}{\hat{\Delta}_{11}(\zeta, k_0)\hat{\Delta}_{22}^{-1}(\zeta, k_0)} [M_{\mathrm{sol}}(x,t,k_{0})]_1.
\end{align}
\end{enumerate}
\end{RHproblem}

Suppose $\{r_1, r_2, \mathsf{Z}, \{c_{k_0}\}_{k_0 \in \mathsf{Z}}\}$ are scattering data associated to some initial data $u_0, u_1$ fulfilling the assumptions of Theorem \ref{asymptoticsth}. Then the solution of RH problem \ref{RHS} exists, is unique, and satisfies $\det M_{\mathrm{sol}}(x,t,k)=1$; moreover, the soliton term $u_{\mathrm{sol}}(x,t)$ in the asymptotic formula (\ref{uasymptotics}) defined in terms of $M_{\mathrm{sol}}(x,t,k)$ via (\ref{usoldef}) is a smooth function of $x$ and $t$, see Lemma \ref{RHSlemma}.

\subsection{Discussion of the asymptotic formulas}
Let us first discuss the soliton term $u_{\mathrm{sol}}(x,t)$ in (\ref{uasymptotics}).
The solution $M_{\mathrm{sol}}$ depends on the reflection coefficients $r_1(k)$ and $r_2(k)$ only via the function $\hat{\Delta}$ that appears in the residue conditions (\ref{Mk0notinR})--(\ref{Mk0inR}). Pure soliton solutions of the Boussinesq equation (\ref{boussinesq}) arise when the reflection coefficients $r_1(k)$ and $r_2(k)$ vanish identically. If we temporarily set $r_1$ and $r_2$ to zero in RH problem \ref{RHS}, then $\hat{\Delta}$ is identically equal to the identity matrix $I$ and the function $u_\sol$ defined in (\ref{usoldef}) is a non-decaying solution of (\ref{boussinesq}); more precisely, it is the (multi-)soliton solution of (\ref{boussinesq}) generated by the scattering data $\{r_1 = 0, r_2 = 0, \mathsf{Z}, \{c_{k_0}\}_{k_0 \in \mathsf{Z}}\}$. 
In the general case, the functions $r_1(k)$ and $r_2(k)$ do not vanish identically. In this case, the function $u_\sol$ defined in (\ref{usoldef}) is a perturbation of the pure multi-soliton generated by $\{0,0, \mathsf{Z}, \{c_{k_0}\}_{k_0 \in \mathsf{Z}}\}$. The perturbation, which is introduced via the function $\hat{\Delta}$ (see (\ref{uasymptoticsnearsoliton})), encodes the effect of the radiative background on the solitons.

\begin{figure}
\begin{center}
\begin{tikzpicture}[master]
\node at (0,0) {\includegraphics[width=6.5cm]{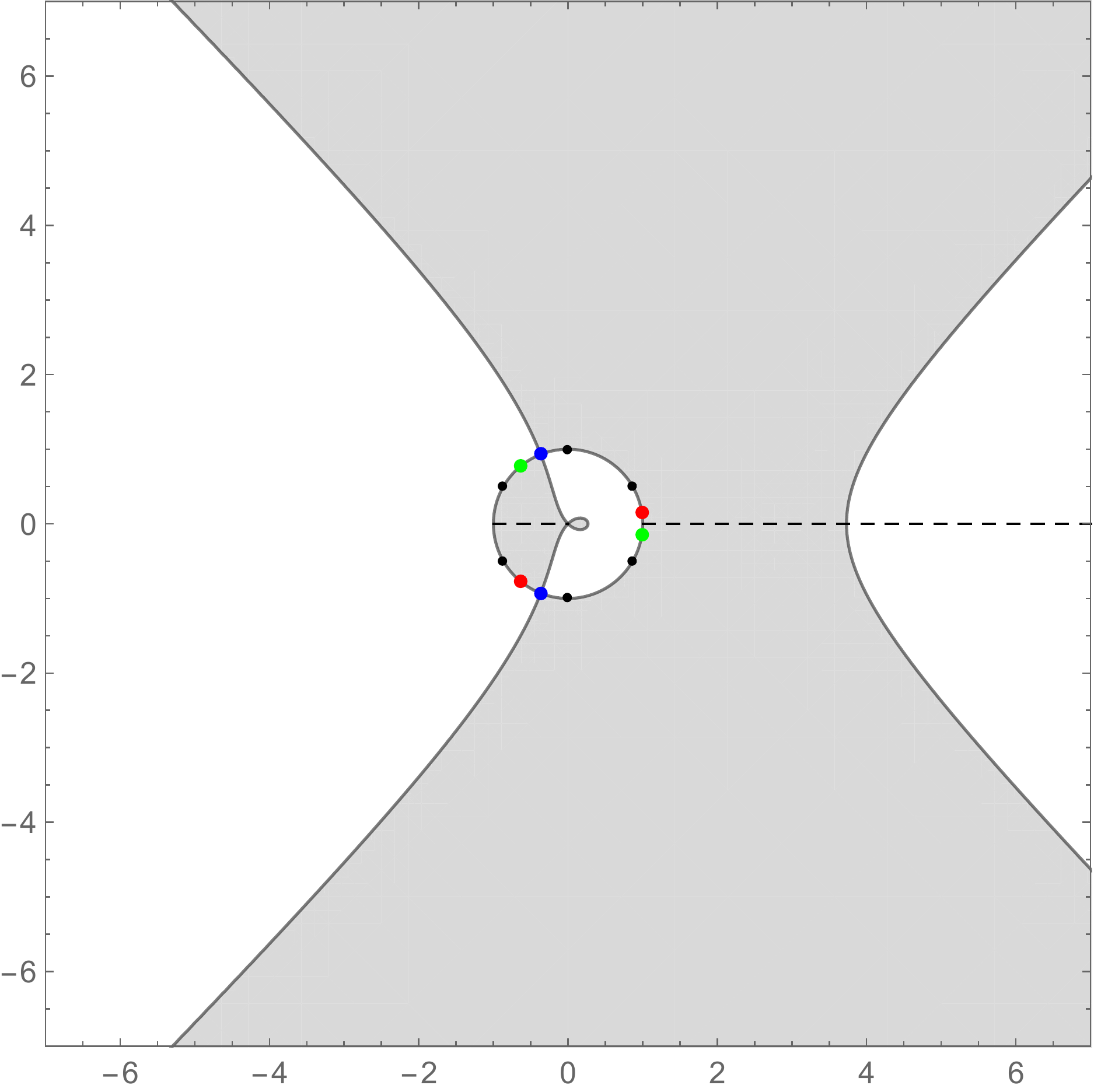}};

\node at (-1.7,0.15) {\small $\re \Phi_{21}\hspace{-0.05cm} < 0$};
\node at (0.8,2) {\small $\re \Phi_{21} \hspace{-0.05cm} > 0$};

\end{tikzpicture} \hspace{0.1cm} 
\begin{tikzpicture}[slave]
\node at (0,0) {\includegraphics[width=6.5cm]{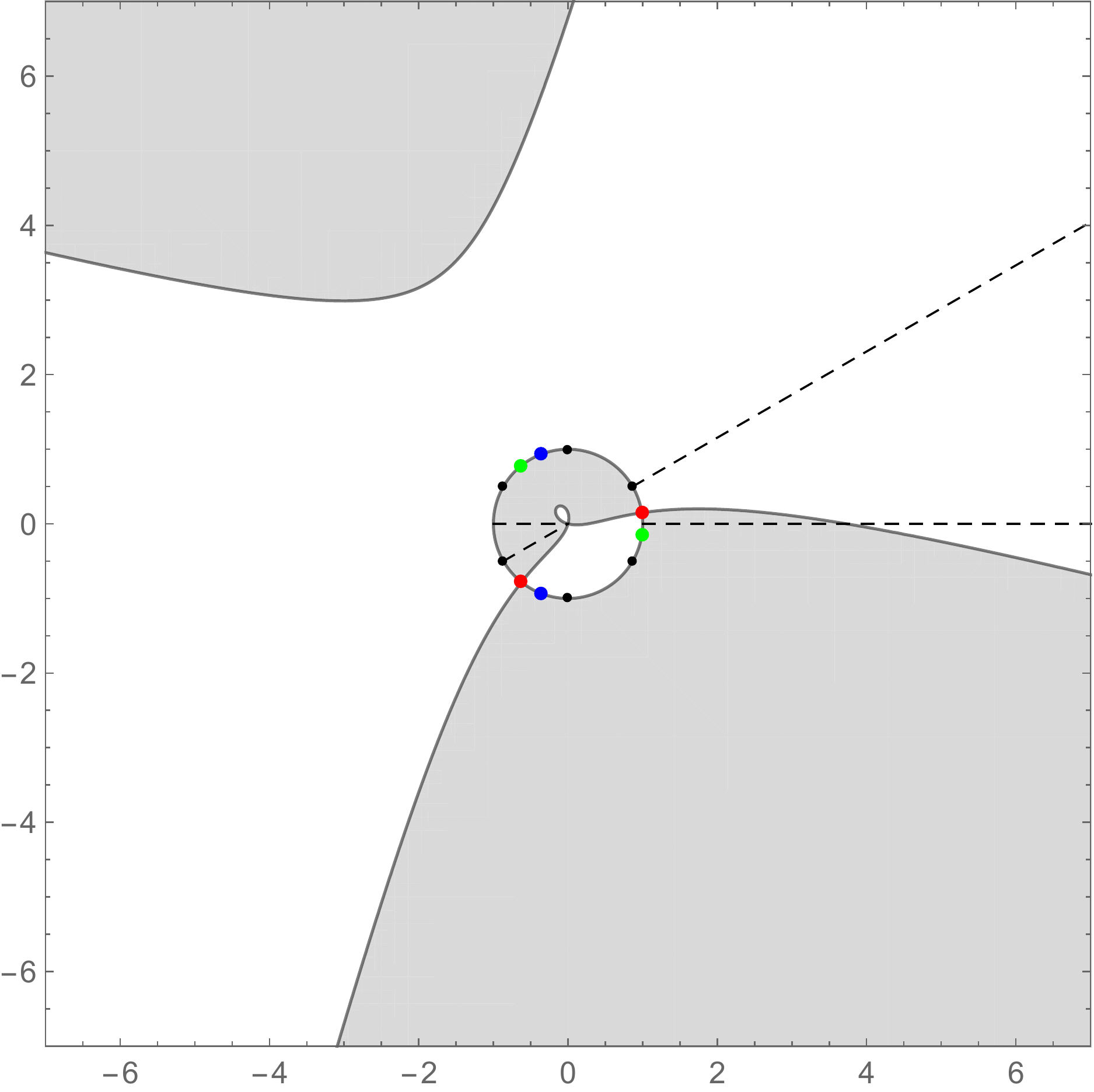}};


\node at (1.25,-1.1) {\small $\re \Phi_{31}\hspace{-0.05cm}> 0$};
\node at (-1.8,-0.35) {\small $\re \Phi_{31} \hspace{-0.05cm}< 0$};

\end{tikzpicture} 
%
%
\end{center}
\begin{figuretext}
\label{fig:soliton and signature table} From left to right: The signature tables for $\Phi_{21}$ and $\Phi_{31}$ for $\zeta=2$. The grey regions correspond to $\{k \,|\, \re \Phi_{ij}>0\}$ and the white regions to $\{k \,|\, \re \Phi_{ij}<0\}$. In the left picture, the dashed line segments represent $(-1,0)\cup (1,\infty)$, while in the right picture the dashed line segments are parts of $\partial D_{\mathrm{reg}}$. The saddle points $k_{1},k_{2}$ of $\Phi_{21}$ are blue, the saddle points $\omega k_{1},\omega k_{2}$ of $\Phi_{31}$ are red, and the saddle points $\omega^{2} k_{1},\omega^{2} k_{2}$ of $\Phi_{32}$ are green. The black dots are the points $i \kappa_j$, $j=1,\ldots,6$.
\end{figuretext}
\end{figure}

We next discuss the leading term in (\ref{uasymptotics2}). This term is a modulated sine-wave of order $O(t^{-1/2})$, where the effect of the solitons is encoded in the phase shift $\arg \frac{\mathcal{P}(\zeta,\omega k_{1})}{\mathcal{P}(\zeta,\omega^{2} k_{1})}$.
The first product in the definition (\ref{def of mathcalP}) of $\mathcal{P}(\zeta, k)$ runs over all points $k_0 \in \mathsf{Z} \setminus \R$ such that $\re \Phi_{31}(\zeta,k_{0})<0$; these points correspond to breather solitons, and the condition $\re \Phi_{31}(\zeta,k_{0})<0$ ensures that only those breather solitons traveling at velocities greater than $\zeta$ are taken into account, see Figure \ref{fig:soliton and signature table}. Similarly, the second product in (\ref{def of mathcalP}) runs over all points $k_0 \in \mathsf{Z} \cap \R$ such that $\re \Phi_{21}(\zeta,k_{0})<0$; these points correspond to one-solitons, and the condition $\re \Phi_{21}(\zeta,k_{0})<0$ ensures that only those one-solitons traveling at velocities larger than $\zeta$ are taken into account. In other words, the asymptotic behavior of $u(x,t)$ in the direction $x/t = \zeta > 1$ is given by a decaying modulated sine-wave and the phase of this sine-wave receives a shift from each right-moving soliton traveling at a speed greater than $\zeta$. 

The function $\mathcal{P}(\zeta,k)$ is in general discontinuous as a function of $\zeta$, because as $\zeta$ decreases below the velocity of an asymptotic soliton, a new factor suddenly appears in one of the products in \eqref{def of mathcalP}. This is consistent with the fact that the asymptotic formula (\ref{uasymptotics2}) only is valid away from the directions of the asymptotic solitons, i.e., away from the points of discontinuity of $\mathcal{P}(\zeta,k)$. The additional complexity of the expression for the term of $O(t^{-1/2})$ in Theorems \ref{asymptoticsth} and \ref{asymptoticsth3} compared to the analogous term in Theorem \ref{asymptoticsth2} can be viewed as a price we have to pay to smoothen out these discontinuities. Indeed, the asymptotic formulas (\ref{uasymptotics}) and (\ref{uasymptoticsnearsoliton}) are uniformly valid also near the directions of the asymptotic solitons and the radiation term $u_{\mathrm{rad}}(x,t)/\sqrt{t}$ in (\ref{uasymptotics}) and (\ref{uasymptoticsnearsoliton}) depends smoothly on $\zeta$.

\subsection{Notation}\label{notationsubsec}
We use $C>0$ and $c>0$ to denote generic constants that may change within a computation.
If $A$ is an $n \times m$ matrix, we define $|A| \geq 0$ by $|A|^2=\Sigma_{i,j}|A_{ij}|^2$. For a piecewise smooth contour $\gamma \subset \C$ and $1 \le p \le \infty$, if $|A|$ belongs to $L^p(\gamma)$, we write $A \in L^p(\gamma)$ and define $\|A\|_{L^p(\gamma)} := \| |A|\|_{L^p(\gamma)}$. We will use $D_\epsilon(k)$ to denote the open disk of radius $\epsilon$ centered at $k \in \C$. The Schwartz space $\mathcal{S}(\R)$ contains all smooth functions $f$ on $\R$ such that $f$ and all its derivatives have rapid decay as $x \to \pm \infty$.

\section{The RH problem for $n$}\label{overviewsec}

Our proofs of Theorems \ref{asymptoticsth} and \ref{asymptoticsth2} are based on a Deift-Zhou steepest descent analysis of a $1\times 3$ RH problem. This RH problem, whose solution is denoted by $n$, was first derived in \cite{CLmain}, and extended to the case when solitons are present in \cite{CLscatteringsolitons}. Its jump contour $\Gamma = \cup_{j=1}^{9}\Gamma_{j}$ is oriented as in Figure \ref{fig: Dn}. For each $j \in \{ 1, \dots, 6\}$, we write $\Gamma_j = \Gamma_{j'} \cup \Gamma_{j''}$, where $\Gamma_{j'} = \Gamma_j \setminus \D$ and $\Gamma_{j''} := \Gamma_j \setminus \Gamma_{j'}$ with $\Gamma_j$ as in Figure \ref{fig: Dn}. We let $\theta_{ij}(x,t,k) = t \, \Phi_{ij}(\zeta,k)$. By Assumption $(iii)$, we have $r_{1}(k)=0$ for $k \in [0,i]$, and thus the jump matrix $v(x,t,k)$ is given for $k \in \Gamma$ by (see \cite{CLmain})
\begin{align}
& v_{1'} = \begin{pmatrix}
1 & -r_{1}(k)e^{-\theta_{21}} & 0 \\
0 & 1 & 0 \\
0 & 0 & 1
\end{pmatrix}, \; v_{1''} = \begin{pmatrix}
1 & 0 & 0 \\
r_{1}(\frac{1}{k})e^{\theta_{21}} & 1 & 0 \\
0 & 0 & 1
\end{pmatrix}, \; v_{2'} = \begin{pmatrix}
1 & 0 & 0 \\
0 & 1 & -r_{2}(\frac{1}{\omega k})e^{-\theta_{32}} \\
0 & 0 & 1
\end{pmatrix}, \nonumber \\
& v_{2''} = \begin{pmatrix}
1 & 0 & 0 \\
0 & 1 & 0 \\
0 & r_{2}(\omega k)e^{\theta_{32}} & 1
\end{pmatrix}, \;  v_{3'} = \begin{pmatrix}
1 & 0 & 0 \\
0 & 1 & 0 \\
-r_{1}(\omega^{2}k)e^{\theta_{31}} & 0 & 1
\end{pmatrix}, \; v_{3''} = \begin{pmatrix}
1 & 0 & r_{1}(\frac{1}{\omega^{2}k})e^{-\theta_{31}} \\
0 & 1 & 0 \\
0 & 0 & 1
\end{pmatrix}, \nonumber \\
& v_{4'} = \begin{pmatrix}
1 & -r_{2}(\frac{1}{k})e^{-\theta_{21}} & 0 \\
0 & 1 & 0 \\
0 & 0 & 1
\end{pmatrix}, \; v_{4''} = \begin{pmatrix}
1 & 0 & 0 \\
r_{2}(k)e^{\theta_{21}} & 1 & 0 \\
0 & 0 & 1
\end{pmatrix}, \; v_{5'} = \begin{pmatrix}
1 & 0 & 0 \\
0 & 1 & -r_{1}(\omega k)e^{-\theta_{32}} \\
0 & 0 & 1
\end{pmatrix}, \nonumber \\
& v_{5''} = \begin{pmatrix}
1 & 0 & 0 \\
0 & 1 & 0 \\
0 & r_{1}(\frac{1}{\omega k})e^{\theta_{32}} & 1
\end{pmatrix}, \;  v_{6'} = \begin{pmatrix}
1 & 0 & 0 \\
0 & 1 & 0 \\
-r_{2}(\frac{1}{\omega^{2} k})e^{\theta_{31}} & 0 & 1
\end{pmatrix}, \; v_{6''} = \begin{pmatrix}
1 & 0 & r_{2}(\omega^{2}k)e^{-\theta_{31}} \\
0 & 1 & 0 \\
0 & 0 & 1
\end{pmatrix}, \nonumber \\
& v_{7} = \begin{pmatrix}
1 & -r_{1}(k)e^{-\theta_{21}} & r_{2}(\omega^{2}k)e^{-\theta_{31}} \\
-r_{2}(k)e^{\theta_{21}} & 1+r_{1}(k)r_{2}(k) & \big(r_{2}(\frac{1}{\omega k})-r_{2}(k)r_{2}(\omega^{2}k)\big)e^{-\theta_{32}} \\
r_{1}(\omega^{2}k)e^{\theta_{31}} & \big(r_{1}(\frac{1}{\omega k})-r_{1}(k)r_{1}(\omega^{2}k)\big)e^{\theta_{32}} & f(\omega^{2}k)
\end{pmatrix}, \nonumber \\
& v_{8} = \begin{pmatrix}
f(k) & r_{1}(k)e^{-\theta_{21}} & \big(r_{1}(\frac{1}{\omega^{2} k})-r_{1}(k)r_{1}(\omega k)\big)e^{-\theta_{31}} \\
r_{2}(k)e^{\theta_{21}} & 1 & -r_{1}(\omega k) e^{-\theta_{32}} \\
\big( r_{2}(\frac{1}{\omega^{2}k})-r_{2}(\omega k)r_{2}(k) \big)e^{\theta_{31}} & -r_{2}(\omega k) e^{\theta_{32}} & 1+r_{1}(\omega k)r_{2}(\omega k)
\end{pmatrix}, \nonumber \\
& v_{9} = \begin{pmatrix}
1+r_{1}(\omega^{2}k)r_{2}(\omega^{2}k) & \big( r_{2}(\frac{1}{k})-r_{2}(\omega k)r_{2}(\omega^{2} k) \big)e^{-\theta_{21}} & -r_{2}(\omega^{2}k)e^{-\theta_{31}} \\
\big(r_{1}(\frac{1}{k})-r_{1}(\omega k) r_{1}(\omega^{2} k)\big)e^{\theta_{21}} & f(\omega k) & r_{1}(\omega k)e^{-\theta_{32}} \\
-r_{1}(\omega^{2}k)e^{\theta_{31}} & r_{2}(\omega k) e^{\theta_{32}} & 1
\end{pmatrix}, \label{vdef}
\end{align}
where $v_j, v_{j'}, v_{j''}$ are the restrictions of $v$ to $\Gamma_{j}$, $\Gamma_{j'}$, and $\Gamma_{j''}$, respectively, and 
\begin{align}\label{def of f}
f(k) := 1+r_{1}(k)r_{2}(k) + r_{1}(\tfrac{1}{\omega^{2}k})r_{2}(\tfrac{1}{\omega^{2}k}), \qquad k \in \partial \D.
\end{align}
The function $v$ obeys the symmetries
\begin{align}\label{vsymm}
v(x,t,k) = \mathcal{A} v(x,t,\omega k)\mathcal{A}^{-1}
 = \mathcal{B} v(x,t, k^{-1})^{-1}\mathcal{B}, \qquad k \in \Gamma,
\end{align}
where $\mathcal{A}$ and $\mathcal{B}$ are given by (\ref{def of Acal and Bcal}).
Note that $v$ depends on $x$ and $t$ only via the phase functions $\Phi_{21}$, $\Phi_{31}$, and $\Phi_{32}$. 

Let $\mathsf{Z}$ be as in Assumption $(\ref{solitonassumption})$ and define $\hat{\mathsf{Z}}$ by (\ref{def of hatZ}).
We denote by $\Gamma_{\star} = \{i\kappa_j\}_{j=1}^6 \cup \{0\}$ the set of intersection points of $\Gamma$. 

\begin{RHproblem}[RH problem for $n$]\label{RHn}
Find a $1 \times 3$-row-vector valued function $n(x,t,k)$ with the following properties:
\begin{enumerate}[$(a)$]
\item\label{RHnitema} $n(x,t,\cdot) : \C \setminus (\Gamma \cup \hat{\mathsf{Z}}) \to \mathbb{C}^{1 \times 3}$ is analytic.

\item\label{RHnitemb} The limits of $n(x,t,k)$ as $k$ approaches $\Gamma \setminus \Gamma_\star$ from the left and right exist, are continuous on $\Gamma \setminus \Gamma_\star$, and are denoted by $n_+$ and $n_-$, respectively. Moreover, 
\begin{align}\label{njump}
  n_+(x,t,k) = n_-(x, t, k) v(x, t, k) \qquad \text{for} \quad k \in \Gamma \setminus \Gamma_\star.
\end{align}

\item\label{RHnitemc} $n(x,t,k) = O(1)$ as $k \to k_{\star} \in \Gamma_\star$.

\item\label{RHnitemd} For $k \in \C \setminus (\Gamma \cup \hat{\mathsf{Z}})$, $n$ obeys the symmetries
\begin{align}\label{nsymm}
n(x,t,k) = n(x,t,\omega k)\mathcal{A}^{-1} = n(x,t,k^{-1}) \mathcal{B}.
\end{align}

\item\label{RHniteme} $n(x,t,k) = (1,1,1) + O(k^{-1})$ as $k \to \infty$.

\item\label{RHnitemf} 
At each point of $\hat{\mathsf{Z}}$, one entry of $n$ has (at most) a simple pole while two entries are analytic. Moreover, the following residue conditions hold: for each $k_{0}\in \mathsf{Z}\setminus \mathbb{R}$, 
\begin{align}\nonumber
& \underset{k = k_0}{\res} n_3(k) = c_{k_{0}}e^{-\theta_{31}(k_0)} n_1(k_0), & & \underset{k = \omega k_0}{\res} n_2(k) = \omega c_{k_{0}}e^{-\theta_{31}(k_0)} n_3(\omega k_0), 
	\\ \nonumber
& \underset{k = \omega^2 k_0}{\res} n_1(k) = \omega^2 c_{k_{0}}e^{-\theta_{31}(k_0)} n_2(\omega^2 k_0), & & \underset{k = k_0^{-1}}{\res} n_3(k) = - k_0^{-2} c_{k_{0}}e^{-\theta_{31}(k_0)} n_2(k_0^{-1}),	
	\\\nonumber
&  \underset{k = \omega^2k_0^{-1}}{\res} n_1(k) = -\tfrac{\omega^2}{k_0^{2}} c_{k_{0}}e^{-\theta_{31}(k_0)} n_3(\omega^2 k_0^{-1}), & & \underset{k = \omega k_0^{-1}}{\res} n_2(k) 
= -\tfrac{\omega}{k_0^{2}} c_{k_{0}}e^{-\theta_{31}(k_0)} n_1(\omega k_0^{-1}), 
	 \\ \nonumber
& \underset{k = \bar{k}_0}{\res} n_2(k) = d_{k_0} e^{\theta_{32}(\bar{k}_0)} n_3(\bar{k}_0), & & \underset{k = \omega \bar{k}_0}{\res} n_1(k) = \omega d_{k_0} e^{\theta_{32}(\bar{k}_0)} n_2( \omega \bar{k}_0),
	\\ \nonumber
& \underset{k = \omega^2 \bar{k}_0}{\res} n_3(k) = \omega^2 d_{k_0} e^{\theta_{32}(\bar{k}_0)} n_1(\omega^2 \bar{k}_0), & & \underset{k = \bar{k}_0^{-1}}{\res} n_1(k) = - \bar{k}_0^{-2} d_{k_0} e^{\theta_{32}(\bar{k}_0)} n_3(\bar{k}_0^{-1}),	
	\\ \label{nresiduesk0}
&  \underset{k = \omega^2 \bar{k}_0^{-1}}{\res} n_2(k) = -\tfrac{\omega^2}{\bar{k}_0^{2}} d_{k_0} e^{\theta_{32}(\bar{k}_0)} n_1( \omega^2 \bar{k}_0^{-1}), & & \underset{k = \omega \bar{k}_0^{-1}}{\res} n_3(k) = -\tfrac{\omega}{\bar{k}_0^{2}} d_{k_0} e^{\theta_{32}(\bar{k}_0)} n_2(\omega \bar{k}_0^{-1}), 
\end{align}
and, for each $k_{0}\in \mathsf{Z}\cap \mathbb{R}$, 
\begin{align}\nonumber
& \underset{k = k_0}{\res} n_2(k) = c_{k_0} e^{-\theta_{21}(k_0)} n_1(k_0), & & \underset{k = \omega k_0}{\res} n_1(k) = \omega c_{k_0} e^{-\theta_{21}(k_0)} n_3(\omega k_0), 
	\\  \nonumber
& \underset{k = \omega^2 k_0}{\res} n_3(k) = \omega^2 c_{k_0} e^{-\theta_{21}(k_0)} n_2(\omega^2 k_0), & & \underset{k = k_0^{-1}}{\res} n_1(k) = - k_0^{-2} c_{k_0} e^{-\theta_{21}(k_0)} n_2(k_0^{-1}),		
	\\ \label{nresiduesk0real}
&  \underset{k = \omega^2k_0^{-1}}{\res} n_2(k) = -\tfrac{\omega^2}{k_0^{2}} c_{k_0} e^{-\theta_{21}(k_0)} n_3(\omega^2 k_0^{-1}), & & \underset{k = \omega k_0^{-1}}{\res} n_3(k) 
= -\tfrac{\omega}{k_0^{2}} c_{k_0} e^{-\theta_{21}(k_0)} n_1(\omega k_0^{-1}),
\end{align}
where the $(x,t)$-dependence of $n$ and of $\theta_{ij}$ has been suppressed for conciseness.
\end{enumerate}
\end{RHproblem}
Since $u_{0}$ and $u_1$ satisfy Assumptions $(i)$--$(iii)$, it follows from \cite[Theorems 2.6 and 2.11]{CLscatteringsolitons} that RH problem \ref{RHn} has a unique solution $n(x,t,k)$ for each $(x,t) \in \R \times [0,\infty)$, that
$$n_{3}^{(1)}(x,t) := \lim_{k\to \infty} k (n_{3}(x,t,k) -1)$$ 
is well-defined and smooth for $(x,t) \in \R \times [0,\infty)$, and that 
\begin{align}\label{recoveruvn}
u(x,t) = -i\sqrt{3}\frac{\partial}{\partial x}n_{3}^{(1)}(x,t)
\end{align}
is a real-valued Schwartz class solution of (\ref{boussinesq}) on $\R \times [0,\infty)$ with initial data $u(x,0)=u_{0}(x)$ and $u_{t}(x,0)=u_{1}(x)$.

Let $\mathcal{I}$ be a fixed compact subset of $(1,\infty)$. 
Recall from Section \ref{mainsec} that $\{k_{j}\}_{j=1}^{4}$ are the saddle points of $\Phi_{21}$. From the identities
$\Phi_{31}(\zeta,k) = - \Phi_{21}(\zeta,\omega^{2}k)$ and $\Phi_{32}(\zeta, k) = \Phi_{21}(\zeta, \omega k)$,
we infer that $\{\omega k_{j}\}_{j=1}^{4}$ are the saddle points of $\Phi_{31}$ and that $\{\omega^{2} k_{j}\}_{j=1}^{4}$ are the saddle points of $\Phi_{32}$. 
For the sector considered in this paper, namely $\zeta \in \mathcal{I}$, it turns out that only the saddle points $\{\omega^{j}k_{1},\omega^{j}k_{2}\}_{j=0}^{2}$ play a role in the asymptotic analysis of $n$. The signature tables for $\Phi_{21}$, $\Phi_{31}$, and $\Phi_{32}$ are shown in Figure \ref{IIbis fig: Re Phi 21 31 and 32 for zeta=0.7} (see also Figure \ref{fig:soliton and signature table} for the signature tables of $\Phi_{21}$ and $\Phi_{31}$ in a larger subset of the complex plane).
 
\begin{figure}[h]
\begin{center}
\begin{tikzpicture}[master]
\node at (0,0) {\includegraphics[width=4.5cm]{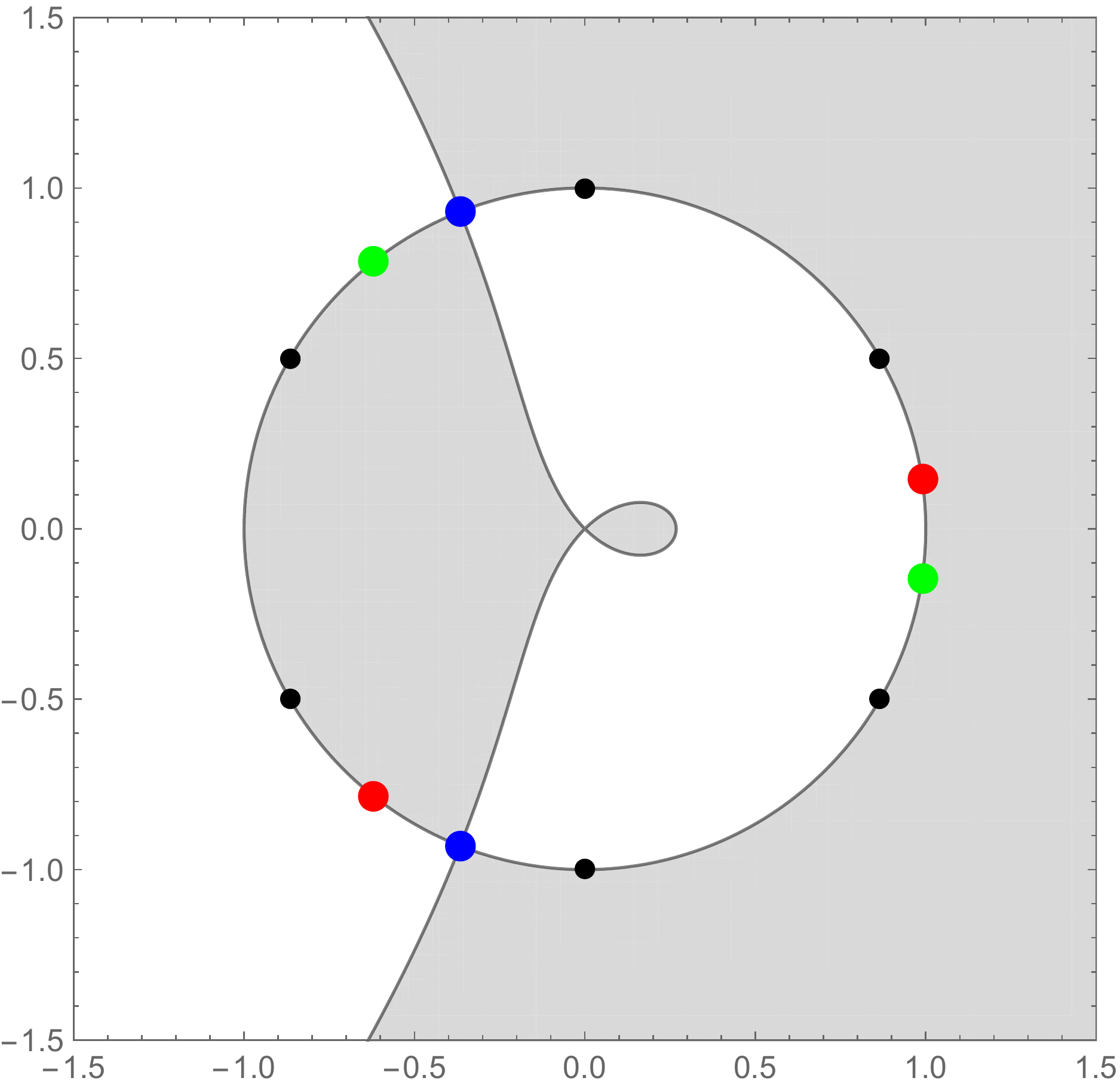}};

\node at (-0.31,1.54) {\tiny $k_1$};
\node at (-0.31,-1.43) {\tiny $k_2$};

\end{tikzpicture} \hspace{0.1cm} 
\begin{tikzpicture}[slave]
\node at (0,0) {\includegraphics[width=4.5cm]{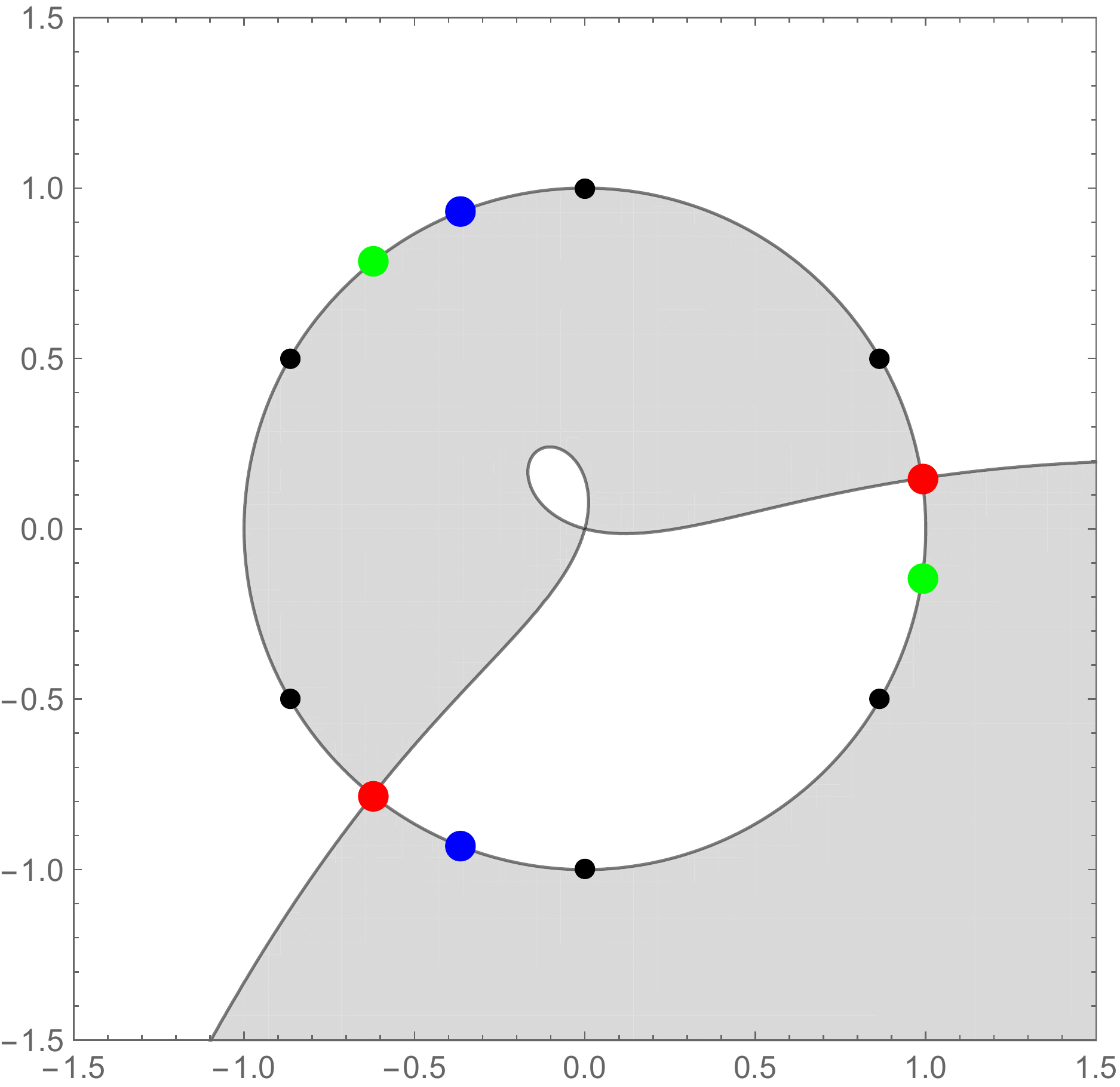}};

\node at (-1.06,-1.01) {\tiny $\omega k_1$};
\node at (1.75,0.45) {\tiny $\omega k_2$};

\end{tikzpicture} \hspace{0.1cm} 
\begin{tikzpicture}[slave]
\node at (0,0) {\includegraphics[width=4.5cm]{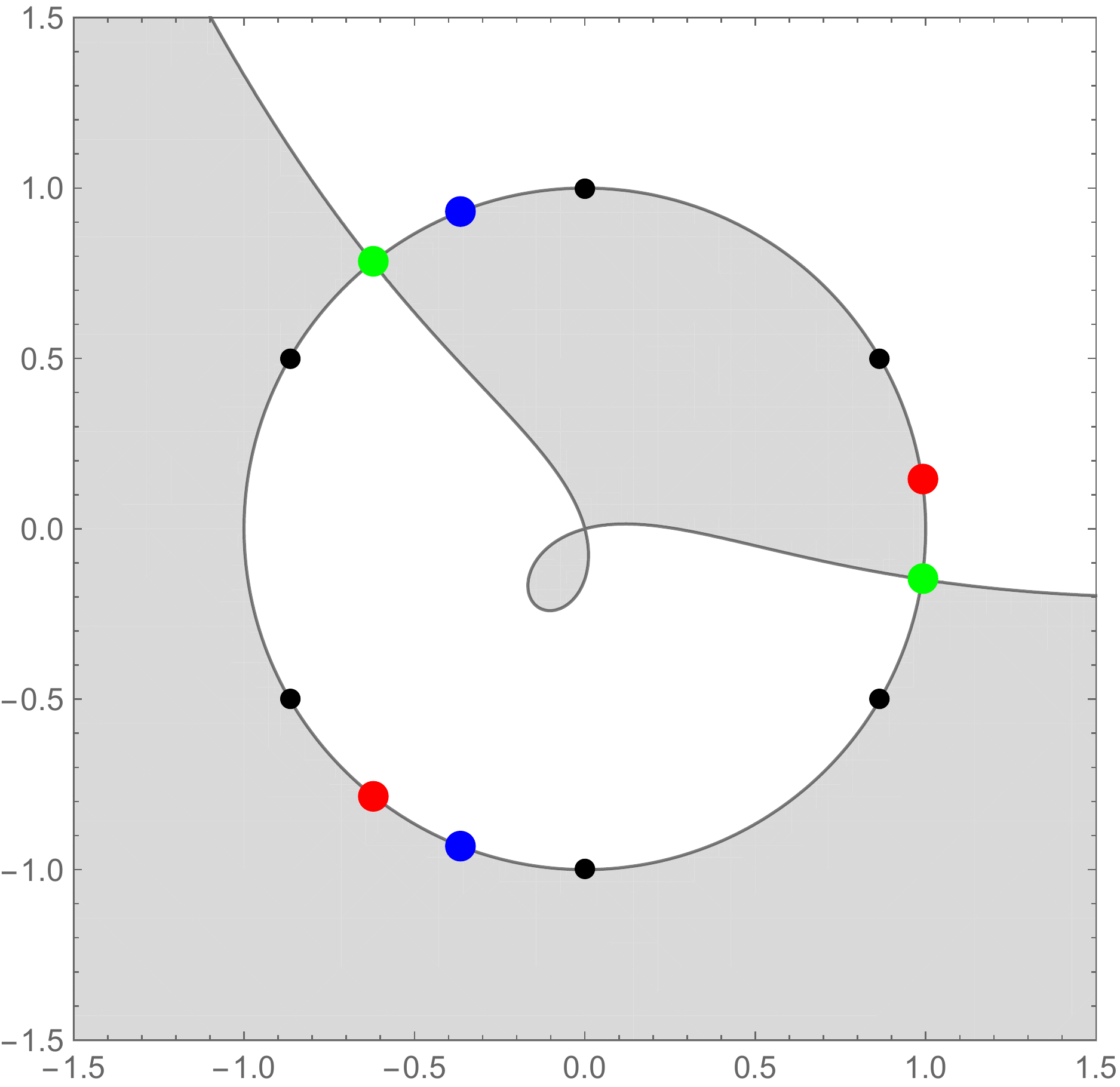}};

\node at (1.77,-0.34) {\tiny $\omega^2 k_1$};
\node at (-1.15,1.2) {\tiny $\omega^2 k_2$};

\end{tikzpicture}
\end{center}
\begin{figuretext}
\label{IIbis fig: Re Phi 21 31 and 32 for zeta=0.7} From left to right: The signature tables for $\Phi_{21}$, $\Phi_{31}$, and $\Phi_{32}$ for $\zeta=2$. The grey regions correspond to $\{k \,|\, \re \Phi_{ij}>0\}$ and the white regions to $\{k \,|\, \re \Phi_{ij}<0\}$. The saddle points $k_{1},k_{2}$ of $\Phi_{21}$ are blue, the saddle points $\omega k_{1},\omega k_{2}$ of $\Phi_{31}$ are red, and the saddle points $\omega^{2} k_{1},\omega^{2} k_{2}$ of $\Phi_{32}$ are green. The black dots are the points $i \kappa_j$, $j=1,\ldots,6$.
\end{figuretext}
\end{figure}

We will use the Deift--Zhou \cite{DZ1993} steepest descent method applied to RH problem \ref{RHn} to obtain asymptotics for $u$. We will proceed via several transformations $n \to n^{(1)} \to n^{(2)} \to n^{(3)} \to n^{(4)} \to \hat{n}$, such that the RH problems satisfied by $n^{(1)}, n^{(2)}, n^{(3)},  n^{(4)}, \hat{n}$ are equivalent to the original RH problem \ref{RHn}. We will write $\Gamma^{(j)}, \hat{\Gamma}$ and $v^{(j)},\hat{v}$ for the contours and jump matrices of the RH problems for $n^{(j)}, \hat{n}$, respectively. At each step, the symmetries (\ref{nsymm}) and (\ref{vsymm}) will be preserved, so that, for $j = 1, 2,3,4$,
\begin{align}\label{vjsymm}
& v^{(j)}(x,t,k) = \mathcal{A} v^{(j)}(x,t,\omega k)\mathcal{A}^{-1}
 = \mathcal{B} v^{(j)}(x,t,k^{-1})^{-1}\mathcal{B}, & & k \in \Gamma^{(j)},
	\\ \label{njsymm}
& n^{(j)}(x,t, k) = n^{(j)}(x,t,\omega k)\mathcal{A}^{-1}
 = n^{(j)}(x,t, k^{-1}) \mathcal{B}, & & k \in \C \setminus \Gamma^{(j)},
\end{align}
and similarly for $\hat{n}$ and $\hat{v}$. Thanks to the symmetries \eqref{vjsymm}--\eqref{njsymm}, the transformations $n \to n^{(1)}$, $n^{(j)} \to n^{(j+1)}$, and $n^{(4)}\to \hat{n}$ will only be defined explicitly in the sector 
\begin{align}\label{mathsfSdef}
\mathsf{S}:= \big\{k \in \mathbb{C} \,|  \arg k \in [\tfrac{\pi}{3},\tfrac{2\pi}{3}]\big\}.
\end{align}
In the rest of the complex plane, they will be defined using the $\mathcal{A}$- and $\mathcal{B}$-symmetries.

Any point $k_0 \in \mathsf{Z}$ with $c_{k_0} = 0$ can be removed from $\mathsf{Z}$ without affecting RH problem \ref{RHn}. We will therefore without loss of generality henceforth assume that $c_{k_0} \neq 0$ for all $k_0 \in \mathsf{Z}$.

\section{The $n \to n^{(1)}$ transformation}\label{nton1sec}
We now begin the proof of Theorem \ref{asymptoticsth}. We recall the following facts from \cite[Theorem 2.3]{CLscatteringsolitons}: $r_1 \in C^\infty(\hat{\Gamma}_{1})$, $r_2 \in C^\infty(\hat{\Gamma}_{4}\setminus \{\pm \omega^{2}\})$, $r_{1}(\kappa_{j})\neq 0$ for $j=1,\ldots,6$, $r_{2}(k)$ has simple zeros at $k=\pm\omega$ and simple poles at $k=\pm \omega^2$, and $r_{1},r_{2}$ have rapid decay as $|k|\to \infty$. Moreover,
\begin{align}
& r_{1}(\tfrac{1}{\omega k}) + r_{2}(\omega k) + r_{1}(\omega^{2} k) r_{2}(\tfrac{1}{k}) = 0, & & k \in \partial \D\setminus \{\pm \omega\}, \label{r1r2 relation on the unit circle} \\
& r_{2}(k) = \tilde{r}(k)\overline{r_{1}(\bar{k}^{-1})}, \qquad \tilde{r}(k) :=\frac{\omega^{2}-k^{2}}{1-\omega^{2}k^{2}}, & & k \in \hat{\Gamma}_{4}\setminus \{0, \pm \omega^{2}\}, \label{r1r2 relation with kbar symmetry} \\
& r_{1}(1) = r_{1}(-1) = 1, \qquad r_{2}(1) = r_{2}(-1) = -1. \label{r1r2at0}
\end{align}

Define $\hat{r}_{1}(k)$ and $\hat{r}_{2}(k)$ for $k \in \partial \D$ with $\arg k \in [\tfrac{\pi}{2},\arg k_{1}]$ by
\begin{align}\label{def of rhat}
& \hat{r}_{j}(k) := \frac{r_{j}(k)}{1+r_{1}(k)r_{2}(k)} = \frac{r_{j}(k)}{1+\tilde{r}(k)|r_{1}(k)|^{2}}, \qquad j = 1,2,
\end{align}
where we have used \eqref{r1r2 relation with kbar symmetry} for the second equality. Since $\arg k_{1} \in (\frac{\pi}{2},\frac{2\pi}{3})$ for $\zeta>1$, and since $\tilde{r}(e^{i\theta})>0$ for $\theta \in [\frac{\pi}{2},\frac{2\pi}{3})$, we have $1+r_{1}(k)r_{2}(k)=1+\tilde{r}(k)|r_{1}(k)|^{2}>1$ for $|k|=1, \arg k \in [\tfrac{\pi}{2},\arg k_{1}]$. It follows that $\hat{r}_{1}$, $\hat{r}_{2}$ are well-defined.
To implement the transformation $n \to n^{(1)}$, we need analytic approximations of the functions $r_{1},r_{2}, \hat{r}_{1}$, and $\hat{r}_{2}$. 

For $K>1$, we define open sets $U_1 = U_1(\zeta,K), \hat{U}_1 = \hat{U}_1(\zeta,K) \subset \C$ as follows:
\begin{align*}
U_{1} = & \; \big(\{k \,| \arg k \in [-\pi,-\tfrac{5\pi}{6}) \cup (\arg k_{1},\pi], \; K^{-1}<|k|<1 \}
	\\
& \; \cup
\{k \,| \arg k \in (-\tfrac{\pi}{2},\tfrac{\pi}{6}), \; 1<|k|<K \}\big) \cap \{k  \,|\, \re \Phi_{21}>0\}, \\
\hat{U}_{1} = & \; \{k \,| \arg k \in (\tfrac{\pi}{2},\arg k_{1}), \; 1<|k|<K \} \cap \{k \,|\, \re \Phi_{21}>0\},
\end{align*}
see also Figure \ref{fig: U1 and U2}. Let $N \geq 1$ be an integer.


\begin{figure}
\begin{center}
\begin{tikzpicture}[master]
\node at (0,0) {\includegraphics[width=5cm]{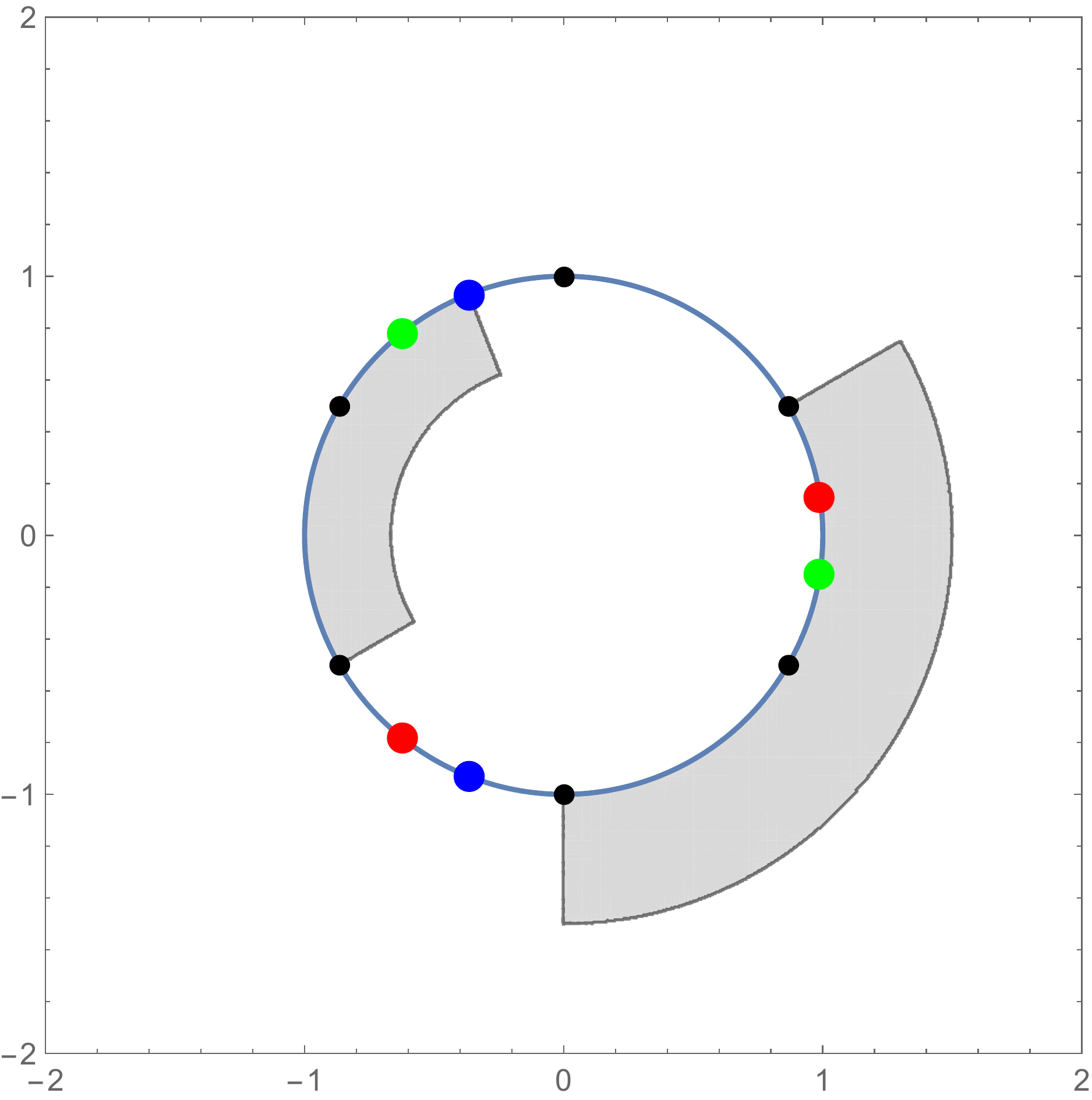}};
\node at (-0.87,0.25) {\tiny $U_{1}$};
\node at (1.2,-0.87) {\tiny $U_{1}$};
\end{tikzpicture} \hspace{0.1cm} \begin{tikzpicture}[slave]
\node at (0,0) {\includegraphics[width=5cm]{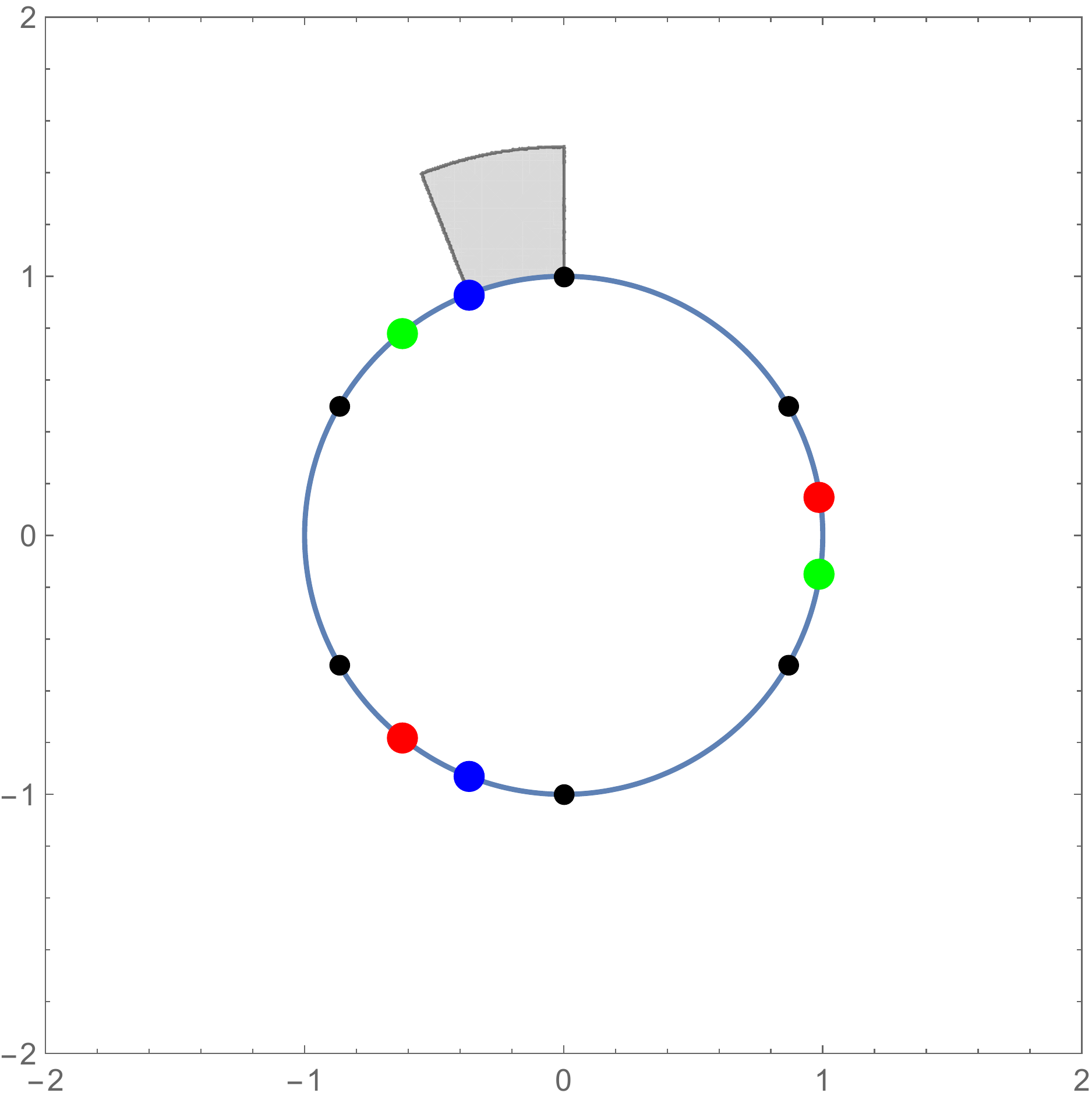}};
\node at (-0.15,1.5) {\tiny $\hat{U}_{1}$};
\end{tikzpicture}
\end{center}
\begin{figuretext}
\label{fig: U1 and U2} The open subsets $U_{1}$ and $\hat{U}_{1}$ of the complex $k$-plane.
\end{figuretext}
\end{figure}
\begin{lemma}[Decomposition lemma]\label{IIbis decompositionlemma}
There exist $K>1$ and decompositions
\begin{align*}
& r_{1}(k) = r_{1,a}(x, t, k) + r_{1,r}(x, t, k), & & k \in \partial U_{1} \cap \partial \D, \\
& \hat{r}_{1}(k) = \hat{r}_{1,a}(x, t, k) + \hat{r}_{1,r}(x, t, k), & & k \in \partial \hat{U}_{1} \cap \partial \D,
\end{align*}
such that the functions $r_{1,a}, r_{1,r}, \hat{r}_{1,a}, \hat{r}_{1,r}$ have the following properties:
\begin{enumerate}[$(a)$]
\item 
For each $\zeta \in \mathcal{I}$ and $t \geq 1$, $r_{1,a}(x, t, k)$ is defined and continuous for $k \in \bar{U}_1$ and analytic for $k \in U_1$, while $\hat{r}_{1,a}(x, t, k)$ is defined and continuous for $k \in \bar{\hat{U}}_1$ and analytic for $k \in \hat{U}_1$.

\item For $\zeta \in \mathcal{I}$ and $t\geq 1$, the functions $r_{1,a}$ and $\hat{r}_{1,a}$ satisfy
\begin{align*} 
& \Big| r_{1,a}(x, t, k)-\sum_{j=0}^{N}\frac{r_{1}^{(j)}(k_{\star})}{j!}(k-k_{\star})^{j}  \Big| \leq C |k-k_{\star}|^{N+1} e^{\frac{t}{4}|\re \Phi_{21}(\zeta,k)|}, & & k \in \bar{U}_{1}, \, k_{\star} \in \mathcal{R}, 
	\\ 
& \Big| \hat{r}_{1,a}(x, t, k)-\sum_{j=0}^{N}\frac{\hat{r}_{1}^{(j)}(k_{\star})}{j!}(k-k_{\star})^{j} \Big| \leq C|k-k_{\star}|^{N+1} e^{\frac{t}{4}|\re \Phi_{21}(\zeta,k)|},  & & k \in \bar{\hat{U}}_{1}, \, k_{\star} \in \hat{\mathcal{R}}, 
\end{align*}
where $\mathcal{R} =\{k_{1},\pm \omega, \omega^{2}k_{2},\pm e^{\frac{\pi i}{6}},\pm e^{-\frac{\pi i}{6}},\pm 1,-i,\omega^{2}k_{1},\omega k_{2}\}$, $\hat{\mathcal{R}} =\{i,k_{1}\}$, and the constant $C$ is independent of $\zeta, t, k$.
\item For each $1 \leq p \leq \infty$, the $L^p$-norm of $r_{1,r}(x,t,\cdot)$ on $\partial U_{1} \cap \partial \D$ is $O(t^{-N})$ and the $L^p$-norm of $\hat{r}_{1,r}(x,t,\cdot)$ on $\partial \hat{U}_{1} \cap \partial \D$ is $O(t^{-N})$ uniformly for $\zeta \in \mathcal{I}$ as $t \to \infty$.
\end{enumerate}
\end{lemma}
\begin{proof}
On each connected component of $\partial U_{1} \cap \partial \D$ and $\partial \hat{U}_{1} \cap \partial \D$, the function $\theta \mapsto \im \Phi_{21}(\zeta,e^{i\theta})=(\zeta-\cos \theta)\sin \theta$ is monotone and $\re  \Phi_{21}$ is identically zero. We can thus prove the statement using the method from \cite{DZ1993}. This method is standard by now, so we omit details. 
\end{proof}

The above lemma establishes decompositions of $r_{1}$ and $\hat{r}_{1}$; we now obtain decompositions of $r_{2}$ and $\hat{r}_{2}$ using the symmetry \eqref{r1r2 relation with kbar symmetry}. Let $U_{2}:=\{k|\bar{k}^{-1}\in U_{1}\}$ and $\hat{U}_{2}:=\{k|\bar{k}^{-1}\in \hat{U}_{1}\}$. We define decompositions $r_{2}=r_{2,a}+r_{2,r}$ and $\hat{r}_{2}=\hat{r}_{2,a}+\hat{r}_{2,r}$ by
\begin{align*}
& r_{2,a}(k) := \tilde{r}(k)\overline{r_{1,a}(\bar{k}^{-1})}, \quad k \in U_{2},
& & r_{2,r}(k) := \tilde{r}(k)\overline{r_{1,r}(\bar{k}^{-1})}, \quad k \in \partial U_{2}\cap \partial \D,
	\\
& \hat{r}_{2,a}(k) := \tilde{r}(k)\overline{\hat{r}_{1,a}(\bar{k}^{-1})}, \quad k \in \hat{U}_{2},
& & \hat{r}_{2,r}(k) := \tilde{r}(k)\overline{\hat{r}_{1,r}(\bar{k}^{-1})}, \quad k \in \partial \hat{U}_{2}\cap \partial \D.
\end{align*}

\begin{figure}[h]
\begin{center}
\begin{tikzpicture}[master,scale=0.9]
\node at (0,0) {};
\draw[black,line width=0.65 mm] (0,0)--(30:7.5);
\draw[black,line width=0.65 mm,->-=0.45,->-=0.91] (0,0)--(90:7);
\draw[black,line width=0.65 mm] (0,0)--(150:7.5);
\draw[dashed,black,line width=0.15 mm] (0,0)--(60:7.5);
\draw[dashed,black,line width=0.15 mm] (0,0)--(120:7.5);

\draw[black,line width=0.65 mm] ([shift=(30:3*1.5cm)]0,0) arc (30:150:3*1.5cm);
\draw[black,arrows={-Triangle[length=0.27cm,width=0.18cm]}]
($(73:3*1.5)$) --  ++(-15:0.001);
\draw[black,arrows={-Triangle[length=0.27cm,width=0.18cm]}]
($(97:3*1.5)$) --  ++(97-90:0.001);
\draw[black,arrows={-Triangle[length=0.27cm,width=0.18cm]}]
($(123-4:3*1.5)$) --  ++(120+90-4:0.001);

\node at (75.5:3.2*1.5) {\small $2$};

\node at (116:3.2*1.5) {\small $8$};

\node at (82.5:1.97*1.5) {\small $1''$};

\node at (86.5:6.15) {\small $4'$};

\node at (100:3.2*1.5) {\small $5$};

\draw[blue,fill] (110.3:4.5) circle (0.12cm);
\draw[green,fill] (129.688:4.5) circle (0.12cm);


\node at (110:4.08) {\small $k_{1}$};
\node at (130:4.08) {\small $\omega^{2}k_2$};

\end{tikzpicture}
\end{center}
\begin{figuretext}
\label{IIbis Gammap0p}The contour $\Gamma^{(0)}$ (solid), the boundary of $\mathsf{S}$ (dashed) and the saddle points $k_{1}$ (blue) and $\omega^{2}k_{2}$ (green).
\end{figuretext}
\end{figure}

We are now in a position to define the first transformation $n \to n^{(1)}$. As explained in Section \ref{overviewsec}, we will focus our attention on the sector $\mathsf{S}$ defined in (\ref{mathsfSdef}). Let $\Gamma^{(0)}$ be as in Figure \ref{IIbis Gammap0p}. Note that 
\begin{align*}
\Gamma_{8}^{(0)}:= \{e^{i\theta} \,|\, \theta \in (\arg k_{1},\tfrac{2\pi}{3})\}, \qquad \Gamma_{2}^{(0)}:= \{e^{i\theta} \,|\, \theta \in (\tfrac{\pi}{3},\tfrac{\pi}{2})\}.
\end{align*} 
The contour $\Gamma^{(0)}$ coincides as a set with $\Gamma\cap \mathsf{S}$, but is oriented and labeled differently.

On $\Gamma_{2}^{(0)}$, we will use the factorization
\begin{align*}
v_{9} = v_{3}^{(1)}v_{2}^{(1)}v_{1}^{(1)},
\end{align*}
where
\begin{align}
& v_{3}^{(1)} = \begin{pmatrix}
1 & 0 & -r_{2,a}(\omega^{2}k)e^{-\theta_{31}} \\
r_{1,a}(\frac{1}{k})e^{\theta_{21}} & 1 & r_{1,a}(\omega k)e^{-\theta_{32}} \\
0 & 0 & 1
\end{pmatrix}, \; v_{1}^{(1)} = \begin{pmatrix}
1 & r_{2,a}(\frac{1}{k})e^{-\theta_{21}} & 0 \\
0 & 1 & 0  \\
-r_{1,a}(\omega^{2}k)e^{\theta_{31}} & r_{2,a}(\omega k)e^{\theta_{32}} & 1
\end{pmatrix}, \nonumber \\
& v_{2}^{(1)} = I + v_{2,r}^{(1)}, \qquad v_{2,r}^{(1)}= \begin{pmatrix}
r_{1,r}(\omega^{2}k)r_{2,r}(\omega^{2}k) & g_{2}(\omega k)e^{-\theta_{21}} & -r_{2,r}(\omega^{2}k)e^{-\theta_{31}} \\
g_{1}(\omega k)e^{\theta_{21}} & g(\omega k) & h_{1}(\omega k)e^{-\theta_{32}} \\
-r_{1,r}(\omega^{2}k)e^{\theta_{31}} & h_{2}(\omega k)e^{\theta_{32}} & 0
\end{pmatrix}, \label{vp1p 123}
\end{align}
and
\begin{align*}
h_{1}(k) = & \; r_{1,r}(k) + r_{1,a}(\tfrac{1}{\omega^{2}k})r_{2,r}(\omega k), \qquad h_{2}(k) =  r_{2,r}(k) + r_{2,a}(\tfrac{1}{\omega^{2}k})r_{1,r}(\omega k), \\
g_{1}(k) = & \; r_{1,r}(\tfrac{1}{\omega^{2}k})-r_{1,r}(\omega k) \big( r_{1,r}(k)+r_{1,a}(\tfrac{1}{\omega^{2}k})r_{2,r}(\omega k) \big), \\
g_{2}(k) = & \; r_{2,r}(\tfrac{1}{\omega^{2}k})-r_{2,r}(\omega k) \big( r_{2,r}(k)+r_{2,a}(\tfrac{1}{\omega^{2}k})r_{1,r}(\omega k) \big), \\
g(k) = & \; r_{1,r}(k)\big(r_{1,r}(\omega k)r_{2,a}(\tfrac{1}{\omega^{2}k})+r_{2,r}(k)\big) \\
& +r_{1,a}(\tfrac{1}{\omega^{2}k})r_{2,r}(\omega k)\big( r_{1,r}(\omega k)r_{2,a}(\tfrac{1}{\omega^{2}k})+r_{2,r}(k) \big) + r_{1,r}(\tfrac{1}{\omega^{2}k})r_{2,r}(\tfrac{1}{\omega^{2}k}).
\end{align*}
On $\Gamma_{5}^{(0)}$, we will use the factorization
\begin{align}
& v_{7}^{-1} = v_{6}^{(1)}v_{5}^{(1)}v_{4}^{(1)}, \qquad v_{5}^{(1)} = \begin{pmatrix}
1+r_{1}(k)r_{2}(k) & 0 & 0 \\
0 & \frac{1}{1+r_{1}(k)r_{2}(k)} & 0 \\
0 & 0 & 1
\end{pmatrix} + v_{5,r}^{(1)}, \nonumber \\
& v_{6}^{(1)} \hspace{-0.07cm} = \hspace{-0.07cm} \begin{pmatrix}
1 & 0 & r_{1,a}(\frac{1}{\omega^{2} k})e^{-\theta_{31}} \\[0.08cm]
\hat{r}_{2,a}(k)e^{\theta_{21}} & 1 & -r_{2,a}(\frac{1}{\omega k})e^{-\theta_{32}} \\
0 & 0 & 1
\end{pmatrix}, \; v_{4}^{(1)} \hspace{-0.07cm} = \hspace{-0.07cm} \begin{pmatrix}
1 & \hspace{-0.07cm} \hat{r}_{1,a}(k)e^{-\theta_{21}} & \hspace{-0.07cm} 0 \\
0 & \hspace{-0.07cm} 1 & \hspace{-0.07cm} 0 \\
r_{2,a}(\tfrac{1}{\omega^{2} k})e^{\theta_{31}} & \hspace{-0.07cm} -r_{1,a}(\tfrac{1}{\omega k})e^{\theta_{32}} & \hspace{-0.07cm} 1
\end{pmatrix}\hspace{-0.07cm}, \label{vp1p 789}
\end{align}
where we have used \eqref{r1r2 relation on the unit circle}. Finally, on $\Gamma_{8}^{(0)}$, we will use the factorization
\begin{align}
& v_{7} = v_{7}^{(1)}v_{8}^{(1)}v_{9}^{(1)}, \;\;\; v_{8}^{(1)} = I + v_{8,r}^{(1)}, \;\;\; v_{8,r}^{(1)} = \begin{pmatrix}
0 & -r_{1,r}(k)e^{-\theta_{21}} & h_{2}(\omega^{2}k)e^{-\theta_{31}} \\
-r_{2,r}(k)e^{\theta_{21}} & r_{1,r}(k)r_{2,r}(k) & g_{2}(\omega^{2}k)e^{-\theta_{32}} \\
h_{1}(\omega^{2}k)e^{\theta_{31}} & g_{1}(\omega^{2}k)e^{\theta_{32}} & g(\omega^{2}k)
\end{pmatrix},
	 \nonumber \\
& v_{7}^{(1)} = \begin{pmatrix}
1 & 0 & 0 \\
-r_{2,a}(k)e^{\theta_{21}} & 1 & 0 \\
r_{1,a}(\omega^{2}k)e^{\theta_{31}} & r_{1,a}(\frac{1}{\omega k})e^{\theta_{32}} & 1
\end{pmatrix}, \quad v_{9}^{(1)} = \begin{pmatrix}
1 & -r_{1,a}(k)e^{-\theta_{21}} & r_{2,a}(\omega^{2}k)e^{-\theta_{31}} \\
0 & 1 & r_{2,a}(\frac{1}{\omega k})e^{-\theta_{32}} \\
0 & 0 & 1 
\end{pmatrix}. \label{vp1p 101112}
\end{align}
We omit the long expression for $v_{5,r}^{(1)}$, but note that Lemma \ref{IIbis decompositionlemma} implies that
\begin{align*}
\| v_{j,r}^{(1)} \|_{(L^{1}\cap L^{\infty})(\Gamma_{j}^{(0)})} = O(t^{-1}) \qquad \mbox{as } t \to \infty, \; j= 2,5,8.
\end{align*}

\begin{figure}
\begin{center}
\begin{tikzpicture}[master]
\node at (0,0) {};
\draw[black,line width=0.65 mm] (0,0)--(30:7.5);
\draw[black,line width=0.65 mm,->-=0.25,->-=0.57,->-=0.71,->-=0.91] (0,0)--(90:7.5);
\draw[black,line width=0.65 mm] (0,0)--(150:7.5);
\draw[dashed,black,line width=0.15 mm] (0,0)--(60:7.5);
\draw[dashed,black,line width=0.15 mm] (0,0)--(120:7.5);

\draw[black,line width=0.65 mm] ([shift=(30:3*1.5cm)]0,0) arc (30:150:3*1.5cm);
\draw[black,arrows={-Triangle[length=0.27cm,width=0.18cm]}]
($(73:3*1.5)$) --  ++(-17:0.001);
\draw[black,arrows={-Triangle[length=0.27cm,width=0.18cm]}]
($(119:3*1.5)$) --  ++(116+90:0.001);

\draw[black,line width=0.65 mm] ([shift=(30:3*1.5cm)]0,0) arc (30:150:3*1.5cm);

\node at (75:3.15*1.5) {\scriptsize $2$};
\node at (75:4.03*1.5) {\scriptsize $1$};
\node at (75:2.58*1.5) {\scriptsize $3$};

\node at (116.5:3.155*1.5) {\scriptsize $8$};
\node at (115.5:3.65*1.5) {\scriptsize $7$};
\node at (114:2.55*1.5) {\scriptsize $9$};

\node at (99.5:1.15*1.5) {\scriptsize $1''$};
\node at (93.5:2.75*1.5) {\scriptsize $1_{r}''$};

\node at (92:6.6) {\scriptsize $4'$};
\node at (93.1:3.4*1.5) {\scriptsize $4_{r}'$};

\node at (100:3.17*1.5) {\scriptsize $5$};
\node at (100:3.62*1.5) {\scriptsize $4$};
\node at (100:2.45*1.5) {\scriptsize $6$};

\node at (110:3.22*1.5) {\scriptsize $k_1$};
\node at (131:3.31*1.5) {\scriptsize $\omega^2 k_2$};

\draw[black,arrows={-Triangle[length=0.27cm,width=0.18cm]}]
($(97:3*1.5)$) --  ++(97-90:0.001);

\draw[black,line width=0.65 mm] (150:3.65)--($(129.688:3*1.5)+(129.688+135:0.5)$)--(129.688:3*1.5)--($(129.688:3*1.5)+(129.688+45:0.5)$)--(150:5.8);
\draw[black,line width=0.65 mm,-<-=0.13,->-=0.78] (90:3.65)--($(110.3:3*1.5)+(110.3-135:0.5)$)--(110.3:3*1.5)--($(110.3:3*1.5)+(110.3-45:0.5)$)--(90:5.8);
\draw[black,line width=0.65 mm,-<-=0.12,->-=0.78] (120:3.65)--($(110.3:3*1.5)+(110.3+135:0.5)$)--(110.3:3*1.5)--($(110.3:3*1.5)+(110.3+45:0.5)$)--(120:5.8);
\draw[black,line width=0.65 mm] (120:3.65)--($(129.688:3*1.5)+(129.688-135:0.5)$)--(129.688:3*1.5)--($(129.688:3*1.5)+(129.688-45:0.5)$)--(120:5.8);

\draw[green,fill] (129.688:3*1.5) circle (0.12cm);
\draw[blue,fill] (110.3:3*1.5) circle (0.12cm);

\draw[black,line width=0.65 mm,->-=0.6] (60:4.5)--(60:5.8);
\draw[black,line width=0.65 mm,->-=0.75] (60:3.65)--(60:4.5);
\draw[black,line width=0.65 mm,->-=0.6] (120:4.5)--(120:5.8);
\draw[black,line width=0.65 mm,->-=0.75] (120:3.65)--(120:4.5);

\draw[black,line width=0.65 mm,-<-=0.70] ([shift=(30:3.65cm)]0,0) arc (30:90:3.65cm);
\draw[black,line width=0.65 mm,-<-=0.71] ([shift=(30:5.8cm)]0,0) arc (30:90:5.8cm);

\node at (122:5) {\scriptsize $1_{s}$};
\node at (122.5:4.2) {\scriptsize $2_{s}$};
\node at (57:5.1) {\scriptsize $3_{s}$};
\node at (56.5:4.15) {\scriptsize $4_{s}$};
\end{tikzpicture}
\end{center}
\begin{figuretext}
\label{IIbis Gammap1p}The contour $\Gamma^{(1)}$ (solid), the boundary of $\mathsf{S}$ (dashed), and the saddle points $k_{1}$ (blue) and $\omega^{2}k_{2}$ (green).
\end{figuretext}
\end{figure}

Let $\Gamma^{(1)}$ be the contour shown in Figure \ref{IIbis Gammap1p}. The function $n^{(1)}$ is defined by
\begin{align}\label{Sector IIbis first transfo}
n^{(1)}(x,t,k) = n(x,t,k)G^{(1)}(x,t,k),\qquad k \in \C \setminus \Gamma^{(1)},
\end{align}
where $G^{(1)}$ is analytic in $\mathbb{C}\setminus \Gamma^{(1)}$. It is given for $k \in \mathsf{S}$ by
\begin{align}\label{IIbis Gp1pdef}
\hspace{-0.1cm} G^{(1)} \hspace{-0.1cm} = \hspace{-0.1cm} \begin{cases} 
v_{3}^{(1)}, & \hspace{-0.3cm} k \mbox{ on the $-$ side of }\Gamma_{2}^{(1)}, \\
(v_{1}^{(1)})^{-1}, & \hspace{-0.3cm} k \mbox{ on the $+$ side of }\Gamma_{2}^{(1)},  \\
v_{6}^{(1)}, & \hspace{-0.3cm} k \mbox{ on the $-$ side of }\Gamma_{5}^{(1)}, \\
(v_{4}^{(1)})^{-1}, & \hspace{-0.3cm} k \mbox{ on the $+$ side of }\Gamma_{5}^{(1)},
\end{cases} \; G^{(1)} \hspace{-0.1cm} = \hspace{-0.1cm} \begin{cases} 
v_{7}^{(1)}, & \hspace{-0.3cm} k \mbox{ on the $-$ side of }\Gamma_{8}^{(1)}, \\
(v_{9}^{(1)})^{-1}, & \hspace{-0.3cm} k \mbox{ on the $+$ side of }\Gamma_{8}^{(1)},  \\
I, & \hspace{-0.3cm} \mbox{otherwise},
\end{cases}
\end{align}
and extended to $\mathbb{C}\setminus \Gamma^{(1)}$ by
\begin{align}\label{symmetry of G1}
G^{(1)}(x,t, k) = \mathcal{A} G^{(1)}(x,t,\omega k)\mathcal{A}^{-1}
 = \mathcal{B} G^{(1)}(x,t, k^{-1}) \mathcal{B}.
\end{align}
The next lemma follows from the signature tables of Figure \ref{IIbis fig: Re Phi 21 31 and 32 for zeta=0.7}.

\begin{lemma}\label{lemma:G1p1p}
$G^{(1)}(x,t,k)$ and $G^{(1)}(x,t,k)^{-1}$ are uniformly bounded for $k \in \mathbb{C}\setminus \Gamma^{(1)}$, $t \geq 1$, and $\zeta \in \mathcal{I}$. Furthermore, $G^{(1)}(x,t,k)=I$ for all large enough $|k|$.
\end{lemma}

Using \eqref{Sector IIbis first transfo}, we infer that $n^{(1)}_{+}=n^{(1)}_{-}v_{j}^{(1)}$ on $\Gamma^{(1)}_{j}$, where $v_{j}^{(1)}$ is given by \eqref{vp1p 123} for $j=1,2,3$, by \eqref{vp1p 789} for $j=4,5,6$, by \eqref{vp1p 101112} for $j=7,8,9$, and
\begin{align}
v_{1''}^{(1)} = & \; v_{1''}, \quad v_{4'}^{(1)} = v_{4'}, \quad v_{1_{r}''}^{(1)} = (v_{3}^{(1)})^{-1} v_{1''}v_{6}^{(1)}, \quad v_{4_{r}'}^{(1)} = v_{1}^{(1)}v_{4'}(v_{4}^{(1)})^{-1},
	 \nonumber \\
 v_{1_s}^{(1)} = &\; G_-^{(1)}(k)^{-1}G_+^{(1)}(k)
= v_{7}^{(1)}(k)^{-1} \mathcal{A} \mathcal{B} v_{9}^{(1)}(\tfrac{1}{\omega k})^{-1} \mathcal{B} \mathcal{A}^{-1}  
= \begin{pmatrix}
 1 & 0 & 0
 \\
 0 & 1 & 0 \\
 (v_{1_s}^{(1)})_{31} & 0 & 1 \\
\end{pmatrix}, \label{v1sexpression}
	\\ \nonumber
 (v_{1_s}^{(1)})_{31} := &\;  -\big(r_{1,a}(\omega^{2}k)+r_{2,a}(k)r_{1,a}(\tfrac{1}{\omega k})+r_{2,a}(\tfrac{1}{\omega^{2}k})\big) e^{\theta_{31}}.
\end{align}
In general, $(v_{1_s}^{(1)})_{31} \neq 0$ for $k \in \Gamma_{1_{s}}^{(1)}$, even though $r_{1}(\omega^{2}k)+r_{2}(k)r_{1}(\tfrac{1}{\omega k})+r_{2}(\tfrac{1}{\omega^{2}k})=0$ for $k \in \partial \D$ by \eqref{r1r2 relation on the unit circle}. The matrices $v_{2_{s}}, v_{3_{s}}, v_{4_{s}}$ are similar to $v_{1_{s}}$ so their explicit expressions are omitted.

\section{The $n^{(1)} \to n^{(2)}$ transformation}\label{n1ton2sec}
For the second transformation, we will use the following factorizations (the matrices with subscripts $u$ and $d$ will be used to deform contours up and down, respectively, see \eqref{IIbis Gp2pdef} below):
\begin{align}\nonumber
& v_{4}^{(1)}\hspace{-0.07cm}=\hspace{-0.07cm}v_{4}^{(2)}v_{4u}^{(2)}, \;\;\; v_{4}^{(2)}\hspace{-0.07cm} =\hspace{-0.07cm} \begin{pmatrix}
1 & \hat{r}_{1,a}(k) e^{-\theta_{21}} & 0 \\
0 & 1 & 0 \\
0 & 0 & 1
\end{pmatrix}\hspace{-0.07cm}, \;\;\; v_{4u}^{(2)}\hspace{-0.07cm} =\hspace{-0.07cm} \begin{pmatrix}
1 & 0 & 0 \\
0 & 1 & 0 \\
r_{2,a}(\frac{1}{\omega^{2}k}) e^{\theta_{31}} & -r_{1,a}(\frac{1}{\omega k})e^{\theta_{32}} & 1
\end{pmatrix}\hspace{-0.07cm}, \nonumber \\
& v_{6}^{(1)}=v_{6d}^{(2)}v_{6}^{(2)}, \quad v_{6d}^{(2)} = \begin{pmatrix}
1 & 0 & r_{1,a}(\frac{1}{\omega^{2} k})e^{-\theta_{31}} \\[0.05cm]
0 & 1 & -r_{2,a}(\frac{1}{\omega k})e^{-\theta_{32}} \\
0 & 0 & 1
\end{pmatrix}, \quad v_{6}^{(2)} = \begin{pmatrix}
1 & 0 & 0 \\
\hat{r}_{2,a}(k) e^{\theta_{21}} & 1 & 0 \\
0 & 0 & 1
\end{pmatrix}, \nonumber \\
& v_{7}^{(1)} \hspace{-0.07cm} = \hspace{-0.07cm} v_{7u}^{(2)}v_{7}^{(2)}\hspace{-0.07cm}, \;\; v_{7}^{(2)} \hspace{-0.07cm}=\hspace{-0.07cm} \begin{pmatrix}
1 & \hspace{-0.07cm}0 & \hspace{-0.07cm}0 \\
-r_{2,a}(k)e^{\theta_{21}} & \hspace{-0.07cm}1 & \hspace{-0.07cm}0 \\
0 & \hspace{-0.07cm}0 & \hspace{-0.07cm}1
\end{pmatrix}\hspace{-0.07cm}, \;\; v_{9}^{(1)} \hspace{-0.07cm}=\hspace{-0.07cm} v_{9}^{(2)}v_{9d}^{(2)}, \;\; v_{9}^{(2)} \hspace{-0.07cm}=\hspace{-0.07cm} \begin{pmatrix}
1 & \hspace{-0.07cm}-r_{1,a}(k)e^{-\theta_{21}} & \hspace{-0.07cm}0 \\
0 & \hspace{-0.07cm}1 & \hspace{-0.07cm}0 \\
0 & \hspace{-0.07cm}0 & \hspace{-0.07cm}1
\end{pmatrix}\hspace{-0.07cm}, \nonumber \\
& v_{7u}^{(2)} = \begin{pmatrix}
1 & 0 & 0 \\
0 & 1 & 0 \\
(r_{1,a}(\omega^{2}k)+r_{1,a}(\frac{1}{\omega k})r_{2,a}(k))e^{\theta_{31}} & r_{1,a}(\frac{1}{\omega k})e^{\theta_{32}} & 1
\end{pmatrix}, \nonumber \\
& v_{9d}^{(2)} = \begin{pmatrix}
1 & 0 & (r_{2,a}( \omega^{2}k)+r_{1,a}(k)r_{2,a}(\frac{1}{\omega k}))e^{-\theta_{31}} \\
0 & 1 & r_{2,a}(\frac{1}{\omega k})e^{-\theta_{32}} \\
0 & 0 & 1
\end{pmatrix}. \label{Vv2def}
\end{align}
\begin{figure}
\begin{center}
\begin{tikzpicture}[master]
\node at (0,0) {};
\draw[black,line width=0.65 mm] (0,0)--(30:7.5);
\draw[black,line width=0.65 mm,->-=0.25,->-=0.57,->-=0.71,->-=0.91] (0,0)--(90:7.5);
\draw[black,line width=0.65 mm] (0,0)--(150:7.5);
\draw[dashed,black,line width=0.15 mm] (0,0)--(60:7.5);
\draw[dashed,black,line width=0.15 mm] (0,0)--(120:7.5);

\draw[black,line width=0.65 mm] ([shift=(30:3*1.5cm)]0,0) arc (30:150:3*1.5cm);
\draw[black,arrows={-Triangle[length=0.27cm,width=0.18cm]}]
($(73:3*1.5)$) --  ++(-17:0.001);
\draw[black,arrows={-Triangle[length=0.27cm,width=0.18cm]}]
($(119:3*1.5)$) --  ++(116+90:0.001);

\draw[black,line width=0.65 mm] ([shift=(30:3*1.5cm)]0,0) arc (30:150:3*1.5cm);

\node at (75:3.15*1.5) {\scriptsize $2$};
\node at (75:4.03*1.5) {\scriptsize $1$};
\node at (75:2.58*1.5) {\scriptsize $3$};

\node at (116.5:3.13*1.5) {\scriptsize $8$};
\node at (115.5:3.55*1.5) {\scriptsize $7$};
\node at (115:2.6*1.5) {\scriptsize $9$};

\node at (99.5:1.15*1.5) {\scriptsize $1''$};
\node at (93.5:2.75*1.5) {\scriptsize $1_{r}''$};

\node at (92:6.6) {\scriptsize $4'$};
\node at (93.1:3.4*1.5) {\scriptsize $4_{r}'$};

\node at (100:3.15*1.5) {\scriptsize $5$};
\node at (100:3.59*1.5) {\scriptsize $4$};
\node at (101.2:2.78*1.5) {\scriptsize $6$};

\draw[black,arrows={-Triangle[length=0.27cm,width=0.18cm]}]
($(97:3*1.5)$) --  ++(97-90:0.001);

\draw[black,line width=0.65 mm] (150:3.65)--($(129.688:3*1.5)+(129.688+135:0.5)$)--(129.688:3*1.5)--($(129.688:3*1.5)+(129.688+45:0.5)$)--(150:5.8);
\draw[black,line width=0.65 mm,-<-=0.13,->-=0.78] (90:3.65)--($(110.3:3*1.5)+(110.3-135:0.5)$)--(110.3:3*1.5)--($(110.3:3*1.5)+(110.3-45:0.5)$)--(90:5.8);
\draw[black,line width=0.65 mm,-<-=0.12,->-=0.78] (120:3.65)--($(110.3:3*1.5)+(110.3+135:0.5)$)--(110.3:3*1.5)--($(110.3:3*1.5)+(110.3+45:0.5)$)--(120:5.8);
\draw[black,line width=0.65 mm] (120:3.65)--($(129.688:3*1.5)+(129.688-135:0.5)$)--(129.688:3*1.5)--($(129.688:3*1.5)+(129.688-45:0.5)$)--(120:5.8);

\draw[black,line width=0.65 mm,->-=0.6] (60:4.5)--(60:5.8);
\draw[black,line width=0.65 mm,->-=0.75] (60:3.65)--(60:4.5);
\draw[black,line width=0.65 mm,->-=0.6] (120:4.5)--(120:5.8);
\draw[black,line width=0.65 mm,->-=0.75] (120:3.65)--(120:4.5);

\draw[black,line width=0.65 mm,-<-=0.70] ([shift=(30:3.65cm)]0,0) arc (30:90:3.65cm);
\draw[black,line width=0.65 mm,-<-=0.71] ([shift=(30:5.8cm)]0,0) arc (30:90:5.8cm);

\node at (122:5) {\scriptsize $1_{s}$};
\node at (122.5:4.2) {\scriptsize $2_{s}$};
\node at (57:5.1) {\scriptsize $3_{s}$};
\node at (56.5:4.15) {\scriptsize $4_{s}$};

\draw[black,line width=0.65 mm] (129.688:3.65)--(129.688:5.8);
\draw[black,line width=0.65 mm,-<-=0.12,->-=0.78] (110.3:3.65)--(110.3:5.8);

\draw[black,line width=0.65 mm] ([shift=(30:3.65cm)]0,0) arc (30:150:3.65cm);
\draw[black,line width=0.65 mm] ([shift=(30:5.8cm)]0,0) arc (30:150:5.8cm);

\node at (108:5.2) {\scriptsize $5_{s}$};
\node at (107.5:3.9) {\scriptsize $6_{s}$};

\draw[green,fill] (129.688:3*1.5) circle (0.12cm);
\draw[blue,fill] (110.3:3*1.5) circle (0.12cm);
\end{tikzpicture}
\end{center}
\begin{figuretext}
\label{IIbis Gammap2p}The contour $\Gamma^{(2)}$ (solid), the boundary of $\mathsf{S}$ (dashed), and the saddle points $k_{1}$ (blue) and $\omega^{2}k_{2}$ (green).
\end{figuretext}
\end{figure}
Let $\Gamma^{(2)}$ be the contour shown in Figure \ref{IIbis Gammap2p}. The function $n^{(2)}$ is defined by
\begin{align}\label{Sector IIbis second transfo}
n^{(2)}(x,t,k) = n^{(1)}(x,t,k)G^{(2)}(x,t,k), \qquad k \in \C \setminus \Gamma^{(2)},
\end{align}
where $G^{(2)}$ is given for $k \in \mathsf{S}$ by
\begin{align}\label{IIbis Gp2pdef}
G^{(2)}(x,t,k) = \begin{cases} 
(v_{4u}^{(2)})^{-1}, & \hspace{-0.15cm}k \mbox{ above }\Gamma_{4}^{(2)}, \\
v_{6d}^{(2)}, & \hspace{-0.15cm}k \mbox{ below }\Gamma_{6}^{(2)}, \\
v_{7u}^{(2)}, & \hspace{-0.15cm}k \mbox{ above }\Gamma_{7}^{(2)}, \\
(v_{9d}^{(2)})^{-1}, & \hspace{-0.15cm}k \mbox{ below }\Gamma_{9}^{(2)}, \\
I, & \hspace{-0.15cm}\mbox{otherwise},
\end{cases}
\end{align}
and $G^{(2)}$ is defined to $\mathbb{C}\setminus \Gamma^{(2)}$ using the $\mathcal{A}$- and $\mathcal{B}$-symmetries (as in \eqref{symmetry of G1}). The following lemma follows from Lemma \ref{IIbis decompositionlemma} and Figure \ref{IIbis fig: Re Phi 21 31 and 32 for zeta=0.7}.

\begin{lemma}
$G^{(2)}(x,t,k)$ and $G^{(2)}(x,t,k)^{-1}$ are uniformly bounded for $k \in \mathbb{C}\setminus \Gamma^{(2)}$, $t\geq 1$, and $\zeta \in \mathcal{I}$. Furthermore, $G^{(2)}(x,t,k)=I$ for all large enough $|k|$.
\end{lemma}
The jumps $v_j^{(2)}$ of $n^{(2)}$ are given for $j = 4,6,7,9$ by (\ref{Vv2def}) and by
\begin{align}
& v_j^{(2)} := v_j^{(1)}, \; j=1,2,3,5,8,1_{s},2_{s},3_{s},4_{s}, \qquad
v_{5_s}^{(2)} :=  v_{4u}^{(2)} v_{7u}^{(2)},\qquad v_{6_s}^{(2)} :=  v_{9d}^{(2)} v_{6d}^{(2)}. \label{II jumps vj diagonal p2p}
\end{align}


On the subcontours of $\Gamma^{(2)}\cap \mathsf{S}$ that are unlabeled in Figure \ref{IIbis Gammap2p}, the matrix $v^{(2)}$ is close to $I$ as $t \to \infty$, so we do not write it down. On $\Gamma^{(2)}\setminus \mathsf{S}$, $v^{(2)}$ can be computed using \eqref{vjsymm}. The proof of the next lemma is similar to \cite[Proof of Lemma 8.5]{CLmain}, so we omit it.

\begin{lemma}\label{vsmallnearilemmaV}
The $L^\infty$-norm of $v^{(2)} - I$ on $\Gamma_{4_r'}^{(2)} \cup \Gamma_{1_r''}^{(2)}$ is $O(t^{-N})$ as $t \to \infty$ uniformly for $\zeta \in \mathcal{I}$.
\end{lemma}

As in \cite[Lemma 6.3]{CLsectorIV}, we can (and do) choose the analytic approximations of Lemma \ref{IIbis decompositionlemma} so that the $L^\infty$-norm of $v_j^{(2)} - I$ on $\Gamma_j^{(2)}$, $j = 1_s, \dots, 6_s$, is $O(t^{-N})$ as $t \to \infty$ uniformly for $\zeta \in \mathcal{I}$. Thus $v^{(2)}-I=O(t^{-N})$ as $t \to \infty$, uniformly for $k \in \Gamma^{(2)} \setminus \big(\cup_{j=0}^{2}\cup_{\ell=1}^{2} D_{\epsilon}(\omega^{j} k_{\ell}) \cup \partial \D \big)$.

\section{The $n^{(2)}\to n^{(3)}$ transformation}\label{n2ton3sec}
The goal of the transformation $n^{(2)}\to n^{(3)}$ is to make $v^{(3)}-I$ uniformly small as $t \to \infty$ for $k \in \partial \D$. 

For each $\zeta \in \mathcal{I}$, we seek an analytic function $\delta(\zeta, \cdot): \mathbb{C}\setminus \Gamma_5^{(2)} \to \mathbb{C}$ obeying the jump condition
\begin{align*}
& \delta_{+}(\zeta, k) = \delta_{-}(\zeta, k)(1 + r_{1}(k)r_{2}(k)), \qquad k \in \Gamma_{5}^{(2)},
\end{align*}
and the normalization condition $\delta(\zeta, k) = 1 + O(k^{-1})$ as $k \to \infty$. As mentioned previously (see the text below \eqref{def of rhat}), $1 + r_{1}(k)r_{2}(k)>1$ for all $k \in \Gamma_{5}^{(2)}$. Hence, by taking the logarithm and using Plemelj's formula, we find that 
\begin{equation} \label{delta1def}
\delta(\zeta, k) = \exp \left\{ \frac{-1}{2\pi i} \int_{i}^{k_{1}} \frac{\ln(1 + r_1(s)r_{2}(s))}{s - k} ds \right\}, \qquad k \in \C \setminus \Gamma_{5}^{(2)},
\end{equation}
where the principal branch is taken for the logarithm, and where the path of integration follows the unit circle in the counterclockwise direction. 

\begin{lemma}\label{IIbis deltalemma}
The function $\delta(\zeta, k)$ has the following properties:
\begin{enumerate}[$(a)$]
\item\label{deltaparta}
 $\delta(\zeta,k)$ can be written as
\begin{align}
& \delta(\zeta,k) = e^{-i \nu \ln_{k_{1}}(k-k_{1})}e^{-\chi(\zeta,k)}, \label{delta expression in terms of log and chi} \\
& \chi(\zeta,k) = \frac{-1}{2\pi i} \int_{i}^{k_{1}}  \ln_{s}(k-s) d\ln(1+r_1(s)r_{2}(s)), \label{def of chi}
\end{align}
where $\nu \leq 0$ is defined in \eqref{nudef} and the path of integration follows the unit circle in the counterclockwise direction. 
For $s \in \Gamma_{5}^{(2)}$, $k \mapsto \ln_{s}(k-s)=\ln |k-s|+i \arg_{s}(k-s)$ has a cut along $\{e^{i \theta} \,|\, \theta \in [\frac{\pi}{2},\arg s], \; \arg s \in [\frac{\pi}{2},\frac{2\pi}{3}]\}\cup(i,i\infty)$ and $\arg_{s}(1)=2\pi$.

\item\label{deltapartb} 
For each $\zeta \in \mathcal{I}$, $\delta(\zeta, k)$ and $\delta(\zeta, k)^{-1}$ are analytic functions of $k \in \mathbb{C}\setminus \Gamma_{5}^{(2)}$. Moreover,
\begin{align*}
& \sup_{\zeta \in \mathcal{I}} \sup_{\substack{k \in \C \setminus \Gamma_{5}^{(2)}}} |\delta(\zeta,k)^{\pm 1}| < \infty.
\end{align*}

\item\label{deltapartc} 
For $k \in \partial \D\setminus \Gamma_{5}^{(2)}$, 
\begin{align}\label{re chi on part of the unit circle}
\re \chi(\zeta,k) = \frac{\pi + \arg_{i} k}{2}\nu + \frac{1}{2\pi} \int_{\Gamma_{5}^{(2)}} \frac{\arg s}{2} d\ln(1+r_1(s)r_{2}(s)),
\end{align}
where $\arg_i k \in (\frac{\pi}{2},\frac{5\pi}{2})$ and $\arg s \in (-\pi,\pi)$. In particular, for $k \in \partial \D\setminus \Gamma_{5}^{(2)}$, $|\delta(\zeta,k)|$ is constant and given by
\begin{align}\label{|delta| is constant on a part of the unit circle}
|\delta(\zeta,k)| = \exp \bigg( \nu \frac{\arg k_{1}}{2} - \frac{1}{2\pi} \int_{\Gamma_{5}^{(2)}} \frac{\arg s}{2} d\ln(1+r_1(s)r_{2}(s)) \bigg),
\end{align}
where $\arg k_{1} \in [\frac{\pi}{2},\frac{2\pi}{3}]$.

\item\label{deltapartd} 
As $k \to k_{1}$ along a path which is nontangential to $\Gamma_{5}^{(2)}$, we have
\begin{align}
& |\chi(\zeta,k)-\chi(\zeta,k_{1})| \leq C |k-k_{1}| (1+|\ln|k- k_{1}||), \label{II asymp chi at k1}
\end{align}
where $C$ is independent of $\zeta \in \mathcal{I}$.  
\end{enumerate}
\end{lemma}
\begin{proof}
Part $(\ref{deltaparta})$ follows from \eqref{delta1def} and an integration by parts. Part $(\ref{deltapartb})$ follows from \eqref{delta expression in terms of log and chi} and the fact that $\im \nu = 0$. For part $(\ref{deltapartc})$, we note from \eqref{def of chi} that, for $k \in \partial \D\setminus \Gamma_{5}^{(2)}$, 
\begin{align*}
\re \chi(\zeta,k) & = \frac{1}{2\pi} \int_{\Gamma_{5}^{(2)}}  \arg_{s}(k-s) d\ln(1+r_1(s)r_{2}(s)) \\
& = \frac{1}{2\pi} \int_{\Gamma_{5}^{(2)}} \frac{\pi + \arg_{i} k + \arg s}{2} d\ln(1+r_1(s)r_{2}(s)),
\end{align*}
and \eqref{re chi on part of the unit circle} now follows from the definition \eqref{nudef} of $\nu$ and the assumption that $r_1 = 0$ on $[0, i]$. By \eqref{delta expression in terms of log and chi}, 
\begin{align*}
|\delta(\zeta,k)| = e^{\nu \arg_{k_{1}}(k-k_{1})}e^{-\re \chi(\zeta,k)}.
\end{align*}
After substituting \eqref{re chi on part of the unit circle} and $\arg_{k_{1}}(k-k_{1}) = \frac{\pi + \arg_{i} k + \arg k_{1}}{2}$ into the above equation, we find \eqref{|delta| is constant on a part of the unit circle}. Part $(\ref{deltapartd})$ follows from \eqref{def of chi} and standard estimates.
\end{proof}
For $k \in \mathbb{C}\setminus \partial \D$, we define
\begin{align*}
\Delta_{33}(\zeta,k) = \frac{\delta(\zeta,\omega k)}{\delta(\zeta,\omega^{2} k)}\frac{\delta(\zeta,\frac{1}{\omega^{2} k})}{\delta(\zeta,\frac{1}{\omega k})}\mathcal{P}(\zeta,k), \;\; \Delta_{11}(\zeta,k)=\Delta_{33}(\zeta,\omega k), \;\; \Delta_{22}(\zeta,k)=\Delta_{33}(\zeta,\omega^{2} k),
\end{align*}
where we recall that $\mathcal{P}$ is defined in \eqref{def of mathcalP}. Note that $\mathcal{P}$ and $\Delta_{33}$ admit the symmetries 
\begin{align}\label{mathcalPDelta33symm}
\mathcal{P}(\zeta, k) = \mathcal{P}(\zeta, \tfrac{1}{k}) = \overline{\mathcal{P}(\zeta, \bar{k})}^{-1}, \qquad \Delta_{33}(\zeta,k)  = \Delta_{33}(\zeta,\tfrac{1}{k}) = \overline{\Delta_{33}(\zeta,\bar{k})}^{-1}.
\end{align}
The function $\Delta$ defined by
\begin{align}\label{Deltadef}
\Delta(\zeta,k) = \begin{pmatrix}
\Delta_{11}(\zeta,k) & 0 & 0 \\
0 & \Delta_{22}(\zeta,k) & 0 \\
0 & 0 & \Delta_{33}(\zeta,k)
\end{pmatrix}, \quad \zeta \in \mathcal{I}, \; k \in \mathbb{C}\setminus \partial \D,
\end{align}
obeys the symmetries
\begin{align}\label{Deltasymm}
\Delta(\zeta,k) = \mathcal{A}\Delta(\zeta,\omega k)\mathcal{A}^{-1} = \mathcal{B}\Delta(\zeta,\tfrac{1}{k})\mathcal{B}
= \mathcal{B}\overline{\Delta(\zeta,\bar{k})}^{-1}\mathcal{B},
\end{align}
and satisfies
\begin{subequations}\label{IIbis jumps Delta11}
\begin{align}
& \Delta_{+}^{-1}(\zeta,k) = \Delta_{-}^{-1}(\zeta,k) \begin{pmatrix}
1+r_{1}(k)r_{2}(k) & 0 & 0 \\
0 & \frac{1}{1+r_{1}(k)r_{2}(k)} & 0 \\
0 & 0 & 1
\end{pmatrix}, & & k \in \Gamma_{5}^{(2)}, \\
& \Delta_{+}^{-1}(\zeta,k) = \Delta_{-}^{-1}(\zeta,k), & & k \in \Gamma_{2}^{(2)}\cup \Gamma_{8}^{(2)}.
\end{align}
\end{subequations}
The next lemma follows from \eqref{Deltadef} and Lemma \ref{IIbis deltalemma}.

\begin{lemma}\label{Deltalemma}
For any $\epsilon>0$, $\Delta(\zeta,k)$ and $\Delta(\zeta,k)^{-1}$ are uniformly bounded for $k \in \mathbb{C}\setminus (\partial \D \cup \cup_{k_{0}\in \hat{\mathsf{Z}}}D_{\epsilon}(k_0))$ and $\zeta \in \mathcal{I}$. Furthermore, $\Delta(\zeta,k)=I+O(k^{-1})$ as $k \to \infty$. 
\end{lemma} 

 We define $n^{(3)}$ by
\begin{align}\label{IIbis def of mp3p}
n^{(3)}(x,t,k) = n^{(2)}(x,t,k)\Delta(\zeta,k), \qquad  k \in \mathbb{C}\setminus \Gamma^{(3)},
\end{align}
where $\Gamma^{(3)}=\Gamma^{(2)}$. The boundary values of $n^{(3)}$ obey $n_{+}^{(3)}=n_{-}^{(3)}v^{(3)}$ on $\Gamma^{(3)}$, where $v^{(3)}=\Delta_{-}^{-1}v^{(2)}\Delta_{+}$. Let $\mathcal{X}^{\epsilon}:= \cup_{j=1}^{4}\mathcal{X}_{j}^{\epsilon}$ be a small cross centered at $k_{1}$, where
\begin{align*}
&\mathcal{X}_{1}^{\epsilon} = \Gamma_{4}^{(3)}\cap D_{\epsilon}(k_{1}), & & \mathcal{X}_{2}^{\epsilon} = \Gamma_{7}^{(3)}\cap D_{\epsilon}(k_{1}), & & \mathcal{X}_{3}^{\epsilon} = \Gamma_{9}^{(3)}\cap D_{\epsilon}(k_{1}), & & \mathcal{X}_{4}^{\epsilon} = \Gamma_{6}^{(3)}\cap D_{\epsilon}(k_{1}),
\end{align*}
are oriented outwards from $k_1$. Let $\hat{\mathcal{X}}^\epsilon = \mathcal{X}^\epsilon \cup \omega \mathcal{X}^\epsilon \cup \omega^2 \mathcal{X}^\epsilon \cup (\mathcal{X}^\epsilon)^{-1} \cup (\omega \mathcal{X}^\epsilon)^{-1} \cup (\omega^2 \mathcal{X}^\epsilon)^{-1}$ be the union of six small crosses centered at $\{\omega^{j}k_{1},\omega^{j}k_{2}\}_{j=0}^{2}$.

\begin{lemma}\label{II v3lemma}
$v^{(3)}$ converges to the identity matrix $I$ as $t \to \infty$ uniformly for $\zeta \in \mathcal{I}$ and $k \in \Gamma^{(3)}\setminus \hat{\mathcal{X}}^\epsilon$. More precisely, for $\zeta \in \mathcal{I}$,
\begin{align}\label{II v3estimatesa}
& \|v^{(3)} - I\|_{(L^1 \cap L^\infty)(\Gamma^{(3)}\setminus \hat{\mathcal{X}}^\epsilon)} \leq Ct^{-1}.
\end{align}
\end{lemma}
\begin{proof}
Using \eqref{vp1p 123}, \eqref{vp1p 789}, \eqref{vp1p 101112}, and \eqref{II jumps vj diagonal p2p} together with the properties \eqref{IIbis jumps Delta11} of $\Delta$, we infer that
\begin{align}\label{lol3}
& v_j^{(3)} = \Delta_{-}^{-1}v_j^{(1)}\Delta_{+} = I + \Delta_{-}^{-1}v_{j,r}^{(1)}\Delta_{+}, \quad j = 2,5,8.
\end{align}
By Lemmas \ref{IIbis decompositionlemma} and \ref{Deltalemma}, the matrices $\Delta_{-}^{-1}v_{j,r}^{(1)}\Delta_{+}$, $j = 2,5,8$, are small as $ t \to \infty$.
Using also Lemma \ref{vsmallnearilemmaV}, we infer that the $L^1$ and $L^\infty$ norms of $v^{(3)}-I$ are $O(t^{-N})$ as $t \to \infty$, uniformly for $\zeta \in \mathcal{I}$ and $k \in \mathsf{S}\cap \Gamma^{(3)}\setminus \hat{\mathcal{X}}^\epsilon$. The claim \eqref{II v3estimatesa} now follows from \eqref{vjsymm}.
\end{proof}

The next lemma, which follows from (\ref{nresiduesk0})--(\ref{nresiduesk0real}) and the definition (\ref{IIbis def of mp3p}) of $n^{(3)}$, establishes the residue conditions of $n^{(3)}$ at the points in $\mathsf{Z}$.

\begin{lemma}[Residue conditions for $n^{(3)}$]\label{lemma:new residue}
At each point of $\hat{\mathsf{Z}}$, one entry of $n^{(3)}$ has a simple pole while two entries are analytic. Moreover, the following residue conditions hold at the points in $\mathsf{Z}$ (suppressing the $(x,t)$-dependence for conciseness):
\begin{enumerate}[$(a)$]
\item For each $k_{0}\in \mathsf{Z}\setminus\R$ with $\re  \Phi_{31}(\zeta,k_{0})<0$,
\begin{align*}
& \underset{k = k_0}{\res} n_{1}^{(3)}(k) = \frac{c_{k_{0}}^{-1}e^{\theta_{31}(k_{0})}n_{3}^{(3)}(k_{0})}{\Delta_{33}'(k_{0})(\Delta_{11}^{-1})'(k_{0})}, 
& & \underset{k = \bar{k}_0}{\res} n_{3}^{(3)}(k) = \frac{d_{k_{0}}^{-1}e^{-\theta_{32}(\bar{k}_{0})}n_{2}^{(3)}(\bar{k}_{0})}{\Delta_{22}'(\bar{k}_{0})(\Delta_{33}^{-1})'(\bar{k}_{0})}.
\end{align*}

\item For each $k_{0}\in \mathsf{Z}\cap\R$ with $\re  \Phi_{21}(\zeta,k_{0})<0$,
\begin{align}\nonumber
& \underset{k = k_0}{\res} n_{1}^{(3)}(k) = \frac{c_{k_0}^{-1} e^{\theta_{21}(k_0)}n_{2}^{(3)}(k_{0})}{\Delta_{22}'(k_{0})(\Delta_{11}^{-1})'(k_{0})}.
\end{align}

\item For each $k_{0}\in \mathsf{Z}\setminus\R$ with $\re  \Phi_{31}(\zeta,k_{0})\geq 0$, 
\begin{align}\nonumber
& \underset{k = k_0}{\res} n_3^{(3)}(k) = \frac{c_{k_{0}}e^{-\theta_{31}(k_0)} n_1^{(3)}(k_0)}{\Delta_{11}(k_{0})\Delta_{33}^{-1}(k_{0})}, & & \underset{k = \bar{k}_0}{\res} n_2^{(3)}(k) = \frac{d_{k_0} e^{\theta_{32}(\bar{k}_0)} n_3^{(3)}(\bar{k}_0)}{\Delta_{33}(\bar{k}_{0})\Delta_{22}^{-1}(\bar{k}_{0})}. 
\end{align}

\item For each $k_{0}\in \mathsf{Z}\cap\R$ with $\re  \Phi_{21}(\zeta,k_{0})\geq 0$,
\begin{align}\nonumber
& \underset{k = k_0}{\res} n_2^{(3)}(k) = \frac{c_{k_0} e^{-\theta_{21}(k_0)} n_1^{(3)}(k_0)}{\Delta_{11}(k_{0})\Delta_{22}^{-1}(k_{0})}.
\end{align}
\end{enumerate}
\end{lemma}

The residue conditions of $n^{(3)}$ at the points of $\hat{\mathsf{Z}} \setminus \mathsf{Z}$ can be obtained from those in Lemma \ref{lemma:new residue} and the $\mathcal{A}$- and $\mathcal{B}$-symmetries (\ref{nsymm}).

\section{The $n^{(3)}\to n^{(4)}$ transformation}\label{n3ton4sec}
The purpose of the transformation $n^{(3)}\to n^{(4)}$ is to remove the poles of $n^{(3)}$ at the points of $\hat{\mathsf{Z}}$. 
Let 
\begin{align}\label{hatMsoldef}
\hat{M}_{\sol}(x,t,k):=M_{\mathrm{sol}}(x,t,k) \mathsf{P}(\zeta,k),
\end{align}
where $M_{\mathrm{sol}}(x,t,k)$ is the unique solution of RH problem \ref{RHS} and 
\begin{align}\label{mathsfPdef}
\mathsf{P}(\zeta,k) := \diag (\mathcal{P}(\zeta,\omega k),\mathcal{P}(\zeta,\omega^{2}k),\mathcal{P}(\zeta,k)).
\end{align}

We define $n^{(4)}$ by
\begin{align}\label{n4def}
n^{(4)}(x,t,k) = n^{(3)}(x,t,k)\hat{M}_{\mathrm{sol}}(x,t,k)^{-1}, \qquad  k \in \mathbb{C}\setminus \Gamma^{(4)},
\end{align}
where $\Gamma^{(4)}=\Gamma^{(3)}$. 
We deduce from (\ref{mathcalPDelta33symm}) and (\ref{Msolsymm}) that
\begin{align}\label{Mhatsolsymm}
\hat{M}_{\mathrm{sol}}(x,t,k) = \mathcal{A} \hat{M}_{\mathrm{sol}}(x,t,\omega k)\mathcal{A}^{-1} = \mathcal{B} \hat{M}_{\mathrm{sol}}(x,t,k^{-1}) \mathcal{B},
\end{align}
and hence $n^{(4)}$ obeys the $\mathcal{A}$- and $\mathcal{B}$-symmetries in (\ref{njsymm}).
The jump matrix for $n^{(4)}$ is given by $v^{(4)} = \hat{M}_{\mathrm{sol}} n^{(4)} \hat{M}_{\mathrm{sol}}^{-1}$.

\begin{lemma}\label{hatMsolresiduelemma}
At each point of $\mathsf{Z}$, one column of $\hat{M}_{\mathrm{sol}}$ has a simple pole while two columns are analytic. Moreover, $\hat{M}_{\mathrm{sol}}(x,t,k)$ satisfies the residue conditions in Lemma \ref{lemma:new residue} with $n^{(3)}_{j}$ replaced by $[\hat{M}_{\mathrm{sol}}(x,t,k)]_{j}$ for $j = 1,2,3$.
\end{lemma}
\begin{proof}
We see from \eqref{hatDeltadef} and \eqref{Deltadef} that $\hat{\Delta}(\zeta,k) = \Delta(\zeta,k) \mathsf{P}(\zeta,k)^{-1}$ and clearly $\hat{\Delta}$ has no zeros or poles in $\hat{\mathsf{Z}}$. 
Let $k_{0}\in \mathsf{Z}\setminus \R$ be such that $\re  \Phi_{31}(\zeta,k_{0})<0$. 
Then $\mathsf{P}_{11}(\zeta,k)$ has a simple pole at $k_0$, $\mathsf{P}_{33}(\zeta,k)$ has a simple zero at $k_0$, and $\mathsf{P}_{22}(\zeta,k)$ has neither a zero nor a pole at $k_0$. 
Consequently, in view of part $(\ref{RHSitemd})$ of RH problem \ref{RHS}, $[\hat{M}_{\mathrm{sol}}]_j = \mathsf{P}_{jj}[M_{\mathrm{sol}}]_j$ has a simple pole at $k_0$ for $j = 1$, while it is analytic for $j = 2$ and $j = 3$. Moreover, using (\ref{Mk0notinR}) and omitting the $(x,t)$-dependence for brevity, we find
\begin{align*}
\underset{k = k_0}{\res} [\hat{M}_{\mathrm{sol}}(k)]_1 
& = \underset{k = k_0}{\res} \mathsf{P}_{11}(k) [M_{\mathrm{sol}}(k_0)]_1 
= \frac{1}{(\mathsf{P}_{11}^{-1})'(k_{0})} \frac{c_{k_0}^{-1} e^{\theta_{31}(k_0)}}{\hat{\Delta}_{11}^{-1}(k_0)\hat{\Delta}_{33}(k_0)} \underset{k = k_0}{\res}  [M_{\mathrm{sol}}(k)]_3 
	\\
& 
= \frac{c_{k_0}^{-1} e^{\theta_{31}(k_0)}}{(\mathsf{P}_{11}^{-1})'(k_{0})  \hat{\Delta}_{11}^{-1}(k_0) \hat{\Delta}_{33}(k_0)} \frac{[\hat{M}_{\mathrm{sol}}(k_0)]_3 }{\mathsf{P}_{33}'(k_0)}
= \frac{c_{k_0}^{-1} e^{\theta_{31}(k_0)} [\hat{M}_{\mathrm{sol}}(k_0)]_3}{(\Delta_{11}^{-1})'(k_0) \Delta_{33}'(k_0)}, 
\end{align*}
which proves the claim for  $k_{0}\in \mathsf{Z}\setminus \R$ with $\re  \Phi_{31}(\zeta,k_{0})<0$. 
The other points in $\mathsf{Z}$ can be treated in a similar way.
\end{proof}

It follows from Lemma \ref{hatMsolresiduelemma} and the symmetries (\ref{Mhatsolsymm}) that $(1,1,1) \hat{M}_{\mathrm{sol}}$ has the same pole structure and obeys the same residue conditions as $n^{(3)}$. The following lemma can be viewed as a consequence of this fact.

\begin{lemma}
  The function $n^{(4)}$ is analytic at every point of $\hat{\mathsf{Z}}$.
\end{lemma}
\begin{proof}
We give the details in the case of points $k_{0}\in \mathsf{Z}\setminus \R$ such that $\re  \Phi_{31}(\zeta,k_{0})\geq 0$; the other points in $\mathsf{Z}$ can be handled similarly and then the $\mathcal{A}$- and $\mathcal{B}$-symmetries yields analyticity also at the points in $\hat{\mathsf{Z}} \setminus \mathsf{Z}$. 

Let $k_{0}\in \mathsf{Z}\setminus \R$ be such that $\re  \Phi_{31}(\zeta,k_{0})\geq 0$. 
Since $\hat{M}_{\mathrm{sol}}$ has unit determinant, it follows from Lemma \ref{lemma:new residue} with $n^{(3)}_{j}$ replaced by $[\hat{M}_{\mathrm{sol}}(x,t,k)]_{j}$ that
$$\hat{M}_{\mathrm{sol}}^{-1}(k) =  \begin{pmatrix} - \frac{c_{k_0} e^{-\theta_{31}(k_0)}}{\Delta_{11}(k_0) \Delta_{33}^{-1}(k_0)(k - k_0)} V + O(1) \\ O(1) \\ V  + O(k-k_0)\end{pmatrix} \qquad \text{as $k \to k_0$},$$
where the row-vector $V = (V_1, V_2, V_3) \in \C^3$ is the (transpose of the) cross product of the two vectors $[\hat{M}_{\mathrm{sol}}(k_0)]_1$ and $[\hat{M}_{\mathrm{sol}}(k_0)]_2$, i.e., 
$$V_1 = (M_{\mathrm{sol}}(k_0))_{21}(M_{\mathrm{sol}}(k_0))_{32} - (M_{\mathrm{sol}}(k_0))_{22}(M_{\mathrm{sol}}(k_0))_{31},$$ 
etc. On the other hand, by Lemma \ref{lemma:new residue},
$$n^{(3)}(k) = \begin{pmatrix}
n_1^{(3)}(k_0) + O(k-k_0) & O(1) & \frac{c_{k_{0}}e^{-\theta_{31}(k_0)} n_1^{(3)}(k_0)}{\Delta_{11}(k_0) \Delta_{33}^{-1}(k_0)(k-k_0)}  + O(1) 
\end{pmatrix} \qquad \text{as $k \to k_0$}.
$$
Hence, as $k \to k_0$,
\begin{align*}
n^{(4)} = &\; n^{(3)} \hat{M}_{\mathrm{sol}}^{-1} 
=  \begin{pmatrix}
n_1^{(3)}(k_0) + O(k-k_0) & O(1) & \frac{c_{k_{0}}e^{-\theta_{31}(k_0)} }{\Delta_{11}(k_0) \Delta_{33}^{-1}(k_0)(k-k_0)} n_1^{(3)}(k_0) + O(1) 
\end{pmatrix}
	\\
&\times \begin{pmatrix} - \frac{c_{k_0} e^{-\theta_{31}(k_0)}}{\Delta_{11}(k_0)\Delta_{33}^{-1}(k_0)(k - k_0)} V + O(1) \\ O(1) \\ V  + O(k-k_0)\end{pmatrix} = O(1),
\end{align*}
showing that $n^{(4)}$ is analytic at $k_0$.
\end{proof}

\section{Local parametrix near $k_{1}$}\label{localparametrixsec}
The matrix $v^{(4)} - I$ is not uniformly small on $\hat{\mathcal{X}}^\epsilon$. In this section, we construct a local parametrix around $k_{1}$ which approximates $n^{(4)}$ near $k_1$. The local parametrices near the five other saddle points $\{k_{2},\omega k_{1},\omega k_{2},\omega^{2}k_{1},\omega^{2}k_{2}\}$ will then be constructed using the $\mathcal{A}$- and $\mathcal{B}$-symmetries. In Section \ref{n3tonhatsec}, we will prove that these parametrices approximate $n^{(4)}$ in $\cup_{j=0}^{2}\big(D_{\epsilon}(\omega^{j} k_{1})\cup D_{\epsilon}(\omega^{j} k_{2}) \big)$.

As $k \to k_{1}$, we have
\begin{align*}
& \Phi_{21}(\zeta,k)-\Phi_{21}(\zeta,k_{1}) =  \Phi_{21,k_{1}}(k-k_{1})^{2} +O((k-k_{1})^{3}), & & \Phi_{21,k_{1}}=\frac{4-3k_{1} \zeta - k_{1}^{3} \zeta}{4k_{1}^{4}}.
\end{align*}
We define $z=z(\zeta,t,k)$ by
\begin{align*}
z=z_{\star}\sqrt{t} (k-k_{1}) \hat{z}, \qquad \hat{z}= \sqrt{\frac{2i(\Phi_{21}(\zeta,k)-\Phi_{21}(\zeta,k_{1}))}{z_{\star}^{2}(k-k_{1})^{2}}},
\end{align*}
where the principal branch is chosen for $\hat{z}=\hat{z}(\zeta,k)$, and 
\begin{align*}
z_{\star} = \sqrt{2} e^{\frac{\pi i}{4}} \sqrt{\Phi_{21,k_{1}}}, \qquad -ik_{1}z_{\star}>0.
\end{align*}
Note that $\hat{z}(\zeta,k_{1})=1$, and $t(\Phi_{21}(\zeta,k)-\Phi_{21}(\zeta,k_{1})) = -\frac{iz^{2}}{2}$. Let $\epsilon>0$ be fixed and so small that the map $z$ is conformal from $D_{\epsilon}(k_{1})$ to a neighborhood of 0. As $k \to k_{1}$,
\begin{align*}
z = z_{\star}\sqrt{t}(k-k_{1})(1+O(k-k_{1})),
\end{align*}
and for all $k \in D_\epsilon(k_1)$, we have
\begin{align*}
\ln_{k_{1}}(k-k_{1}) = \ln_{0}[z_{\star}(k-k_{1})\hat{z}]- \ln \hat{z} -\ln z_{\star}
\end{align*}
where $\ln_{0}(k):= \ln|k|+i\arg_{0}k$, $\arg_{0}(k)\in (0,2\pi)$, and $\ln$ is the principal logarithm. Decreasing $\epsilon > 0$ if needed, $k \mapsto \ln \hat{z}$ is analytic in $D_{\epsilon}(k_{1})$, $\ln \hat{z} = O(k-k_{1})$ as $k \to k_{1}$, and
\begin{align*}
\ln z_{\star} = \ln |z_{\star}| + i \arg z_{\star} = \ln |z_{\star}| + i \big( \tfrac{\pi}{2}-\arg k_{1} \big),
\end{align*}
where $\arg k_{1} \in (\tfrac{\pi}{2},\tfrac{2\pi}{3})$. By \eqref{delta expression in terms of log and chi}, we have
\begin{align*}
& \delta(\zeta,k) = e^{-i\nu\ln_{k_{1}}(k-k_{1})}e^{-\chi(\zeta,k)} = e^{-i\nu \ln_{0} z} t^{\frac{i \nu}{2}} e^{i \nu \ln \hat{z}} e^{i \nu \ln z_{\star}} e^{-\chi(\zeta,k)}.
\end{align*}
From now on, for conciseness we will use the notation $z_{(0)}^{- i \nu}:=e^{-i\nu \ln_{0} z}$, $\hat{z}^{i\nu}:=e^{i \nu \ln \hat{z}}$ and $z_{\star}^{i \nu}:=e^{i \nu \ln z_{\star}}$, so that
\begin{align*}
\delta(\zeta,k) = z_{(0)}^{- i \nu} t^{\frac{i \nu}{2}} \hat{z}^{i\nu} z_{\star}^{i \nu} e^{-\chi(\zeta,k)}.
\end{align*}
Using this, we find
\begin{align}
& \frac{\Delta_{11}(\zeta,k)}{\Delta_{22}(\zeta,k)} = \frac{\delta(\frac{1}{k})^{2} \delta(\omega k)\delta(\omega^{2}k)}{\delta(k)^{2}\delta(\frac{1}{\omega^{2}k})\delta(\frac{1}{\omega k})} \frac{\mathcal{P}(\zeta,\omega k)}{\mathcal{P}(\zeta,\omega^{2} k)} = z_{(0)}^{2i \nu}d_{0}(\zeta,t)d_{1}(\zeta,k), \label{IIbis Delta33 Delta11} 
	\\ \label{d0def}
& d_{0}(\zeta,t) = e^{2\chi(\zeta,k_{1})}t^{-i \nu} z_{\star}^{-2i\nu} \frac{\delta(\frac{1}{k_{1}})^{2} \delta(\omega k_{1})\delta(\omega^{2}k_{1})}{\delta(\frac{1}{\omega^{2}k_{1}})\delta(\frac{1}{\omega k_{1}})}, 
	 \\
& d_{1}(\zeta,k) = e^{2(\chi(\zeta,k)-\chi(\zeta,k_{1}))} \hat{z}^{-2i\nu} \frac{\delta(\frac{1}{k})^{2} \delta(\omega k)\delta(\omega^{2}k)}{\delta(\frac{1}{\omega^{2}k})\delta(\frac{1}{\omega k})} \bigg( \frac{\delta(\frac{1}{k_{1}})^{2} \delta(\omega k_{1})\delta(\omega^{2}k_{1})}{\delta(\frac{1}{\omega^{2}k_{1}})\delta(\frac{1}{\omega k_{1}})} \bigg)^{-1} \frac{\mathcal{P}(\zeta,\omega k)}{\mathcal{P}(\zeta,\omega^{2} k)}. \nonumber
\end{align}
Define
\begin{align*}
Y(\zeta,t) =  d_{0}(\zeta,t)^{-\frac{\tilde{\sigma}_{3}}{2}}\tilde{r}(k_{1})^{-\frac{1}{4}\tilde{\sigma}_{3}}e^{-\frac{t}{2}\Phi_{21}(\zeta,k_{1})\tilde{\sigma}_{3}} \big( \tfrac{\mathcal{P}(\zeta,\omega k_{1})}{\mathcal{P}(\zeta,\omega^{2} k_{1})} \big)^{-\frac{1}{2}\tilde{\sigma}_{3}},
\end{align*}
where $\tilde{\sigma}_{3} = \diag (1,-1,0)$, and note that $\tilde{r}(k_{1})^{\pm\frac{1}{4}}>0$ (the choices of the branches of $d_{0}(\zeta,t)^{\frac{1}{2}}$ and $\big( \tfrac{\mathcal{P}(\zeta,\omega k_{1})}{\mathcal{P}(\zeta,\omega^{2} k_{1})} \big)^{\frac{1}{2}}$ are unimportant for the analysis below). For $k \in D_\epsilon(k_1)$, we define
\begin{align*}
\tilde{n}(x,t,k) = n^{(4)}(x,t,k) \hat{M}_{\mathrm{sol}}(x,t,k) Y(\zeta,t).
\end{align*}
Let $\tilde{v}$ be the jump matrix of $\tilde{n}$, and let $\tilde{v}_{j}$ be the restriction of $\tilde{v}$ to $\Gamma_{j}^{(4)}\cap D_{\epsilon}(k_{1})$. By \eqref{Vv2def}, \eqref{IIbis def of mp3p}, and \eqref{IIbis Delta33 Delta11}, we have
\begin{align*}
& \tilde{v}_{4} = \begin{pmatrix}
1 & \tilde{r}(k_{1})^{\frac{1}{2}}\hat{r}_{1,a}(k) \tfrac{\mathcal{P}(\zeta,\omega k_{1})}{\mathcal{P}(\zeta,\omega^{2} k_{1})} d_{1}^{-1}z_{(0)}^{-2i\nu}e^{\frac{iz^{2}}{2}} & 0 \\
0 & 1 & 0 \\
0 & 0 & 1
\end{pmatrix}, \\
& \tilde{v}_{7} = \begin{pmatrix}
1 & 0 & 0 \\
-\tilde{r}(k_{1})^{-\frac{1}{2}}r_{2,a}(k) \tfrac{\mathcal{P}(\zeta,\omega^{2} k_{1})}{\mathcal{P}(\zeta,\omega k_{1})} d_{1}z_{(0)}^{2i\nu} e^{-\frac{iz^{2}}{2}} & 1 & 0 \\
0 & 0 & 1
\end{pmatrix}, \\
& \tilde{v}_{9} = \begin{pmatrix}
1 & -\tilde{r}(k_{1})^{\frac{1}{2}}r_{1,a}(k) \tfrac{\mathcal{P}(\zeta,\omega k_{1})}{\mathcal{P}(\zeta,\omega^{2} k_{1})} d_{1}^{-1}z_{(0)}^{-2i\nu}e^{\frac{iz^{2}}{2}} & 0 \\
0 & 1 & 0 \\
0 & 0 & 1 
\end{pmatrix}, \\
& \tilde{v}_{6} = \begin{pmatrix}
1 & 0 & 0 \\
\tilde{r}(k_{1})^{-\frac{1}{2}}\hat{r}_{2,a}(k) \tfrac{\mathcal{P}(\zeta,\omega^{2} k_{1})}{\mathcal{P}(\zeta,\omega k_{1})} d_{1}z_{(0)}^{2i\nu}e^{-\frac{iz^{2}}{2}} & 1 & 0 \\
0 & 0 & 1
\end{pmatrix}.
\end{align*}
Let $q:=-\tilde{r}(k_{1})^{-\frac{1}{2}}r_{2}(k_{1})$. We have $-\tilde{r}(k_{1})^{\frac{1}{2}}r_{1}(k_{1})=\bar{q}$, $1+r_{1}(k_{1})r_{2}(k_{1}) = 1+|q|^{2}$ and $\nu = -\frac{1}{2\pi}\ln (1+|q|^{2})$. Note also that $d_{1}(\zeta,k_{1}) = \tfrac{\mathcal{P}(\zeta,\omega k_{1})}{\mathcal{P}(\zeta,\omega^{2} k_{1})}$.

The above discussion suggests that $\tilde{n} = n^{(4)} \hat{M}_{\mathrm{sol}} Y$ is well approximated for large $t$ by $(1,1,1) \hat{M}_{\mathrm{sol}} Y\tilde{m}^{k_{1}}$, where $\tilde{m}^{k_{1}}(x,t,k)$ is analytic for $k \in D_{\epsilon}(k_{1}) \setminus \mathcal{X}^{\epsilon}$, obeys the jump relation $\tilde{m}^{k_{1}}_+ = \tilde{m}^{k_{1}}_- \tilde{v}_{\mathcal{X}^{\epsilon}}^{X}$ on $\mathcal{X}^{\epsilon}$ where
\begin{align*}
& \tilde{v}_{\mathcal{X}_{1}^{\epsilon}}^{X} = \begin{pmatrix}
1 & \frac{\tilde{r}(k_{1})^{\frac{1}{2}}r_{1}(k_{1})}{1+r_{1}(k_{1})r_{2}(k_{1})}z_{(0)}^{-2i\nu}e^{\frac{iz^{2}}{2}} & 0 \\
0 & 1 & 0 \\
0 & 0 & 1
\end{pmatrix} = \begin{pmatrix}
1 & \frac{-\bar{q}}{1+|q|^{2}}z_{(0)}^{-2i\nu}e^{\frac{iz^{2}}{2}} & 0 \\
0 & 1 & 0 \\
0 & 0 & 1
\end{pmatrix}, \\
& \tilde{v}_{\mathcal{X}_{2}^{\epsilon}}^{X} = \begin{pmatrix}
1 & 0 & 0 \\
-\tilde{r}(k_{1})^{-\frac{1}{2}}r_{2}(k_{1}) z_{(0)}^{2i\nu} e^{-\frac{iz^{2}}{2}} & 1 & 0 \\
0 & 0 & 1
\end{pmatrix} = \begin{pmatrix}
1 & 0 & 0 \\
q z_{(0)}^{2i\nu} e^{-\frac{iz^{2}}{2}} & 1 & 0 \\
0 & 0 & 1
\end{pmatrix}, \\
& \tilde{v}_{\mathcal{X}_{3}^{\epsilon}}^{X} = \begin{pmatrix}
1 & -\tilde{r}(k_{1})^{\frac{1}{2}}r_{1}(k_{1}) z_{(0)}^{-2i\nu}e^{\frac{iz^{2}}{2}} & 0 \\
0 & 1 & 0 \\
0 & 0 & 1 
\end{pmatrix} = \begin{pmatrix}
1 & \bar{q} z_{(0)}^{-2i\nu}e^{\frac{iz^{2}}{2}} & 0 \\
0 & 1 & 0 \\
0 & 0 & 1 
\end{pmatrix}, \\
& \tilde{v}_{\mathcal{X}_{4}^{\epsilon}}^{X} = \begin{pmatrix}
1 & 0 & 0 \\
\frac{\tilde{r}(k_{1})^{-\frac{1}{2}}r_{2}(k_{1})}{1+r_{1}(k_{1})r_{2}(k_{1})} z_{(0)}^{2i\nu} e^{-\frac{iz^{2}}{2}} & 1 & 0 \\
0 & 0 & 1
\end{pmatrix} = \begin{pmatrix}
1 & 0 & 0 \\
\frac{-q}{1+|q|^{2}} z_{(0)}^{2i\nu} e^{-\frac{iz^{2}}{2}} & 1 & 0 \\
0 & 0 & 1
\end{pmatrix},
\end{align*}
and satisfies $\tilde{m}^{k_{1}}(x,t,k) \to I$ as $t\to\infty$ uniformly for $k\in \partial D_{\epsilon}(k_{1})$.
This motivates us to define $m^{k_{1}}$ for $k \in D_\epsilon(k_1)$ by
\begin{align}\label{mk0def}
m^{k_{1}}(x,t,k) = \hat{M}_{\mathrm{sol}}(x,t,k) Y(\zeta,t)m^{X}(q,z(\zeta,t,k))Y(\zeta,t)^{-1} \hat{M}_{\mathrm{sol}}(x,t,k)^{-1}, 
\end{align}
where $m^{X}$ is the solution to the model RH problem of Lemma \ref{IIbis Xlemma 3}.
We expect $n^{(4)}$ to be well approximated by $(1,1,1) m^{k_{1}}$  for large $t$; in Section \ref{n3tonhatsec} we will show that this is indeed the case.

\begin{lemma}\label{lemma: bound on Y}
The function $Y(\zeta,t)$ is uniformly bounded:
\begin{align}\label{Ybound}
\sup_{\zeta \in \mathcal{I}} \sup_{t \geq 2} | Y(\zeta,t)^{\pm 1}| \leq C.
\end{align}
Moreover, the functions $d_0(\zeta, t)$ and $d_1(\zeta, k)$ satisfy
\begin{align}
& |d_0(\zeta, t)| = e^{2\pi \nu}, & & \zeta \in \mathcal{I}, \ t \geq 2, \label{d0estimate} \\
& |d_1(\zeta, k) - \tfrac{\mathcal{P}(\zeta,\omega k_{1})}{\mathcal{P}(\zeta,\omega^{2} k_{1})}| \leq C |k - k_1| (1+ |\ln|k-k_1||), & & \zeta \in \mathcal{I}, \ k \in \mathcal{X}^{\epsilon}.\label{d1estimate}
\end{align}
\end{lemma}
\begin{proof}
The estimates \eqref{Ybound} and \eqref{d1estimate} follow from Lemma \ref{IIbis deltalemma}. Let us prove \eqref{d0estimate}. Since $\nu \in \mathbb{R}$, by \eqref{d0def} we have
\begin{align}\label{lol2}
|d_{0}(\zeta,t)| = e^{2 \, \re\chi(\zeta,k_{1})} e^{2\nu \arg z_{\star}} \left|\frac{\delta(\frac{1}{k_{1}})^{2} \delta(\omega k_{1})\delta(\omega^{2}k_{1})}{\delta(\frac{1}{\omega^{2}k_{1}})\delta(\frac{1}{\omega k_{1}})} \right|,
\end{align}
and we recall that $\arg z_{\star} = \tfrac{\pi}{2}-\arg k_{1}$. Using \eqref{|delta| is constant on a part of the unit circle}, we find
\begin{align}\label{lol1}
\left|\frac{\delta(\frac{1}{k_{1}})^{2} \delta(\omega k_{1})\delta(\omega^{2}k_{1})}{\delta(\frac{1}{\omega^{2}k_{1}})\delta(\frac{1}{\omega k_{1}})} \right| = \exp \bigg( \nu \arg k_{1} - \frac{1}{2\pi} \int_{\Gamma_{5}^{(2)}} \arg s \; d\ln(1+r_1(s)r_{2}(s)) \bigg).
\end{align}
We obtain \eqref{d0estimate} after substituting \eqref{re chi on part of the unit circle} evaluated at $k_1$ and \eqref{lol1} into \eqref{lol2}. 
\end{proof}

We establish in Lemma \ref{lemma:Msol P is unif bounded} that for any small enough $\epsilon>0$, $\hat{M}_{\mathrm{sol}}(x,t,k)$ is uniformly bounded for $t \geq 2$, $\zeta \in \mathcal{I}$, and $k\in \C\setminus \mathcal{D}_{\sol}$, where $\mathcal{D}_{\sol}$ is defined in \eqref{def of Dsol}.
Using this fact, the proof of the following lemma is similar to \cite[Proof of Lemma 9.11]{CLmain}, so we omit it.

\begin{lemma}\label{k0lemma}
For each $t \geq 2$ and $\zeta \in \mathcal{I}$, $m^{k_1}(x,t,k)$ defined in (\ref{mk0def}) is analytic for $k \in D_\epsilon(k_1) \setminus \mathcal{X}^\epsilon$. Furthermore, $m^{k_1}(x,t,k)$ is uniformly bounded for $t \geq 2$, $\zeta \in \mathcal{I}$ and $k \in D_\epsilon(k_1) \setminus \mathcal{X}^\epsilon$.
On $\mathcal{X}^\epsilon$, $m^{k_1}$ obeys the jump condition $m_+^{k_1} =  m_-^{k_1} v^{k_1}$, where $v^{k_1}$ satisfies
\begin{align}\label{v3vk0estimate}
\begin{cases}
 \| v^{(4)} - v^{k_1} \|_{L^1(\mathcal{X}^\epsilon)} \leq C t^{-1} \ln t,
	\\
\| v^{(4)} - v^{k_1} \|_{L^\infty(\mathcal{X}^\epsilon)} \leq C t^{-1/2} \ln t,
\end{cases} \qquad \zeta \in \mathcal{I}, \ t \geq 2.
\end{align}
Furthermore, as $t \to \infty$,
\begin{align}\label{mmodmuestimate2}
& \| m^{k_1}(x,t,\cdot) - I \|_{L^\infty(\partial D_\epsilon(k_1))} = O(t^{-1/2}),
	\\ \label{mmodmuestimate1}
& m^{k_1} - I =  \frac{\hat{M}_{\mathrm{sol}} Y m_1^X Y^{-1} \hat{M}_{\mathrm{sol}}^{-1}}{z_{\star}\sqrt{t} (k-k_{1}) \hat{z}(\zeta,k)}  + O(t^{-1}),
\end{align}
uniformly for $\zeta \in \mathcal{I}$, and $m_1^X=m_1^X(q)$ is given by \eqref{mXasymptotics}.
\end{lemma}

\section{The $n^{(4)}\to \hat{n}$ transformation}\label{n3tonhatsec}

We use the symmetries
\begin{align*}
m^{k_1}(x,t,k) = \mathcal{A} m^{k_1}(x,t,\omega k)\mathcal{A}^{-1} = \mathcal{B} m^{k_1}(x,t,k^{-1}) \mathcal{B}
\end{align*}
to extend the domain of definition of $m^{k_1}$ from $D_\epsilon(k_1)$ to $\mathcal{D}$, where $\mathcal{D} := D_\epsilon(k_1) \cup \omega D_\epsilon(k_1) \cup \omega^2 D_\epsilon(k_1) \cup D_\epsilon(k_1^{-1}) \cup \omega D_\epsilon(k_1^{-1}) \cup \omega^2 D_\epsilon(k_1^{-1})$. 
We will show that
\begin{align}\label{Sector I final transfo}
\hat{n} :=
\begin{cases}
n^{(4)} (m^{k_1})^{-1}, & k \in \mathcal{D}, \\
n^{(4)}, & \text{elsewhere},
\end{cases}
\end{align}
is close to $(1,1,1)$ for large $t$ and $\zeta \in \mathcal{I}$. 
Let $\hat{\Gamma} = \Gamma^{(4)} \cup \partial \mathcal{D}$ be the contour displayed in Figure \ref{fig:Gammahat}, where each circle that is part of $\partial \mathcal{D}$ is oriented clockwise, and define
\begin{align}\label{def of vhat II}
\hat{v}= \begin{cases}
v^{(4)}, & k \in \hat{\Gamma} \setminus \bar{\mathcal{D}},
	\\
m^{k_1}, & k \in \partial \mathcal{D},
	\\
m_-^{k_1} v^{(4)}(m_+^{k_1})^{-1}, & k \in \hat{\Gamma} \cap \mathcal{D}.
\end{cases}
\end{align}
Let $\hat{\Gamma}_{\star}$ be the set of self-intersection points of $\hat{\Gamma}$. The matrix $\hat{n}$ satisfies the following RH problem: (a) $\hat{n}:\mathbb{C}\setminus \hat{\Gamma}\to \mathbb{C}^{1\times 3}$ is analytic, (b) $\hat{n}_{+} = \hat{n}_{-}\hat{v}$ for $k \in \hat{\Gamma}\setminus \hat{\Gamma}_{\star}$, (c) $\hat{n}(x,t,k) = (1,1,1)+O(k^{-1})$ as $k \to \infty$, and (d) $\hat{n}(x,t,k)=O(1)$ as $k\to k_{\star}\in \hat{\Gamma}_{\star}$.

\begin{figure}
\begin{center}
\begin{tikzpicture}[master]
\node at (0,0) {};
\draw[black,line width=0.5 mm] (0,0)--(30:7.5);
\draw[black,line width=0.5 mm,->-=0.25,->-=0.57,->-=0.71,->-=0.91] (0,0)--(90:7.5);
\draw[black,line width=0.5 mm] (0,0)--(150:7.5);
\draw[dashed,black,line width=0.15 mm] (0,0)--(60:7.5);
\draw[dashed,black,line width=0.15 mm] (0,0)--(120:7.5);

\draw[black,line width=0.5 mm] ([shift=(30:3*1.5cm)]0,0) arc (30:150:3*1.5cm);
\draw[black,arrows={-Triangle[length=0.27cm,width=0.18cm]}]
($(73:3*1.5)$) --  ++(-17:0.001);
\draw[black,arrows={-Triangle[length=0.21cm,width=0.14cm]}]
($(116.5:3*1.5)$) --  ++(116+90:0.001);

\draw[black,line width=0.5 mm] ([shift=(30:3*1.5cm)]0,0) arc (30:150:3*1.5cm);

\draw[black,arrows={-Triangle[length=0.27cm,width=0.18cm]}]
($(97:3*1.5)$) --  ++(97-90:0.001);

\draw[black,line width=0.5 mm] (150:3.65)--($(129.688:3*1.5)+(129.688+135:0.5)$)--(129.688:3*1.5)--($(129.688:3*1.5)+(129.688+45:0.5)$)--(150:5.8);
\draw[black,line width=0.5 mm,-<-=0.13,->-=0.78] (90:3.65)--($(110.3:3*1.5)+(110.3-135:0.5)$)--(110.3:3*1.5)--($(110.3:3*1.5)+(110.3-45:0.5)$)--(90:5.8);
\draw[black,line width=0.5 mm,-<-=0.08,->-=0.81] (120:3.65)--($(110.3:3*1.5)+(110.3+135:0.5)$)--(110.3:3*1.5)--($(110.3:3*1.5)+(110.3+45:0.5)$)--(120:5.8);
\draw[black,line width=0.5 mm] (120:3.65)--($(129.688:3*1.5)+(129.688-135:0.5)$)--(129.688:3*1.5)--($(129.688:3*1.5)+(129.688-45:0.5)$)--(120:5.8);

\draw[black,line width=0.5 mm,->-=0.6] (60:4.5)--(60:5.8);
\draw[black,line width=0.5 mm,->-=0.75] (60:3.65)--(60:4.5);
\draw[black,line width=0.5 mm,->-=0.6] (120:4.5)--(120:5.8);
\draw[black,line width=0.5 mm,->-=0.75] (120:3.65)--(120:4.5);

\draw[black,line width=0.5 mm,-<-=0.70] ([shift=(30:3.65cm)]0,0) arc (30:90:3.65cm);
\draw[black,line width=0.5 mm,-<-=0.71] ([shift=(30:5.8cm)]0,0) arc (30:90:5.8cm);

\draw[black,line width=0.5 mm] (129.688:3.65)--(129.688:5.8);
\draw[black,line width=0.5 mm,-<-=0.17,->-=0.84] (110.3:3.65)--(110.3:5.8);

\draw[black,line width=0.5 mm] ([shift=(30:3.65cm)]0,0) arc (30:150:3.65cm);
\draw[black,line width=0.5 mm] ([shift=(30:5.8cm)]0,0) arc (30:150:5.8cm);

\draw[green,fill] (129.688:3*1.5) circle (0.12cm);
\draw[blue,fill] (110.3:3*1.5) circle (0.12cm);

\draw[black,line width=0.5 mm] (129.688:3*1.5) circle (0.5cm);
\draw[black,line width=0.5 mm] (110.3:3*1.5) circle (0.5cm);
\end{tikzpicture}
\end{center}
\begin{figuretext}
\label{fig:Gammahat}
The contour $\hat{\Gamma}=\Gamma^{(4)}\cup \partial \mathcal{D}$ for $\arg k \in [\frac{\pi}{6},\frac{5\pi}{6}]$ (solid), the boundary of $\mathsf{S}$ (dashed) and, from right to left, the saddle points $k_{1}$ (blue) and $\omega^{2}k_{2}$ (green).
\end{figuretext}
\end{figure}

\begin{lemma}\label{whatlemma}
Let $\hat{w} = \hat{v}-I$. The following estimates hold uniformly for $t \geq 2$ and $\zeta \in \mathcal{I}$:
\begin{subequations}\label{hatwestimate}
\begin{align}\label{hatwestimate1}
& \| \hat{w}\|_{(L^1\cap L^\infty)(\hat\Gamma \setminus (\partial \mathcal{D} \cup \hat{\mathcal{X}}^\epsilon))} \leq C t^{-1},
	\\\label{hatwestimate3}
& \| \hat{w}\|_{(L^1\cap L^{\infty})(\partial \mathcal{D})} \leq C t^{-1/2},	\\\label{hatwestimate4}
& \| \hat{w}\|_{L^1(\hat{\mathcal{X}}^\epsilon)} \leq C t^{-1}\ln t,
	\\\label{hatwestimate5}
& \| \hat{w}\|_{L^\infty(\hat{\mathcal{X}}^\epsilon)} \leq C t^{-1/2}\ln t.
\end{align}
\end{subequations}
\end{lemma}
\begin{proof}
Note that $\hat{\Gamma} \setminus (\partial \mathcal{D} \cup \hat{\mathcal{X}}^\epsilon) = \Gamma^{(4)}\setminus \hat{\mathcal{X}}^\epsilon$. Hence \eqref{hatwestimate1} follows from \eqref{def of vhat II}, Lemma \ref{II v3lemma}, and the uniform boundedness of $m^{k_{1}}$. The inequality \eqref{hatwestimate3} is a consequence of \eqref{def of vhat II} and \eqref{mmodmuestimate2}. The estimates \eqref{hatwestimate4} and \eqref{hatwestimate5} follow from \eqref{def of vhat II}, \eqref{v3vk0estimate}, and the uniform boundedness of $m^{k_1}$.
\end{proof}
We define the Cauchy transform $\hat{\mathcal{C}}h$ of a function $h$ defined on $\hat{\Gamma}$ by
\begin{align}\label{Cauchyoperatordef}
(\hat{\mathcal{C}}h)(z) = \frac{1}{2\pi i} \int_{\hat{\Gamma}} \frac{h(z')dz'}{z' - z}, \qquad z \in \C \setminus \hat{\Gamma}.
\end{align}
The estimates in Lemma \ref{whatlemma} show that
\begin{align}\label{hatwLinfty}
\begin{cases}
\|\hat{w}\|_{L^1(\hat{\Gamma})}\leq C t^{-1/2},
	\\
\|\hat{w}\|_{L^\infty(\hat{\Gamma})}\leq C t^{-1/2}\ln t,
\end{cases}	 \qquad t \geq 2, \ \zeta \in \mathcal{I},
\end{align}
and hence, employing the inequality $\| f \|_{L^p} \leq \| f \|_{L^1}^{1/p}\|f \|_{L^{\infty}}^{(p-1)/p}$, we get
\begin{align}\label{Lp norm of what}
& \|\hat{w}\|_{L^p(\hat{\Gamma})} 
\leq C t^{-\frac{1}{2}} (\ln t)^{\frac{p-1}{p}},  \qquad t \geq 2, \ \zeta \in \mathcal{I},
\end{align}
for each $1 \leq p \leq \infty$. From \eqref{Lp norm of what}, we infer that
$\hat{w} \in L^{2}(\hat{\Gamma}) \cap L^{\infty}(\hat{\Gamma})$, and thus the operator $\hat{\mathcal{C}}_{\hat{w}}=\hat{\mathcal{C}}_{\hat{w}(x,t,\cdot)}: L^{2}(\hat{\Gamma})+L^{\infty}(\hat{\Gamma}) \to L^{2}(\hat{\Gamma})$, $h \mapsto \hat{\mathcal{C}}_{\hat{w}}h := \hat{\mathcal{C}}_{-}(h \hat{w})$ is well-defined. Let $\mathcal{B}(L^{2}(\hat{\Gamma}))$ be the space of bounded linear operators on $L^{2}(\hat{\Gamma})$. We deduce from \eqref{Lp norm of what} that there exists $T > 0$ such that $I - \hat{\mathcal{C}}_{\hat{w}(x, t, \cdot)} \in \mathcal{B}(L^{2}(\hat{\Gamma}))$ is invertible for all $t \geq T$ and $\zeta \in \mathcal{I}$. Hence $\hat{n}$ is the solution of a small-norm RH problem, and we have
\begin{align}\label{IIbis hatmrepresentation}
\hat{n}(x, t, k) = (1,1,1) + \hat{\mathcal{C}}(\hat{\mu}\hat{w}) = (1,1,1) + \frac{1}{2\pi i}\int_{\hat{\Gamma}} \hat{\mu}(x, t, s) \hat{w}(x, t, s) \frac{ds}{s - k}
\end{align}
for $t \geq T$ and $\zeta \in \mathcal{I}$, where 
\begin{align}\label{IIbis hatmudef}
\hat{\mu} = (1,1,1) + (I - \hat{\mathcal{C}}_{\hat{w}})^{-1}\hat{\mathcal{C}}_{\hat{w}}(1,1,1) \in (1,1,1) + L^{2}(\hat{\Gamma}).
\end{align}
Furthermore, using \eqref{Lp norm of what}, \eqref{IIbis hatmudef}, and the fact that $\|\hat{\mathcal{C}}_{-}\|_{\mathcal{B}(L^{p}(\hat{\Gamma}))}< \infty$ for $p \in (1,\infty)$, we find that for any $p \in (1,\infty)$ there exists $C_{p}>0$ such that
\begin{align}\label{estimate on mu}
& \|\hat{\mu} - (1,1,1)\|_{L^p(\hat{\Gamma})} \leq  C_{p} t^{-\frac{1}{2}}(\ln t)^{\frac{p-1}{p}}
\end{align}
holds for all $\zeta \in \mathcal{I}$ and large enough $t$. 

Let us now obtain the long-time asymptotics of $\hat{n}$. Define $\hat{n}^{(1)}$ as the nontangential limit
\begin{align*}
& \hat{n}^{(1)}(x,t):=\ntlim_{k\to \infty} k(\hat{n}(x,t,k) - (1,1,1))
= - \frac{1}{2\pi i}\int_{\hat{\Gamma}} \hat{\mu}(x,t,k) \hat{w}(x,t,k) dk.
\end{align*}
\begin{lemma}\label{lemma: hatmplp asymp sector I}
As $t \to \infty$, 
\begin{align}\label{limlhatm}
& \hat{n}^{(1)}(x,t) = -\frac{(1,1,1)}{2\pi i}\int_{\partial \mathcal{D}} \hat{w}(x,t,k) dk + O(t^{-1}\ln t).
\end{align}
\end{lemma}
\begin{proof}
Since
\begin{align*}
\hat{n}^{(1)}(x,t) = & -\frac{(1,1,1)}{2\pi i}\int_{\partial \mathcal{D}} \hat{w}(x,t,k) dk  -\frac{(1,1,1)}{2\pi i}\int_{\hat{\Gamma}\setminus\partial \mathcal{D}} \hat{w}(x,t,k) dk
	\\
& -\frac{1}{2\pi i}\int_{\hat{\Gamma}} (\hat{\mu}(x,t,k)-(1,1,1))\hat{w}(x,t,k) dk,
\end{align*}
the statement follows from \eqref{estimate on mu} and Lemma \ref{whatlemma}.
\end{proof}
We define $\{F^{(l)}\}_{l \in \mathbb{Z}}$ by
\begin{align*}
F^{(l)}(\zeta,t) = & - \frac{1}{2\pi i} \int_{\partial D_\epsilon(k_1)} k^{l-1}\hat{w}(x,t,k) dk	
= - \frac{1}{2\pi i}\int_{\partial D_\epsilon(k_1)}k^{l-1} (m^{k_1} - I) dk.
\end{align*}
Using \eqref{mmodmuestimate1} and $\hat{z}(\zeta,k_{1})=1$, and the fact that $\partial D_\epsilon(k_{1})$ is oriented clockwise, we obtain
\begin{align}
F^{(l)}(\zeta, t) &  =  k_{1}^{l-1}\frac{\hat{M}_{\mathrm{sol}}(x,t,k_{1}) Y(\zeta,t) m_1^{X} Y(\zeta,t)^{-1} \hat{M}_{\mathrm{sol}}(x,t,k_{1})^{-1}}{z_{\star}\sqrt{t}} + O( t^{-1}) \nonumber \\
& =  -ik_{1}^{l}Z(\zeta,t) + O( t^{-1}) \label{asymptotics for Fl}
\end{align}
as $t \to \infty$ uniformly for $\zeta \in \mathcal{I}$, where
\begin{align*}
Z(\zeta,t) & = \frac{\hat{M}_{\mathrm{sol}}(x,t,k_{1}) Y(\zeta,t) m_1^{X} Y(\zeta,t)^{-1} \hat{M}_{\mathrm{sol}}(x,t,k_{1})^{-1}}{-ik_{1}z_{\star}\sqrt{t}} 
	 \\
& = \frac{\hat{M}_{\mathrm{sol}}(x,t,k_{1})}{-ik_{1}z_{\star}\sqrt{t}} \begin{pmatrix}
0 & \hspace{-0.15cm}\frac{\beta_{12}\tilde{r}(k_{1})^{-\frac{1}{2}}}{d_{0} e^{t\Phi_{21}(\zeta,k_{1})} } \tfrac{\mathcal{P}(\zeta,\omega^{2} k_{1})}{\mathcal{P}(\zeta,\omega k_{1})} & \hspace{-0.15cm}0 \\
\frac{d_{0}  \tilde{r}(k_{1})^{\frac{1}{2}} \beta_{21}}{e^{-t \Phi_{21}(\zeta,k_{1})}} \tfrac{\mathcal{P}(\zeta,\omega k_{1})}{\mathcal{P}(\zeta,\omega^{2} k_{1})}& \hspace{-0.15cm}0 & \hspace{-0.15cm}0 \\
0 & \hspace{-0.15cm}0 & \hspace{-0.15cm}0
\end{pmatrix}\hat{M}_{\mathrm{sol}}(x,t,k_{1})^{-1}.
\end{align*}
Since $\hat{M}_{\sol} = M_{\mathrm{sol}} \mathsf{P}$ with $\mathsf{P}$ given by (\ref{mathsfPdef}), the expression for $Z$ simplifies to
\begin{align}\label{Zsimplified}
Z(\zeta,t) = \frac{M_{\mathrm{sol}}(x,t,k_{1})}{-ik_{1}z_{\star}\sqrt{t}} \begin{pmatrix}
0 & \hspace{-0.15cm}\frac{\beta_{12}\tilde{r}(k_{1})^{-\frac{1}{2}}}{d_{0} e^{t\Phi_{21}(\zeta,k_{1})} } & \hspace{-0.15cm}0 \\
\frac{d_{0}  \tilde{r}(k_{1})^{\frac{1}{2}} \beta_{21}}{e^{-t \Phi_{21}(\zeta,k_{1})}} & \hspace{-0.15cm}0 & \hspace{-0.15cm}0 \\
0 & \hspace{-0.15cm}0 & \hspace{-0.15cm}0
\end{pmatrix} M_{\mathrm{sol}}(x,t,k_{1})^{-1}.
\end{align}

\begin{lemma}\label{lemma: some integrals by symmetry}
For $l \in \mathbb{Z}$ and $j=0,1,2$, we have
\begin{align}
& -\frac{1}{2\pi i}\int_{\omega^{j} \partial D_\epsilon(k_1)} k^{l}\hat{w}(x,t,k)dk = \omega^{j(l+1)} \mathcal{A}^{-j}F^{(l+1)}(\zeta,t)\mathcal{A}^{j}, \label{int1} \\
& -\frac{1}{2\pi i}\int_{\omega^{j} \partial D_\epsilon(k_1^{-1})} k^{l}\hat{w}(x,t,k)dk = -\omega^{j(l+1)}\mathcal{A}^{-j}\mathcal{B} F^{(-l-1)}(\zeta,t) \mathcal{B}\mathcal{A}^{j}. \label{int1 tilde}
\end{align}
\end{lemma}
\begin{proof}
The assertions are a consequence of the symmetries $\hat{w}(x, t, k) = \mathcal{A} \hat{w}(x, t, \omega k) \mathcal{A}^{-1} = \mathcal{B} \hat{w}(x, t, k^{-1}) \mathcal{B}$ which hold for $k \in \partial \mathcal{D}$. For further details we refer to \cite[Proof of Lemma 12.4]{CLsectorIV}.
\end{proof}

Using Lemma \ref{lemma: some integrals by symmetry}, we find
\begin{align*}
&  \frac{-1}{2\pi i}\int_{\partial \mathcal{D}} \hat{w}(x,t,k) dk =  \sum_{j=0}^{2}  \frac{-1}{2\pi i}\int_{\omega^{j} \partial D_\epsilon(k_1)} \hat{w}(x,t,k) dk - \sum_{j=0}^{2}  \frac{1}{2\pi i}\int_{\omega^{j} \partial D_\epsilon(k_1^{-1})} \hat{w}(x,t,k) dk \\
& = \sum_{j=0}^{2} \omega^{j} \mathcal{A}^{-j}F^{(1)}(\zeta,t)\mathcal{A}^{j} - \sum_{j=0}^{2} \omega^{j}\mathcal{A}^{-j}\mathcal{B} F^{(-1)}(\zeta,t) \mathcal{B}\mathcal{A}^{j}.
\end{align*}
Therefore, \eqref{limlhatm} and \eqref{asymptotics for Fl} imply that $\hat{n}^{(1)}=(1,1,1)\hat{m}^{(1)}$, where
\begin{align}\nonumber
\hat{m}^{(1)}(x,t) & =  \sum_{j=0}^{2} \omega^{j} \mathcal{A}^{-j}F^{(1)}(\zeta,t)\mathcal{A}^{j} - \sum_{j=0}^{2} \omega^{j}\mathcal{A}^{-j}\mathcal{B} F^{(-1)}(\zeta,t) \mathcal{B}\mathcal{A}^{j} + O(t^{-1}\ln t) \\
& = -ik_{1}\sum_{j=0}^{2} \omega^{j} \mathcal{A}^{-j}Z(\zeta,t)\mathcal{A}^{j} + ik_{1}^{-1} \sum_{j=0}^{2} \omega^{j}\mathcal{A}^{-j}\mathcal{B} Z(\zeta,t) \mathcal{B}\mathcal{A}^{j} + O(t^{-1}\ln t) \label{mhatplp asymptotics}
\end{align}
as $t \to \infty$, uniformly for $\zeta \in \mathcal{I}$. 

\section{Asymptotics of $u$}\label{usec}

\subsection{Proof of Theorem \ref{asymptoticsth}}\label{proofth1subsec}
By \eqref{recoveruvn}, we have $u(x,t) = -i\sqrt{3}\frac{\partial}{\partial x}n_{3}^{(1)}(x,t)$, where $n_{3}(x,t,k) = 1+n_{3}^{(1)}(x,t)k^{-1}+O(k^{-2})$ as $k \to \infty$. By inverting the transformations $n \to n^{(1)}\to n^{(2)}\to n^{(3)} \to n^{(4)}\to \hat{n}$ using \eqref{Sector IIbis first transfo}, \eqref{Sector IIbis second transfo}, \eqref{IIbis def of mp3p}, \eqref{n4def}, and \eqref{Sector I final transfo}, we get
\begin{align*}
n = \hat{n}M_{\mathrm{sol}} \mathsf{P}\Delta^{-1}(G^{(2)})^{-1}(G^{(1)})^{-1}, \qquad k \in \mathbb{C}\setminus (\hat{\Gamma}\cup\bar{\mathcal{D}}),
\end{align*}
where $G^{(1)}$, $G^{(2)}$, $\Delta$ are defined in \eqref{IIbis Gp1pdef}, \eqref{IIbis Gp2pdef}, and \eqref{Deltadef}, respectively. The functions $\Delta(\zeta,k)$ and $\mathsf{P}(\zeta,k)$ are in general not smooth as functions of $\zeta$, because the number of factors appearing in the products in the definition \eqref{def of mathcalP} of $\mathcal{P}$ may vary discontinuously with $\zeta$; note however that these products disappear when considering $\hat{\Delta}(\zeta,k) =\Delta(\zeta,k) \mathsf{P}(\zeta,k)^{-1}$, so that $\hat{\Delta}(\zeta,k)$ is smooth as a function of $\zeta$. Let $n_{\mathrm{sol}}(x,t,k):=(1,1,1)M_{\mathrm{sol}}(x,t,k)$. Since $\partial_x = t^{-1}\partial_\zeta$, we find
\begin{align}\nonumber
u(x,t) & = -i\sqrt{3}\frac{\partial}{\partial x}\bigg( \hat{n}_{3}^{(1)}(x,t) + \big(n_{\mathrm{sol}}^{(1)}(x,t)\big)_3 + \ntlim_{k\to \infty} k (\hat{\Delta}_{33}(\zeta,k)^{-1}-1) \bigg) 
	\\
&= u_{\mathrm{sol}}(x,t) -i\sqrt{3}\frac{\partial}{\partial x} \hat{n}_{3}^{(1)}(x,t)  + O(t^{-1}), \qquad t \to \infty, \label{IIbis recoverun}
\end{align}
where $\hat{n}_{3}^{(1)}$ and $n_{\mathrm{sol}}^{(1)}$ are defined through the expansions
\begin{align*}
\hat{n}_{3}(x,t,k) & = 1+\hat{n}_{3}^{(1)}(x,t)k^{-1}+O(k^{-2}), \\
n_{\mathrm{sol}}(x,t,k) & = (1,1,1) + n_{\mathrm{sol}}^{(1)}(x,t)k^{-1}+O(k^{-2})
\end{align*}
as $k \to \infty$, and $u_{\mathrm{sol}}(x,t):=-i\sqrt{3}\frac{\partial}{\partial x} (n_{\mathrm{sol}}^{(1)}(x,t))_{3}$. 
By \eqref{mhatplp asymptotics}, we have
\begin{align*}
\hat{n}_{3}^{(1)} & = -ik_{1}\sum_{j=0}^{2} \omega^{j} (1,1,1)Z(\zeta,t)\mathcal{A}^{j}[I]_{3} + ik_{1}^{-1} \sum_{j=0}^{2} \omega^{j}(1,1,1) Z(\zeta,t) \mathcal{B}\mathcal{A}^{j}[I]_{3} + O(t^{-1}\ln t) \\
& = -ik_{1} (1,1,1)Z(\zeta,t) \begin{pmatrix}
\omega \\ \omega^{2} \\ 1
\end{pmatrix} + ik_{1}^{-1} (1,1,1) Z(\zeta,t) \begin{pmatrix}
\omega^{2} \\ \omega \\ 1
\end{pmatrix} + O(t^{-1}\ln t) \\
& =  (1,1,1)Z(\zeta,t)\begin{pmatrix}
\omega^{2} i k_{1}^{-1} - \omega i k_{1}  \\ \omega ik_{1}^{-1} - \omega^{2}ik_{1} \\ ik_{1}^{-1}-ik_{1}
\end{pmatrix} + O(t^{-1}\ln t).
\end{align*}
Combining the above formula with \eqref{IIbis recoverun}, we get
\begin{align}\label{uusolpartialZ}
u(x,t)  &= u_{\mathrm{sol}}(x,t) -i\sqrt{3}\frac{\partial}{\partial x} (1,1,1)Z(\zeta,t)\begin{pmatrix}
\omega^{2} i k_{1}^{-1} - \omega i k_{1}  \\ \omega ik_{1}^{-1} - \omega^{2}ik_{1} \\ ik_{1}^{-1}-ik_{1}
\end{pmatrix} + O(t^{-1}\ln t).
\end{align}
Substituting in the expression (\ref{Zsimplified}) for $Z$, we arrive at
\begin{align*}\nonumber
u(x,t)  = &\; u_{\mathrm{sol}}(x,t) -i\sqrt{3}\frac{n_{\mathrm{sol}}(x,t,k_{1})}{-ik_{1}z_{\star}\sqrt{t}} \begin{pmatrix}
0 & \hspace{-0.15cm}\frac{\beta_{12}\tilde{r}(k_{1})^{-\frac{1}{2}}\frac{d}{d\zeta}[\Phi_{21}(\zeta,k_{1}(\zeta))]}{-d_{0} e^{t\Phi_{21}(\zeta,k_{1})} } & \hspace{-0.15cm}0 \\
\frac{d_{0}  \tilde{r}(k_{1})^{\frac{1}{2}} \beta_{21}\frac{d}{d\zeta}[\Phi_{21}(\zeta,k_{1}(\zeta))]}{e^{-t \Phi_{21}(\zeta,k_{1})}} & \hspace{-0.15cm}0 & \hspace{-0.15cm}0 \\
0 & \hspace{-0.15cm}0 & \hspace{-0.15cm}0
\end{pmatrix}
	\\
& \times M_{\mathrm{sol}}(x,t,k_{1})^{-1}\begin{pmatrix}
\omega^{2} i k_{1}^{-1} - \omega i k_{1}  \\ \omega ik_{1}^{-1} - \omega^{2}ik_{1} \\ ik_{1}^{-1}-ik_{1}
\end{pmatrix} + O(t^{-1}\ln t).
\end{align*}
A long but straightforward computation using \eqref{d0def}, \eqref{delta expression in terms of log and chi}, and \eqref{def of chi} shows that $\arg d_{0}(\zeta,t)$ can be rewritten as in \eqref{argd0}.
Using (\ref{betaXdef}) as well as the relations $|d_0(\zeta, t)| = e^{2\pi \nu}$, $q=-\tilde{r}(k_{1})^{-\frac{1}{2}}r_{2}(k_{1})$, $\bar{q}=-\tilde{r}(k_{1})^{\frac{1}{2}}r_{1}(k_{1})$, and
\begin{align}\label{ddzeetaPhi21}
\frac{d}{d\zeta}[\Phi_{21}(\zeta,k_{1}(\zeta))] = i \im k_1, 
\end{align}
the asymptotic formula (\ref{uasymptotics}) follows. This completes the proof of Theorem \ref{asymptoticsth}.

\subsection{Proof of Theorem \ref{asymptoticsth2}}\label{proofth2subsec}
Let us next derive the asymptotic formula of Theorem \ref{asymptoticsth2}. Fix $\epsilon > 0$ and let $S_{\epsilon} \subset \mathcal{I}$ be the set defined in (\ref{Sdeltadef}). By Lemma \ref{hatMsolresiduelemma}, $\hat{M}_{\mathrm{sol}}(x,t,k)$ satisfies the residue conditions in Lemma \ref{lemma:new residue} with $n^{(3)}_{j}$ replaced by $[\hat{M}_{\mathrm{sol}}(x,t,k)]_{j}$. For $\zeta \in \mathcal{I} \setminus S_{\epsilon}$, each residue condition in Lemma \ref{lemma:new residue} has a coefficient that is uniformly exponentially small. It follows that $\hat{M}_{\mathrm{sol}}(x,t,k_1) = I + O(t^{-N})$ and $u_{\mathrm{sol}}(x,t) = O(t^{-N})$ for every $N \geq 1$ as $t \to \infty$, uniformly for $\zeta \in \mathcal{I} \setminus S_{\epsilon}$.
Thus, for $\zeta \in \mathcal{I} \setminus S_{\epsilon}$, (\ref{uusolpartialZ}) simplifies to 
\begin{align*}
u(x,t)  &=  -i\sqrt{3}\frac{\partial}{\partial x}\hat{n}_{3}^{(1)}  + O(t^{-1}\ln t),
\end{align*}
where
\begin{align*}
 \hat{n}_{3}^{(1)} =&\; Z_{12} \big(\omega i k_{1}^{-1} - \omega^{2} i k_{1} \big) + Z_{21} \big( \omega^{2} i k_{1}^{-1} - \omega i k_{1} \big) + O(t^{-1}\ln t)
	 \\ 
= &\: \frac{\beta_{12}(\omega i k_{1}^{-1}-\omega^{2} i k_{1})}{-i k_{1} z_{\star} \sqrt{t} d_{0}e^{t \Phi_{21}(\zeta,k_{1})}\tilde{r}(k_{1})^{\frac{1}{2}}}\frac{\mathcal{P}(\zeta,\omega^{2} k_{1})}{\mathcal{P}(\zeta,\omega k_{1})} 
	\\
& + \frac{d_{0}e^{t\Phi_{21}(\zeta,k_{1})}\tilde{r}(k_{1})^{\frac{1}{2}}\beta_{21}(\omega^{2} i k_{1}^{-1}-\omega i k_{1})}{-ik_{1}z_{\star} \sqrt{t}} \frac{\mathcal{P}(\zeta,\omega k_{1})}{\mathcal{P}(\zeta,\omega^{2} k_{1})}  + O(t^{-1}\ln t)
\end{align*}
as $t \to \infty$. Since $|d_0(\zeta, t)| = e^{2\pi \nu}$, we have $\frac{\beta_{12}}{d_{0}} \frac{\mathcal{P}(\zeta,\omega^{2} k_{1})}{\mathcal{P}(\zeta,\omega k_{1})} = -\bar{d_{0}} \bar{\beta}_{21} \frac{\overline{\mathcal{P}(\zeta,\omega k_{1})}}{\overline{\mathcal{P}(\zeta,\omega^{2} k_{1})}}$. Using also that $-ik_{1}z_{\star}>0$ and $(\omega i k_{1}^{-1}-\omega^{2} i k_{1})/\tilde{r}(k_{1})^{\frac{1}{2}} = \tilde{r}(k_{1})^{\frac{1}{2}}(\omega^{2} i k_{1}^{-1}-\omega i k_{1}) \in \R$, we infer that
\begin{align*}
\hat{n}_{3}^{(1)} = 2 i \, \im \frac{\beta_{12}(\omega i k_{1}^{-1}-\omega^{2} i k_{1}) \mathcal{P}(\zeta,\omega^{2} k_{1})}{-i k_{1} z_{\star} \sqrt{t} d_{0}e^{t \Phi_{21}(\zeta,k_{1})}\tilde{r}(k_{1})^{\frac{1}{2}} \mathcal{P}(\zeta,\omega k_{1})} + O(t^{-1}\ln t) \qquad \mbox{as } t \to \infty.
\end{align*}
Employing the identities
\begin{align*}
& \beta_{12} = \frac{\sqrt{2\pi}e^{\frac{\pi i}{4}}e^{\frac{3\pi \nu}{2}}}{q \Gamma(i\nu)}, \quad d_{0}(\zeta,t) = e^{2\pi \nu}e^{i\arg d_{0}(\zeta,t)}, \quad |q| = \sqrt{e^{-2\pi \nu}-1}, \quad \bigg|\frac{\mathcal{P}(\zeta,\omega^{2} k_{1})}{\mathcal{P}(\zeta,\omega k_{1})}\bigg| = 1, 
	\\
& |\Gamma(i\nu)| = \frac{\sqrt{2\pi}}{\sqrt{-\nu}\sqrt{e^{-\pi \nu}-e^{\pi \nu}}} = \frac{\sqrt{2\pi}}{\sqrt{-\nu}e^{\frac{\pi\nu}{2}}|q|}, \quad \frac{\omega i k_{1}^{-1}-\omega^{2} i k_{1}}{\tilde{r}(k_{1})^{\frac{1}{2}}} = -\sqrt{-1-2\cos (2\arg k_{1})}, 
\end{align*}
we obtain
\begin{align*}
\hat{n}_{3}^{(1)} \hspace{-0.07cm} =  & -2 i \frac{\sqrt{-\nu}\sqrt{-1-2\cos (2\arg k_{1})}}{-ik_{1}z_{\star}\sqrt{t}} \\
& \times \sin \Big( \frac{\pi}{4}-\arg q - \arg \Gamma(i \nu) - \arg d_{0} - \arg \frac{\mathcal{P}(\zeta,\omega k_{1})}{\mathcal{P}(\zeta,\omega^{2} k_{1})} - t \im \Phi_{21}(\zeta,k_{1}) \hspace{-0.07cm} \Big) \\
& + O(t^{-1}\ln t) \qquad \mbox{as } t \to \infty,
\end{align*}
uniformly for $\zeta \in \mathcal{I} \setminus S_{\epsilon}$. Combining the above formula with \eqref{IIbis recoverun} and using (\ref{ddzeetaPhi21}), we get
\begin{align*}
 u(x,t) = & -i\sqrt{3}\frac{\partial}{\partial x}\hat{n}_{3}^{(1)}(x,t) + O(t^{-1})  
 	\\
 = & - 2\sqrt{3} \frac{1}{t} \frac{\partial}{\partial \zeta} \bigg( \sin \Big( \frac{\pi}{4}-\arg q - \arg \Gamma(i \nu) - \arg d_{0} - \arg \frac{\mathcal{P}(\zeta,\omega k_{1})}{\mathcal{P}(\zeta,\omega^{2} k_{1})}
 	\\
& - t \im \Phi_{21}(\zeta,k_{1}) \Big) \times \frac{\sqrt{-\nu}\sqrt{-1-2\cos (2\arg k_{1})}}{-ik_{1}z_{\star}\sqrt{t}} \bigg) + O(t^{-1}\ln t), \\
= &\; \frac{A(\zeta)}{\sqrt{t}} \cos \Big( \frac{\pi}{4}-\arg q - \arg \Gamma(i \nu) - \arg d_{0} - \arg \frac{\mathcal{P}(\zeta,\omega k_{1})}{\mathcal{P}(\zeta,\omega^{2} k_{1})} 
	\\
& \hspace{3.5cm} - t \im \Phi_{21}(\zeta,k_{1}) \Big) + O(t^{-1}\ln t), \\
= &\; \frac{A(\zeta)}{\sqrt{t}} \cos \Big( \frac{3\pi}{4}+\arg r_{2}(k_{1}) + \arg \Gamma(i \nu) + \arg d_{0} + \arg \frac{\mathcal{P}(\zeta,\omega k_{1})}{\mathcal{P}(\zeta,\omega^{2} k_{1})} \\
& \hspace{3.5cm} + t \im \Phi_{21}(\zeta,k_{1}) \Big) + O(t^{-1}\ln t),
\end{align*}
as $t \to \infty$, uniformly for $\zeta \in \mathcal{I} \setminus S_{\epsilon}$, where $A(\zeta)$ is given in the statement of Theorem \ref{asymptoticsth}, and where in the last step we have used $\arg q = \pi+\arg r_{2}(k_{1}) \mod 2\pi$ and $\cos(-x)=\cos(x)$. This finishes the proof of Theorem \ref{asymptoticsth2}.

\section{The pure multi-soliton RH problem}\label{puresolitonsec}
Pure multi-soliton solutions of (\ref{boussinesq}) are obtained by considering RH problem \ref{RHn} for $n(x,t,k)$ in the special case when $r_1 = r_2 = 0$, i.e., when the jump matrix $v$ is identically equal to $I$. 
The next lemma establishes the explicit formula (\ref{usdef}) for the pure multi-solitons of (\ref{boussinesq}).

\begin{lemma}\label{uslemma}
Let $\mathsf{Z}$ be a finite subset of $D_{\mathrm{reg}} \cup (-1,0) \cup (1,\infty)$ and let $\{c_{k_0}\}_{k_0 \in \mathsf{Z}} \subset \C$ be such that 
\begin{align}\label{ck0positivitycondition}
i(\omega^{2}k_0^{2} - \omega)c_{k_0} \geq 0 \quad \text{for every $k_0 \in \mathsf{Z}\cap \mathbb{R}$}. 
\end{align}
Define $\{d_{k_0}\}_{k_0 \in \mathsf{Z}\setminus \R}$ as in (\ref{dk0def}) and write $\mathsf{Z} = \{\lambda_j\}_1^{n_b} \sqcup \{k_j\}_1^{n_s}$ where $\{\lambda_j\}_1^{n_b} := \mathsf{Z} \cap D_{\mathrm{reg}}$ and $\{k_j\}_1^{n_s} := \mathsf{Z} \cap \R$.
Then the following hold:
\begin{enumerate}[$(a)$]
\item RH problem \ref{RHn} with $v \equiv I$ has a unique solution $n(x,t,k)$ for every $(x,t) \in \R \times [0,\infty)$.

\item The function $u_s$  defined by
$$u_s\Big(x,t; (\lambda_j, c_{\lambda_j})_{j=1}^{n_b}, (k_j, c_{k_j})_{j=1}^{n_s}\Big) := -i\sqrt{3}\frac{\partial}{\partial x} \lim_{k\to\infty}k\big(n_3(x,t,k) -1\big)$$
is a real-valued Schwartz class solution of (\ref{boussinesq}) for $(x,t) \in \R \times [0, \infty)$.

\item $u_s\Big(x,t; (\lambda_j, c_{\lambda_j})_{j=1}^{n_b}, (k_j, c_{k_j})_{j=1}^{n_s}\Big)$ is given explicitly by (\ref{usdef}).

\end{enumerate}
\end{lemma}
\begin{proof}
Assertions $(a)$ and $(b)$ are a consequence of \cite[Theorem 2.6 and Remark 2.7]{CLscatteringsolitons}; we emphasize that the proof of \cite[Theorem 2.6]{CLscatteringsolitons} uses condition (\ref{ck0positivitycondition}).  

Let us prove $(c)$.
In the proof, we omit the $(x,t)$-dependence of $n(x,t,k)$ and $\theta_{ij}(x,t,k)$ for brevity.
Utilizing the residue conditions (\ref{nresiduesk0}) and (\ref{nresiduesk0real}) as well as the normalization condition $n(x,t,k) = (1,1,1) + O(k^{-1})$ as $k \to \infty$, we obtain
\begin{align}\nonumber
n_3(k) = &\;  1 
+ \sum_{l=1}^{n_b} \bigg\{ \frac{c_{\lambda_l} e^{-\theta_{31}(\lambda_l)} n_1(\lambda_l)}{k - \lambda_l} + \frac{- \lambda_l^{-2} c_{\lambda_l} e^{-\theta_{31}(\lambda_l)} n_2(\lambda_l^{-1})}{k - \lambda_l^{-1}} 
	\\ \nonumber
&\hspace{1.6cm} + \frac{\omega^2 d_{\lambda_l} e^{\theta_{32}(\bar{\lambda}_l)} n_1(\omega^2 \bar{\lambda}_l)}{k - \omega^2 \bar{\lambda}_l} + \frac{-\omega \bar{\lambda}_l^{-2} d_{\lambda_l} e^{\theta_{32}(\bar{\lambda}_l)} n_2(\omega \bar{\lambda}_l^{-1})}{k - \omega \bar{\lambda}_l^{-1}}\bigg\}
	\\ \label{M3fromresidues}
& + \sum_{l=1}^{n_s} \bigg\{  \frac{\omega^2 c_{k_l} e^{-\theta_{21}(k_l)} n_2(\omega^2 k_l)}{k -  \omega^2 k_l} + \frac{-\omega k_l^{-2} c_{k_l} e^{-\theta_{21}(k_l)} n_1(\omega k_l^{-1})}{k - \omega k_l^{-1}} \bigg\}.
\end{align}
The symmetries (\ref{nsymm}) imply that
$$n_3(k) = n_2(\omega k) = n_1(\omega^2 k) = n_3(k^{-1}),$$
so we may write (\ref{M3fromresidues}) as
\begin{align}\nonumber
n_3(k) = &\; 1 
+ \sum_{l=1}^{n_b} \bigg\{ \bigg(\frac{1}{k - \lambda_l} 
+ \frac{- \lambda_l^{-2} }{k - \lambda_l^{-1}} \bigg) c_{\lambda_l} e^{-\theta_{31}(\lambda_l)} n_3(\omega \lambda_l)
	\\
& \hspace{1.6cm}+ \bigg( \frac{\omega^2 }{k - \omega^2 \bar{\lambda}_l} + \frac{-\omega \bar{\lambda}_l^{-2} }{k - \omega \bar{\lambda}_l^{-1}}\bigg) d_{\lambda_l} e^{\theta_{32}(\bar{\lambda}_l)} n_3(\bar{\lambda}_l)\bigg\}
	\\\nonumber
& + \sum_{l=1}^{n_s} \bigg\{  \frac{\omega^2}{k -  \omega^2 k_l} + \frac{-\omega k_l^{-2} }{k - \omega k_l^{-1}} \bigg\}  c_{k_l} e^{-\theta_{21}(k_l)} n_3(\omega k_l).
\end{align}
Since
$$\frac{\frac{i(\lambda_l^2-1)}{2 \sqrt{3} \lambda_l^2}}{l_3(k) - l_3(\lambda_l)} = \frac{1}{k - \lambda_l} 
+ \frac{- \lambda_l^{-2} }{k - \lambda_l^{-1}},$$
we deduce that
\begin{align}\label{n3C1D1}
n_3(k) = &\; 1 + \sum_{l=1}^{n_b} \bigg\{ \frac{\tilde{c}_{\lambda_l} e^{-\theta_{31}(\lambda_l)} n_3(\omega \lambda_l)}{l_3(k) - l_3(\lambda_l)} 
+ \frac{\bar{\tilde{c}}_{\lambda_l} e^{\theta_{32}(\bar{\lambda}_l)} n_3(\bar{\lambda}_l)}{l_3(k) - l_3(\omega^2 \bar{\lambda}_l)}  \bigg\}
+ \sum_{l=1}^{n_s} \frac{\tilde{c}_{k_l} e^{-\theta_{21}(k_l)} n_3(\omega k_l)}{l_3(k) - l_3(\omega^2 k_l)} ,
\end{align}
where $\tilde{c}_{\lambda_l} := \frac{i(\lambda_l^2-1)}{2 \sqrt{3} \lambda_l^2} c_{\lambda_l}$ and $\tilde{c}_{k_l} := \frac{i(k_l^2-\omega^2)}{2 \sqrt{3} k_l^2} \omega^2 c_{k_l}$, and we have used that $\frac{i(\bar{\lambda}_l^2-\omega^2)}{2 \sqrt{3} \bar{\lambda}_l^2} \omega^2 d_{\lambda_l} = \bar{\tilde{c}}_{\lambda_l}$.
Evaluating (\ref{n3C1D1}) at $k = \omega \lambda_j$, $k = \bar{\lambda}_j$, and $k = \omega k_j$, we find the system
\begin{align*}
n_3(\omega \lambda_j) = &\; 1 + \sum_{l=1}^{n_b} \bigg\{\frac{\tilde{c}_{\lambda_l} e^{-\theta_{31}(\lambda_l)} n_3(\omega \lambda_l)}{l_1(\lambda_j) - l_3(\lambda_l)} 
 + \frac{\bar{\tilde{c}}_{\lambda_l} e^{\theta_{32}(\bar{\lambda}_l)} n_3(\bar{\lambda}_l)}{l_1(\lambda_j) - l_2(\bar{\lambda}_l)}  \bigg\}
+ \sum_{l=1}^{n_s} \frac{\tilde{c}_{k_l} e^{-\theta_{21}(k_l)} n_3(\omega k_l)}{l_1(\lambda_j) - l_2(k_l)} ,
	\\
n_3(\bar{\lambda}_j) =&\; 1 + \sum_{l=1}^{n_b} \bigg\{\frac{\tilde{c}_{\lambda_l} e^{-\theta_{31}(\lambda_l)} n_3(\omega \lambda_l)}{l_3(\bar{\lambda}_j) - l_3(\lambda_l)} 
+ \frac{\bar{\tilde{c}}_{\lambda_l} e^{\theta_{32}(\bar{\lambda}_l)} n_3(\bar{\lambda}_l) }{l_3(\bar{\lambda}_j) - l_2(\bar{\lambda}_l)} \bigg\}
 + \sum_{l=1}^{n_s} \bigg\{\frac{\tilde{c}_{k_l} e^{-\theta_{21}(k_l)} n_3(\omega k_l)}{l_3(\bar{\lambda}_j) - l_2(k_l)} \bigg\},
	\\
n_3(\omega k_j) =&\; 1 + \sum_{l=1}^{n_b} \bigg\{\frac{\tilde{c}_{\lambda_l} e^{-\theta_{31}(\lambda_l)} n_3(\omega \lambda_l)}{l_1(k_j) - l_3(\lambda_l)} 
 + \frac{\bar{\tilde{c}}_{\lambda_l} e^{\theta_{32}(\bar{\lambda}_l)}  n_3(\bar{\lambda}_l)}{l_1(k_j) - l_2(\bar{\lambda}_l)}\bigg\}
 + \sum_{l=1}^{n_s} \bigg\{\frac{\tilde{c}_{k_l} e^{-\theta_{21}(k_l)} n_3(\omega k_l)}{l_1(k_j) - l_2(k_l)} \bigg\}.
\end{align*}
Multiplying the first equation by $e^{x l_1(\lambda_j) + t z_1(\lambda_j)}$, the second by $e^{x l_3(\bar{\lambda}_j) + t z_3(\bar{\lambda}_j)}$, and the third by $e^{x l_1(k_j) + t z_1(k_j)}$, we obtain
\begin{align*}
(I - B) \begin{pmatrix}
e^{x l_1(\lambda_1) + t z_1(\lambda_1)}n_3(\omega \lambda_1) 
	\\
\vdots
	\\
e^{x l_1(\lambda_{n_b}) + t z_1(\lambda_{n_b})}n_3(\omega \lambda_{n_b}) 
	\\
e^{x l_3(\bar{\lambda}_1) + t z_3(\bar{\lambda}_1)} n_3(\bar{\lambda}_1) 
	\\
\vdots
	\\
e^{x l_3(\bar{\lambda}_{n_b}) + t z_3(\bar{\lambda}_{n_b})} n_3(\bar{\lambda}_{n_b})
	\\
e^{x l_1(k_1) + t z_1(k_1)} n_3(\omega k_1) 
	\\
\vdots
	\\
e^{x l_1(k_{n_s}) + t z_1(k_{n_s})} n_3(\omega k_{n_s}) 
\end{pmatrix}
= W, \;\; \text{where} \;\; W = W(x,t) := \begin{pmatrix}
e^{x l_1(\lambda_1) + t z_1(\lambda_1)}
	\\
\vdots
	\\
e^{x l_1(\lambda_{n_b}) + t z_1(\lambda_{n_b})}
	\\
e^{x l_3(\bar{\lambda}_1) + t z_3(\bar{\lambda}_1)} 
	\\
\vdots
	\\
e^{x l_3(\bar{\lambda}_{n_b}) + t z_3(\bar{\lambda}_{n_b})} 
	\\
e^{x l_1(k_1) + t z_1(k_1)}
	\\
\vdots
	\\
e^{x l_1(k_{n_s}) + t z_1(k_{n_s})}
\end{pmatrix}
\end{align*}
and $B = B(x,t)$ is the matrix defined in (\ref{Bdef}). The existence and uniqueness of $n$ imply that $I-B$ is invertible. Since, by (\ref{n3C1D1}),
\begin{align}\nonumber
\lim_{k\to\infty}k\big(n_3(k) -1\big)
= & -2i\sqrt{3}\bigg\{\sum_{l=1}^{n_b} \Big(\tilde{c}_{\lambda_l} e^{-\theta_{31}(\lambda_l)} n_3(\omega \lambda_l) 
+ \bar{\tilde{c}}_{\lambda_l} e^{\theta_{32}(\bar{\lambda}_l)} n_3(\bar{\lambda}_l)\Big)
	\\
& + \sum_{l=1}^{n_s} \tilde{c}_{k_l} e^{-\theta_{21}(k_l)} n_3(\omega k_l) \bigg\},
\end{align}
we find
\begin{align}\nonumber
\lim_{k\to\infty}k\big(n_3(k) -1\big) & =-2i\sqrt{3} V^T (I - B)^{-1}  W =-2i\sqrt{3}\tr\big\{ (I - B)^{-1} W V^T \big\},
\end{align}
where
$$V = V(x,t) := \begin{pmatrix} 
\tilde{c}_{\lambda_1} e^{-x l_3(\lambda_1) - t z_3(\lambda_1)} 
	\\
\vdots
	\\
\tilde{c}_{\lambda_{n_b}} e^{-x l_3(\lambda_{n_b}) - t z_3(\lambda_{n_b})} 
 	\\
 \bar{\tilde{c}}_{\lambda_{1}} e^{-x l_2(\bar{\lambda}_{1}) - t z_2(\bar{\lambda}_{1})} 
	\\
\vdots
	\\
 \bar{\tilde{c}}_{\lambda_{n_b}} e^{-x l_2(\bar{\lambda}_{n_b}) - t z_2(\bar{\lambda}_{n_b})} 
  	\\
 \tilde{c}_{k_{1}} e^{-x l_2(k_{1}) - t z_2(k_{1})} 
	\\
\vdots
	\\
 \tilde{c}_{k_{n_s}} e^{-x l_2(k_{n_s}) - t z_2(k_{n_s})} 
\end{pmatrix}.
$$
Since $W V^T = B_x$, this yields
$$\lim_{k\to\infty}k\big(n_3(k) -1\big)=-2i\sqrt{3}\tr\big\{ (I - B)^{-1}  B_x\big\} =2i\sqrt{3}\frac{\partial}{\partial x} \ln \det (I-B(x,t)),$$
where we have used Jacobi's formula in the second step. This implies assertion $(c)$.
\end{proof}

\subsection{Proof of Theorem \ref{usolth}} \label{usolexplicitproofsec}
The definition (\ref{usoldef}) of the leading term $u_{\sol}(x,t)$ involves the solution $M_{\sol}$ of RH problem \ref{RHS}. 
Comparing RH problem \ref{RHS} for $M_{\sol}$ with RH problem \ref{RHn}, we see that, for any fixed $(x,t) \in \R \times [0,\infty)$, $n_{\sol} := (1,1,1)M_{\sol}$ satisfies RH problem \ref{RHn} with $v \equiv I$ and with the replacements:
\begin{subequations}\label{ck0replacements}
\begin{align}
& \text{$c_{k_0}$ is replaced by $\frac{c_{k_0}}{\hat{\Delta}_{11}(\zeta, k_0)\hat{\Delta}_{33}^{-1}(\zeta, k_0)}$ for $k_0 \in \mathsf{Z} \setminus \R$, and}
	\\
& \text{$c_{k_0}$ is replaced by $\frac{c_{k_0}}{\hat{\Delta}_{11}(\zeta, k_0)\hat{\Delta}_{22}^{-1}(\zeta, k_0)}$ for $k_0 \in \mathsf{Z} \cap \R$.}
\end{align}
\end{subequations}
Indeed, $n_{\sol}$ clearly obeys the residue conditions of RH problem \ref{RHn} at every $k_0 \in \mathsf{Z}$.
Moreover, the symmetries (\ref{mathcalPDelta33symm}) imply that $\hat{\Delta}(\zeta,k) = \mathcal{B}\overline{\hat{\Delta}(\zeta,\bar{k})}^{-1}\mathcal{B}$, and hence
\begin{subequations}\label{hatDeltaconjugationsymmetries}
\begin{align}\label{hatDeltaconjugationsymmetriesa}
& \overline{\hat{\Delta}_{11}(\zeta, k)\hat{\Delta}_{33}^{-1}(\zeta, k)}
= \hat{\Delta}_{33}(\zeta, \bar{k})\hat{\Delta}_{22}^{-1}(\zeta, \bar{k}), 
	\\ \label{hatDeltaconjugationsymmetriesb}
& \overline{\hat{\Delta}_{11}(\zeta, k)\hat{\Delta}_{22}^{-1}(\zeta, k)} = \hat{\Delta}_{11}(\zeta, \bar{k})\hat{\Delta}_{22}^{-1}(\zeta, \bar{k}).
\end{align}
\end{subequations}
The relation (\ref{hatDeltaconjugationsymmetriesa}) together with the definition (\ref{dk0def}) of $d_{k_0}$  implies that $n_{\sol}$ obeys also the residue condition at $\bar{k}_0$ for every $k_0 \in \mathsf{Z}  \setminus \R$. The remaining residue conditions then follow from the symmetries in (\ref{Msolsymm}).
Furthermore, the relation (\ref{hatDeltaconjugationsymmetriesb}) implies that the positivity condition (\ref{ck0positivitycondition}) is preserved by the above replacements.
We may therefore apply Lemma \ref{uslemma} with the $c_{k_0}$ replaced as in (\ref{ck0replacements}) to immediately obtain the explicit formula (\ref{usolexplicit}) for $u_{\sol}$. This completes the proof of Theorem \ref{usolth}.

\section{The RH problem for $M_{\mathrm{sol}}$}\label{Msolsec}
In this section, we let $u_0,u_1 \in \mathcal{S}(\R)$ be real-valued initial data such that $\int_{\mathbb{R}}u_{1} dx =0$ and such that Assumptions $(i)$--$(iii)$ hold. Let $\{r_1, r_2, \mathsf{Z}, \{c_{k_0}\}_{k_0 \in \mathsf{Z}}\}$ be the associated scattering data defined by (\ref{r1r2def}) and (\ref{ck0def}).

\begin{lemma}\label{RHSlemma}
For any $(x,t)\in \R \times [0,\infty)$, the solution of RH problem \ref{RHS} exists, is unique, and satisfies $\det M_{\mathrm{sol}}(x,t,k)=1$. Moreover, the function $u_{\mathrm{sol}}(x,t)$ defined in (\ref{usoldef}) is a smooth function of $x$ and $t$.
\end{lemma}
\begin{proof}
In view of (\ref{hatDeltaconjugationsymmetries}), the coefficients in (\ref{Mk0notinR})--(\ref{Mk0inR}) obey the appropriate complex conjugation properties for \cite[Lemma A.2]{CLscatteringsolitons} to apply, which implies that the solution of RH problem \ref{RHS} exists and is unique. The identity $\det M_{\mathrm{sol}}(x,t,k)=1$ can be proved by transforming the RH problem into an equivalent RH problem in which the poles are replaced by jumps on small circles and noting that the jumps have unit determinant.
The fact that $u_{\mathrm{sol}}(x,t)$ defined in (\ref{usoldef}) is a smooth function of $x$ and $t$ follows from \cite[Theorem 2.6 and Remark 2.7]{CLscatteringsolitons}.
\end{proof}

For each $k_{0}\in \mathsf{Z}$, we let $D_{\epsilon}(k_{0})$ be a small open disk centered at $k_0$ of radius $\epsilon>0$. We let $D_{\epsilon}(k_{0})^{-1}$, $D_{\epsilon}(k_{0})^*$, and $D_{\epsilon}(k_{0})^{-*}$ be the images of $D_{\epsilon}(k_{0})$ under the maps $k \mapsto k^{-1}$, $k \mapsto \bar{k}$, and $k \mapsto \bar{k}^{-1}$, respectively. If $k_{0} \in \mathbb{R}$, then $D_{\epsilon}(k_{0})^*=D_{\epsilon}(k_{0})$. Let $\partial D_{\epsilon}(k_{0})$ and $\partial D_{\epsilon}(k_{0})^*$ be oriented counterclockwise, and let $\partial D_{\epsilon}(k_{0})^{-1}$ and $\partial D_{\epsilon}(k_{0})^{-*}$ be oriented clockwise. Let 
\begin{align}
& \mathcal{D}_{\sol} = \bigcup_{k_{0}\in \mathsf{Z}}\bigcup_{j=0,1,2} \bigg( \omega^{j} D_{\epsilon}(k_{0}) \cup \omega^{j}  D_{\epsilon}(k_{0})^{-1} \cup \omega^{j} D_{\epsilon}(k_{0})^{*} \cup \omega^{j}  D_{\epsilon}(k_{0})^{-*}\bigg), \label{def of Dsol} \\
& \partial \mathcal{D}_{\sol} = \bigcup_{k_{0}\in \mathsf{Z}}\bigcup_{j=0,1,2} \bigg( \omega^{j} \partial D_{\epsilon}(k_{0}) \cup \omega^{j}  \partial D_{\epsilon}(k_{0})^{-1} \cup \omega^{j} \partial D_{\epsilon}(k_{0})^{*} \cup \omega^{j}  \partial D_{\epsilon}(k_{0})^{-*}\bigg), \nonumber
\end{align}
so that $\partial \mathcal{D}_{\sol}$ is the union of $6 |\mathsf{Z}\cap \mathbb{R}| + 12 |\mathsf{Z}\setminus \mathbb{R}|$ small circles.
By decreasing $\epsilon>0$ if necessary, we may assume that these circles do not intersect each other. Let $\mathcal{I}$ be a compact subset of $(1,\infty)$ and let $\mathsf{P}(\zeta, k)$ be the matrix-valued function defined in (\ref{mathsfPdef}). The following uniform bound on $\hat{M}_{\mathrm{sol}}$ is used in the proof of Lemma \ref{k0lemma} and in the proof of Theorem \ref{asymptoticsth3} (which is given in Section \ref{proofth3subsec} below).

\begin{lemma}\label{lemma:Msol P is unif bounded}
For any fixed and small enough $\epsilon>0$, $\hat{M}_{\mathrm{sol}}(x,t,k):=M_{\mathrm{sol}}(x,t,k) \mathsf{P}(\zeta,k)$ and its inverse are uniformly bounded for $t \geq 2$, $\zeta \in \mathcal{I}$, and $k\in \C\setminus \mathcal{D}_{\sol}$.
\end{lemma}
\begin{proof}
First observe that $\det M_{\mathrm{sol}} \equiv 1$ by Lemma \ref{RHSlemma} and that $\det \mathsf{P} \equiv 1$ by (\ref{mathsfPdef}) and (\ref{def of mathcalP}). Thus it is enough to bound $\hat{M}_{\mathrm{sol}}$.

We know from Lemma \ref{hatMsolresiduelemma} that $\hat{M}_{\mathrm{sol}}(x,t,k)$ satisfies the residue conditions in Lemma \ref{lemma:new residue} with $n^{(3)}_{j}$ replaced by $[\hat{M}_{\mathrm{sol}}(x,t,k)]_{j}$. 
We apply a transformation $\hat{M}_{\mathrm{sol}}\to \tilde{M}_{\mathrm{sol}}$ which replaces each of these residue conditions by a jump on the corresponding circle in $\partial \mathcal{D}_{\sol}$. The function $\tilde{M}_{\mathrm{sol}}$ is defined to equal $\hat{M}_{\mathrm{sol}}$ except in $\mathcal{D}_{\sol}$ where we define $\tilde{M}_{\mathrm{sol}}(x,t,k)$ as follows: If $k_{0} \in \mathsf{Z}\setminus \mathbb{R}$, we define $\tilde{M}_{\mathrm{sol}}$ for $k \in D_{\epsilon}(k_{0}) \cup D_{\epsilon}(k_{0})^*$ by
\begin{align}\label{tildeMsoldef}
\tilde{M}_{\mathrm{sol}}(x,t,k) = \begin{cases} 
\hat{M}_{\mathrm{sol}}(x,t,k) Q_1^{<}(x,t,k)  & \text{if $k \in D_{\epsilon}(k_{0})$ and $\re \Phi_{31}(\zeta,k_0) < 0$},
	\\
\hat{M}_{\mathrm{sol}}(x,t,k) Q_7^{<}(x,t,k) & \text{if $k \in D_{\epsilon}(k_{0})^*$ and $\re \Phi_{31}(\zeta,k_0) < 0$},
	\\
\hat{M}_{\mathrm{sol}}(x,t,k) Q_1^{\geq}(x,t,k) & \text{if $k \in D_{\epsilon}(k_{0})$ and $\re \Phi_{31}(\zeta,k_0) \geq 0$},
	\\
\hat{M}_{\mathrm{sol}}(x,t,k) Q_7^{\geq}(x,t,k) & \text{if $k \in D_{\epsilon}(k_{0})^*$ and $\re \Phi_{31}(\zeta,k_0) \geq 0$},
\end{cases}
\end{align}
where
\begin{align*}
& Q_1^{<} = \begin{pmatrix} 1 & 0 & 0 \\ 0 & 1 & 0 \\ 
\frac{-c_{k_0}^{-1}e^{\theta_{31}(x,t,k_0)}}{\Delta_{33}'(\zeta,k_{0})(\Delta_{11}^{-1})'(\zeta,k_{0}) (k-k_0)}
 & 0 & 1 \end{pmatrix}, &&
 Q_7^{<} = \begin{pmatrix} 1 & 0 & 0 \\ 0 & 1 & \frac{-d_{k_0}^{-1} e^{-\theta_{32}(x,t,\bar{k}_0)}}{\Delta_{22}'(\zeta,\bar{k}_{0})(\Delta_{33}^{-1})'(\zeta,\bar{k}_{0})(k- \bar{k}_0)} 
 \\ 0 & 0  & 1 \end{pmatrix},
 	\\
& Q_1^{\geq} = \begin{pmatrix} 1 & 0 & \frac{-c_{k_{0}}e^{-\theta_{31}(x,t,k_0)} }{\Delta_{11}(\zeta,k_0) \Delta_{33}^{-1}(\zeta,k_0)(k-k_0)} 
\\ 0 & 1 & 0 \\ 0 & 0 & 1 \end{pmatrix}, &&
 Q_7^{\geq} = \begin{pmatrix} 1 & 0 & 0 \\ 0 & 1 & 0 \\ 0 & \frac{-d_{k_0} e^{\theta_{32}(x,t,\bar{k}_0)} }{\Delta_{33}(\zeta,\bar{k}_0) \Delta_{22}^{-1}(\zeta,\bar{k}_0) (k-\bar{k}_0)} 
 & 1 \end{pmatrix}, 
\end{align*}
and if $k_{0} \in \mathsf{Z} \cap \mathbb{R}$, we define $\tilde{M}_{\mathrm{sol}}$ for $k \in D_{\epsilon}(k_{0})$ by
\begin{align}\label{tildeMsoldef2}
\tilde{M}_{\mathrm{sol}}(x,t,k) = 
\begin{cases}
\hat{M}_{\mathrm{sol}}(x,t,k) P_1^{<}(x,t,k) \; & \text{if $k \in D_{\epsilon}(k_{0})$ and $\re \Phi_{21}(\zeta,k_0) < 0$},
	\\
\hat{M}_{\mathrm{sol}}(x,t,k) P_1^{\geq}(x,t,k) & \text{if $k \in D_{\epsilon}(k_{0})$ and $\re \Phi_{21}(\zeta,k_0) \geq 0$},
\end{cases}
\end{align}
where
\begin{align*}
& P_1^{<} = \begin{pmatrix} 1 & 0 & 0 \\ 
\frac{-c_{k_0}^{-1} e^{\theta_{21}(x,t,k_0)}}{\Delta_{22}'(\zeta,k_{0})(\Delta_{11}^{-1})'(\zeta,k_{0}) (k-k_0)}
& 1 & 0 \\ 0 & 0 & 1 \end{pmatrix},
&&
P_1^{\geq} = \begin{pmatrix} 1 & \frac{-c_{k_0} e^{-\theta_{21}(x,t,k_0)}}{\Delta_{11}(\zeta,k_0) \Delta_{22}^{-1}(\zeta,k_0) (k- k_0)} 
& 0 \\ 0 & 1 & 0 \\ 0 & 0 & 1 \end{pmatrix}.
\end{align*}
Here prime denotes differentiation with respect to the $k$-variable.
We then extend $\tilde{M}_{\mathrm{sol}}$ to all of $\mathcal{D}_{\sol}$ by means of the $\mathcal{A}$- and $\mathcal{B}$-symmetries (\ref{Msolsymm}).

It follows from the above definition of $\tilde{M}_{\mathrm{sol}}$ and Lemma \ref{lemma:new residue} (with $n^{(3)}_{j}$ replaced by $[\hat{M}_{\mathrm{sol}}(x,t,k)]_{j}$) that $\tilde{M}_{\mathrm{sol}}$ has no poles at the points in $\hat{\mathsf{Z}}$. Moreover, on $\partial \mathcal{D}_{\sol}$, the boundary values of $\tilde{M}_{\mathrm{sol}}$ exist, are continuous, and satisfy $(\tilde{M}_{\mathrm{sol}})_+ = (\tilde{M}_{\mathrm{sol}})_-v_{\mathrm{sol}}$, where 
\begin{align}\label{vregdef}
v_{\mathrm{sol}}(x,t,k) := \begin{cases}
Q_1^<(x,t,k) \; & \text{if $k \in \partial D_{\epsilon}(k_{0})$, $k_{0} \in \mathsf{Z}\setminus \R$, and $\re \Phi_{31}(\zeta,k_0) < 0$},
    \\
Q_7^<(x,t,k) & \text{if $k \in \partial D_{\epsilon}(k_{0})^*$, $k_{0} \in \mathsf{Z}\setminus \R$, and $\re \Phi_{31}(\zeta,k_0) < 0$},
	\\
Q_1^\geq(x,t,k) & \text{if $k \in \partial D_{\epsilon}(k_{0})$, $k_{0} \in \mathsf{Z}\setminus \R$, and $\re \Phi_{31}(\zeta,k_0) \geq 0$},
    \\
Q_7^\geq(x,t,k) & \text{if $k \in \partial D_{\epsilon}(k_{0})^*$, $k_{0} \in \mathsf{Z}\setminus \R$, and $\re \Phi_{31}(\zeta,k_0) \geq 0$},
	\\
P_1^<(x,t,k) & \text{if $k \in \partial D_{\epsilon}(k_{0})$, $k_{0} \in \mathsf{Z}\cap \R$, and $\re \Phi_{21}(\zeta,k_0) < 0$},
	\\
P_1^\geq(x,t,k) & \text{if $k \in \partial D_{\epsilon}(k_{0})$, $k_{0} \in \mathsf{Z}\cap \R$, and $\re \Phi_{21}(\zeta,k_0) \geq 0$},
\end{cases}
\end{align}
and $v_{\mathrm{sol}}$ is extended to all of $\partial \mathcal{D}_{\sol}$ by means of the $\mathcal{A}$- and $\mathcal{B}$-symmetries (\ref{vjsymm}). As in Section \ref{mainsec}, let $S_0 \subset \mathcal{I}$ be the (finite) set of all $\zeta \in \mathcal{I}$ for which there exists a $k_{0}$ such that either $k_{0} \in (\mathsf{Z}\setminus \R) \cap \{k\,|\, \re \Phi_{31}(\zeta,k) = 0\}$ or $k_{0} \in \mathsf{Z}\cap \R \cap \{k\,|\, \re \Phi_{21}(\zeta,k) = 0\}$.
In other words, $S_0$ is the set of velocities of the asymptotic solitons encoded in $u_{\mathrm{sol}}$ that lie in $\mathcal{I}$. The jump matrix $v_{\mathrm{sol}}$ is exponentially small as $t\to\infty$ with $\zeta \in \mathcal{I} \setminus S_0$ fixed. Thus, standard small-norm estimates show that $\tilde{M}_{\mathrm{sol}}(x,t,k)$, and hence also $\hat{M}_{\mathrm{sol}}(x,t,k)$, is uniformly bounded for $t \geq 2$ and $k\in \C\setminus \mathcal{D}_{\sol}$ as long as $\zeta \in \mathcal{I}$ stays a bounded distance away from $S_0$. This argument fails if $\zeta \in S_0$, because then $v_{\mathrm{sol}}$ is of order $1$.
To establish the desired estimate uniformly for all $\zeta \in \mathcal{I}$, we therefore proceed as follows.

Let $w_{\mathrm{sol}} = v_{\mathrm{sol}} - I$. Let $\mathcal{C}$ be the Cauchy operator on $\partial \mathcal{D}_{\sol}$ (defined as in (\ref{Cauchyoperatordef}) with $\hat{\Gamma}$ replaced by $\partial \mathcal{D}_{\sol}$) and let $\mathcal{C}_{w_{\mathrm{sol}}}h = \mathcal{C}_{-}(h w_{\mathrm{sol}})$. Then $\tilde{M}_{\mathrm{sol}} = I + \mathcal{C}(\mu_{\mathrm{sol}} w_{\mathrm{sol}})$ where $\mu_{\mathrm{sol}} = I + (I - \mathcal{C}_{w_{\mathrm{sol}}})^{-1}\mathcal{C}_{w_{\mathrm{sol}}} I \in I + L^{2}(\partial \mathcal{D}_{\sol})$. Thus it is enough to show that, for every sufficienty small $\epsilon > 0$, $\|(I - \mathcal{C}_{w_{\mathrm{sol}}})^{-1}\|_{\mathcal{B}(L^2(\partial \mathcal{D}_{\sol}))}$ is uniformly bounded for $t \geq 2$ and $\zeta \in \mathcal{I}$.

If $\zeta_1, \zeta_2 \in S_0$ are such that no other point in $S_0$ lies between $\zeta_1$ and $\zeta_2$, then the integers 
\begin{align}\label{integers}
|\{k_{0} \in \mathsf{Z}\setminus \R \,| \, \re \Phi_{31}(\zeta,k_0) < 0\}|\quad \text{and} \quad |\{k_{0} \in \mathsf{Z}\cap \R\,| \, \re \Phi_{21}(\zeta,k_0) < 0\}|
\end{align}
are independent of $\zeta \in [\zeta_{1},\zeta_{2})$.\footnote{The interval includes $\zeta_1$ but not $\zeta_2$ because $\re \partial_\zeta \Phi_{31}(\zeta,k_0) > 0$ for each $k_0 \in \mathsf{Z}\setminus \R$ and $\re \partial_\zeta \Phi_{21}(\zeta,k_0) > 0$ for each $k_0 \in \mathsf{Z} \cap \R$.}
Since $S_0$ is a finite set, it is therefore sufficient to prove that $\|(I - \mathcal{C}_{w_{\mathrm{sol}}})^{-1}\|_{\mathcal{B}(L^2(\partial \mathcal{D}_{\sol}))}$ is uniformly bounded for $t \geq 2$ and $\zeta \in \mathcal{J}$ whenever $\mathcal{J}$ is a subinterval of $\mathcal{I}$ such that the integers in (\ref{integers}) are independent of $\zeta \in \mathcal{J}$.

To each set $\mathsf{A} = \{A_{k_0}\}_{k_0 \in \mathsf{Z}} \subset \C^{|\mathsf{Z}|}$, we associate a jump matrix $\mathsf{v}(k; \mathsf{A})$ on $\partial \mathcal{D}_{\sol}$ by setting
\begin{align*}
\mathsf{v}(k; \mathsf{A}) := \begin{cases}
\mathsf{Q}_1^<(k; \mathsf{A}) \; & \text{if $k \in \partial D_{\epsilon}(k_{0})$, $k_{0} \in \mathsf{Z}\setminus \R$, and $\re \Phi_{31}(\zeta,k_0) < 0$ for $\zeta \in \mathcal{J}$},
    \\
\mathsf{Q}_7^<(k; \mathsf{A}) & \text{if $k \in \partial D_{\epsilon}(k_{0})^*$, $k_{0} \in \mathsf{Z}\setminus \R$, and $\re \Phi_{31}(\zeta,k_0) < 0$ for $\zeta \in \mathcal{J}$},
	\\
\mathsf{Q}_1^\geq(k; \mathsf{A}) & \text{if $k \in \partial D_{\epsilon}(k_{0})$, $k_{0} \in \mathsf{Z}\setminus \R$, and $\re \Phi_{31}(\zeta,k_0) \geq 0$ for $\zeta \in \mathcal{J}$},
    \\
\mathsf{Q}_7^\geq(k; \mathsf{A}) & \text{if $k \in \partial D_{\epsilon}(k_{0})^*$, $k_{0} \in \mathsf{Z}\setminus \R$, and $\re \Phi_{31}(\zeta,k_0) \geq 0$ for $\zeta \in \mathcal{J}$},
	\\
\mathsf{P}_1^<(k; \mathsf{A}) & \text{if $k \in \partial D_{\epsilon}(k_{0})$, $k_{0} \in \mathsf{Z}\cap \R$, and $\re \Phi_{21}(\zeta,k_0) < 0$ for $\zeta \in \mathcal{J}$},
	\\
\mathsf{P}_1^\geq(k; \mathsf{A}) & \text{if $k \in \partial D_{\epsilon}(k_{0})$, $k_{0} \in \mathsf{Z}\cap \R$, and $\re \Phi_{21}(\zeta,k_0) \geq 0$ for $\zeta \in \mathcal{J}$},
\end{cases}
\end{align*}
where
\begin{align*}
& \mathsf{Q}_1^{<} = \begin{pmatrix} 1 & 0 & 0 \\ 0 & 1 & 0 \\ 
\frac{A_{k_0}}{k-k_0}
 & 0 & 1 \end{pmatrix}, \quad
 \mathsf{Q}_7^{<} = \begin{pmatrix} 1 & 0 & 0 \\ 0 & 1 & \frac{\bar{A}_{k_0}}{k- \bar{k}_0} \frac{\omega^2 (\omega^2 - \bar{k}_0^2)}{\bar{k}_0^2-1}
  \\ 0 & 0  & 1 \end{pmatrix},
 \quad
 \mathsf{Q}_1^{\geq} = \begin{pmatrix} 1 & 0 & \frac{A_{k_0} }{k-k_0} 
\\ 0 & 1 & 0 \\ 0 & 0 & 1 \end{pmatrix}, 
	\\
&
 \mathsf{Q}_7^{\geq} = \begin{pmatrix} 1 & 0 & 0 \\ 0 & 1 & 0 \\ 0 & \frac{\bar{A}_{k_0}}{k-\bar{k}_0} \frac{\bar{k}_0^2-1}{\omega^2 (\omega^2 - \bar{k}_0^2)}
 & 1 \end{pmatrix}, 
 \quad
 \mathsf{P}_1^{<} = \begin{pmatrix} 1 & 0 & 0 \\ 
\frac{A_{k_0}}{k-k_0}
& 1 & 0 \\ 0 & 0 & 1 \end{pmatrix},
\quad
\mathsf{P}_1^{\geq} = \begin{pmatrix} 1 & \frac{A_{k_0}}{k- k_0} 
& 0 \\ 0 & 1 & 0 \\ 0 & 0 & 1 \end{pmatrix},
\end{align*}
and $\mathsf{v}$ is extended to all of $\partial \mathcal{D}_{\sol}$ by means of the $\mathcal{A}$- and $\mathcal{B}$-symmetries (\ref{vjsymm}). 

Note that $\overline{\theta_{31}(x,t,k_0)} = - \theta_{32}(x,t,\bar{k}_0)$ and that, by (\ref{Deltasymm}), $\mathsf{\Delta}_{11}(\zeta, k) = \overline{\mathsf{\Delta}_{22}(\zeta, \bar{k})}^{-1}$ and $\mathsf{\Delta}_{33}(\zeta, k) = \overline{\mathsf{\Delta}_{33}(\zeta, \bar{k})}^{-1}$.
Hence there is a compact subset $K$ of $\C^{|\mathsf{Z}|}$ such that
\begin{align}\label{vsolcontainedinvA}
\{v_{\mathrm{sol}}(x,t,\cdot) \,|\, t \geq 2, \zeta \in \mathcal{J}\} \subset 
\{\mathsf{v}(\cdot; \mathsf{A}) \,|\, \mathsf{A} \in K\}.
\end{align}
Let $U$ be the open subset of $\mathcal{B}(L^2(\partial \mathcal{D}_{\sol}))$ consisting of all invertible operators. Let $\mathsf{w} = \mathsf{v}-I$. By the existence results for RH problem \ref{RHS}, we may assume\footnote{We write ``we may assume", because $K$ cannot be chosen completely arbitrarily. Indeed, for Lemma \ref{RHSlemma} to apply, the following conditions must be satisfied: For $k_{0} \in \mathsf{Z}\cap \R$ such that $A_{k_0} \neq 0$, we must have $i(\omega^{2}k_{0}^{2} - \omega)\mathrm{c}_{k_{0}} \geq 0$, where $\mathrm{c}_{k_{0}} := -1/(\Delta_{22}'(\zeta,k_{0})(\Delta_{11}^{-1})'(\zeta,k_{0})A_{k_{0}})$ if $\re \Phi_{21}(\zeta,k_0) < 0$ for $\zeta \in \mathcal{J}$, and $\mathrm{c}_{k_{0}} := -\Delta_{11}(\zeta,k_{0})\Delta_{22}^{-1}(\zeta,k_{0})A_{k_{0}}$ if $\re \Phi_{21}(\zeta,k_0) \geq 0$ for $\zeta \in \mathcal{J}$. } that the operator $I - \mathcal{C}_{\mathsf{w}(\cdot; \mathsf{A})}$ belongs to $U$ for each $\mathsf{A} \in K$.
The map
$\mathsf{A} \mapsto \mathsf{w}(\cdot; \mathsf{A}): \C^{|\mathsf{Z}|} \to L^\infty(\partial \mathcal{D}_{\sol})$
is continuous, and the map
\begin{align*}
\mathsf{w} \mapsto I - \mathcal{C}_\mathsf{w}: L^\infty(\partial \mathcal{D}_{\sol}) \to U \subset \mathcal{B}(L^2(\partial \mathcal{D}_{\sol}))
\end{align*}
is continuous by the estimate $\|\mathcal{C}_\mathsf{w}\|_{\mathcal{B}(L^2(\partial \mathcal{D}_{\sol}))} \leq C \|\mathsf{w}\|_{L^\infty(\partial \mathcal{D}_{\sol})}$.
Moreover, the map $T \mapsto T^{-1}$ is continuous on $U$. 
By composing these maps, we see that the map that sends $\mathsf{A}$ to $\|(I - \mathcal{C}_{\mathsf{w}(\cdot; \mathsf{A})})^{-1}\|_{\mathcal{B}(L^2(\partial \mathcal{D}_{\sol}))}$ is continuous $K \to \R$ and hence bounded on $K$.
It now follows from (\ref{vsolcontainedinvA}) that $\|(I - \mathcal{C}_{w_{\mathrm{sol}}})^{-1}\|_{\mathcal{B}(L^2(\partial \mathcal{D}_{\sol}))}$ is uniformly bounded for $t \geq 2$ and $\zeta \in \mathcal{J}$.
\end{proof}

\subsection{Proof of Theorem \ref{asymptoticsth3}}\label{proofth3subsec}
Let $\zeta_0 \in S_0$ and let $\epsilon > 0$ be so small that $\zeta_0 - \epsilon > 1$ and $\mathcal{I}_0 := \{\zeta \in \R \,| \, |\zeta - \zeta_0| \leq \epsilon\}$ contains no other point in $S_0$.
Let $M_{\sol}^0$ be the unique solution of RH problem \ref{RHS} with $\mathsf{Z}$ replaced by $\mathsf{Z}^0 := \{k_0 \in \mathsf{Z} \, | \, \zeta_{k_0} = \zeta_0 \}$, i.e., $M_{\sol}^0$ solves the same RH problem as $M_{\sol}$ except that only those residue conditions that give rise to asymptotic solitons propagating at velocity $\zeta_0$ are included. 
Let $\hat{M}_{\mathrm{sol}} = M_{\mathrm{sol}} \mathsf{P}$ and $\hat{M}_{\mathrm{sol}}^0 = M_{\mathrm{sol}}^0 \mathsf{P}$, where $\mathsf{P}$ is given by \eqref{mathsfPdef}. Define $\tilde{M}_{\mathrm{sol}}$ and $v_{\sol}$ in terms of $\hat{M}_{\mathrm{sol}}$ as in the proof of Lemma \ref{lemma:Msol P is unif bounded}. Define $\tilde{M}_{\mathrm{sol}}^0$ and $v_{\sol}^0$ analogously in terms of $\hat{M}_{\mathrm{sol}}^0$. Then $v_{\sol}^0$ is the restriction of $v_{\sol}$ to the subcontour $\partial \mathcal{D}_{\sol}^0$ of $\partial \mathcal{D}_{\sol}$, where $\mathcal{D}_{\sol}^0$ is given by (\ref{def of Dsol}) with $\mathsf{Z}$ replaced by $\mathsf{Z}^0$. 

Shrinking the disks that make up $\mathcal{D}_{\sol}$ if necessary, the jump matrix $v_{\mathrm{sol}}$ is $O(e^{-ct})$ on $\partial \mathcal{D}_{\sol} \setminus \partial \mathcal{D}_{\sol}^0$ uniformly for $\zeta \in \mathcal{I}_0$ as  $t \to \infty$.
The matrix $\tilde{M}_{\mathrm{sol}} (\tilde{M}_{\mathrm{sol}}^0)^{-1}$ has no jump on $\partial \mathcal{D}_{\sol}^0$ and the jump on  $\partial \mathcal{D}_{\sol} \setminus \partial \mathcal{D}_{\sol}^0$ is given by $\tilde{M}_{\mathrm{sol}}^0 v_{\mathrm{sol}} (\tilde{M}_{\mathrm{sol}}^0)^{-1}$.
Since $\tilde{M}_{\mathrm{sol}}^0$ and $(\tilde{M}_{\mathrm{sol}}^0)^{-1}$ are uniformly bounded for $t \geq 2$, $\zeta \in \mathcal{I}_0$, and $k\in \C\setminus \mathcal{D}_{\sol}^{0}$ by Lemma \ref{lemma:Msol P is unif bounded}, we conclude that the jump of $\tilde{M}_{\mathrm{sol}} (\tilde{M}_{\mathrm{sol}}^0)^{-1}$ is exponentially small as  $t \to \infty$ uniformly for $\zeta \in \mathcal{I}_0$. Consequently, by standard estimates for small-norm RH problems,
$$\lim_{k\to\infty}k\Big(\big((1,1,1)M_{\mathrm{sol}}(x,t,k)\big)_{3}-1\Big)
= \lim_{k\to\infty}k\Big(\big((1,1,1)M_{\mathrm{sol}}^0(x,t,k)\big)_{3}-1\Big) + O(e^{-ct})$$
as $t \to \infty$ uniformly for $\zeta \in \mathcal{I}_0$, and this formula can be differentiated with respect to $x$ without changing the error term.
By (\ref{usoldef}) and Theorem \ref{usolth} (applied with $\mathsf{Z}$ replaced by $\mathsf{Z}^0$), this implies that
\begin{align*}
u_{\sol}(x,t) = &\; u_s\bigg(x,t; \bigg(\lambda_j, \frac{c_{\lambda_j}}{\hat{\Delta}_{11}(\zeta, \lambda_j)\hat{\Delta}_{33}^{-1}(\zeta, \lambda_j)}\bigg)_{j \in I_0}, \bigg(k_j, \frac{c_{k_j}}{\hat{\Delta}_{11}(\zeta, k_j)\hat{\Delta}_{22}^{-1}(\zeta, k_j)}\bigg)_{j \in J_0}\bigg)
	\\
& + O(e^{-ct})
\end{align*}
as $t \to \infty$ uniformly for $\zeta \in \mathcal{I}_0$.
Substituting this formula into (\ref{uasymptotics}), we arrive at the asymptotic formula (\ref{uasymptoticsnearsoliton}).

Finally, note that if $C_0 >0$, then
$$\frac{c_{k_j}}{\hat{\Delta}_{11}(\zeta, k_j)\hat{\Delta}_{22}^{-1}(\zeta, k_j)}
 =  \frac{c_{k_j}}{\hat{\Delta}_{11}(\zeta_{k_j}, k_j)\hat{\Delta}_{22}^{-1}(\zeta_{k_j}, k_j)} 
 + O(t^{-1}) \qquad \text{as $t \to \infty$}$$
uniformly for $x \in [\zeta_{k_j} t - C_0, \zeta_{k_j} t + C_0]$. Hence the asymptotic formula (\ref{usechasymptotics}) follows from (\ref{usonesoliton}) and (\ref{uasymptoticsnearsoliton}).
The proof of Theorem \ref{asymptoticsth3} is complete.

\appendix

\section{A model RH problem}\label{modelapp}

The model RH problem considered in this appendix is needed for the local parametrix near $k_{1}$. The solution of this RH problem is constructed in terms of parabolic cylinder functions as in \cite{I1981}. We omit the proof here, but refer to \cite[Appendix A]{CLsectorV} for the construction of the solution to a similar (but more complicated) RH problem.

Let $X = X_1 \cup \cdots \cup X_4 \subset \C$ be the cross defined by
\begin{align} \nonumber
&X_1 = \bigl\{se^{\frac{i\pi}{4}}\, \big| \, 0 \leq s < \infty\bigr\}, && 
X_2 = \bigl\{se^{\frac{3i\pi}{4}}\, \big| \, 0 \leq s < \infty\bigr\},  
	\\ \label{Xdef}
&X_3 = \bigl\{se^{-\frac{3i\pi}{4}}\, \big| \, 0 \leq s < \infty\bigr\}, && 
X_4 = \bigl\{se^{-\frac{i\pi}{4}}\, \big| \, 0 \leq s < \infty\bigr\},
\end{align}
and oriented away from the origin, see Figure \ref{X.pdf}. 

\begin{figure}
\begin{center}
 \begin{overpic}[width=.35\textwidth]{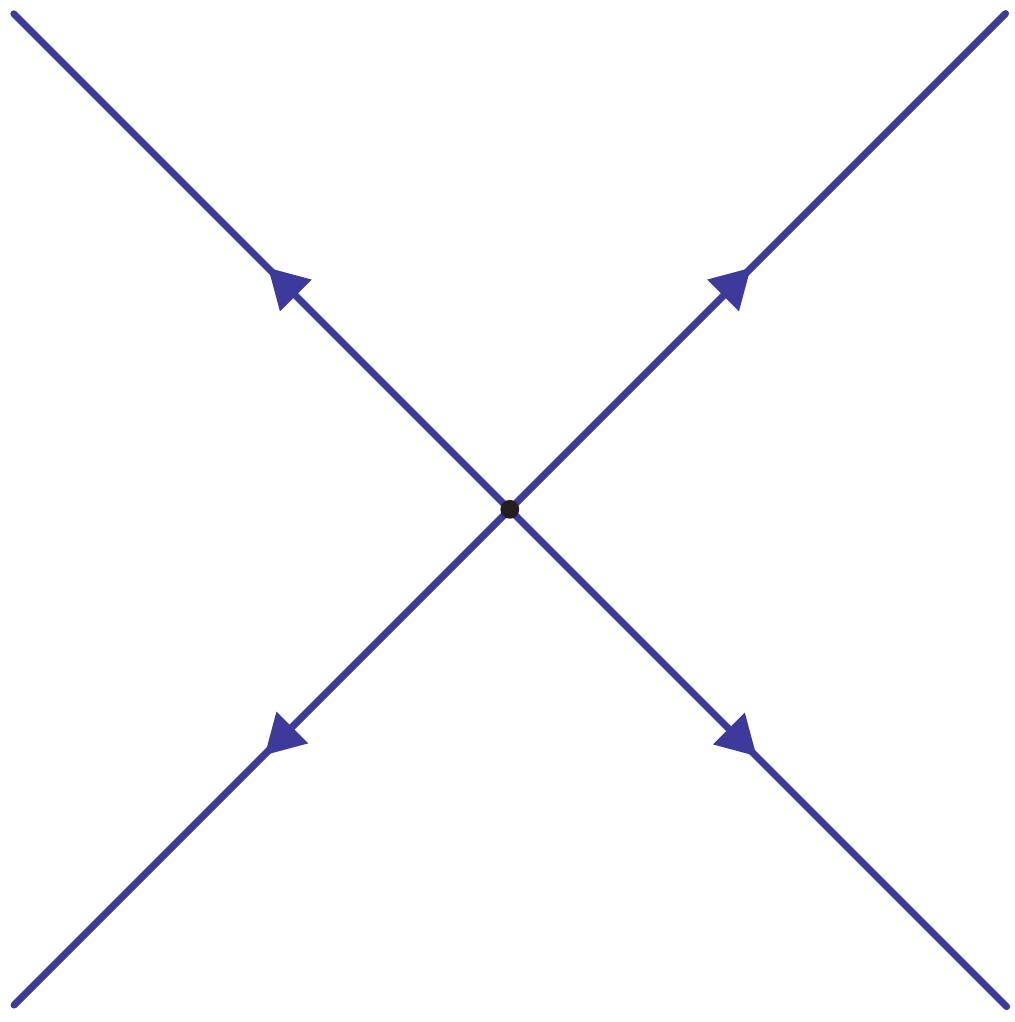}
      \put(74,68){\small $X_1$}
      \put(18,68){\small $X_2$}
      \put(17,28){\small $X_3$}
      \put(75,27){\small $X_4$}
      \put(48,43){$0$}
    \end{overpic}
     \begin{figuretext}\label{X.pdf}
        The contour $X = X_1 \cup X_2 \cup X_3 \cup X_4$.
     \end{figuretext}
     \end{center}
\end{figure}

\begin{lemma}[Model RH problem needed near $k=k_{1}$]\label{IIbis Xlemma 3}
Let $q \in \mathbb{C}$, and define
\begin{align*}
\nu = -\tfrac{1}{2\pi} \ln(1 + |q|^2) \leq 0.
\end{align*}
Define the jump matrix $v^{X}(z)$ for $z \in X$ by
\begin{align}\nonumber 
& \begin{pmatrix} 
1 & \frac{-\bar{q}}{1 + |q|^2} z_{(0)}^{-2i\nu}e^{\frac{iz^2}{2}} & 0 \\
0	& 1 & 0 \\
0 & 0 & 1
\end{pmatrix} \mbox{ if } z \in X_{1}, & & 
 \begin{pmatrix} 
1	& 0 & 0	\\
q z_{(0)}^{2i\nu}e^{-\frac{iz^2}{2}}	& 1 & 0 \\
0 & 0 & 1
\end{pmatrix} \mbox{ if } z \in X_{2}, 
	\\ \label{vXdef}
& \begin{pmatrix} 
1 & \bar{q} z_{(0)}^{-2i\nu} e^{\frac{iz^2}{2}} & 0 \\
0 & 1 & 0 \\
0 & 0 & 1
\end{pmatrix} \mbox{ if } z \in X_{3}, & & \begin{pmatrix} 
1 & 0 & 0 \\
\frac{-q}{1 + |q|^2}z_{(0)}^{2i\nu} e^{-\frac{iz^2}{2}} & 1 & 0 \\
0 & 0 & 1 \end{pmatrix} \mbox{ if } z \in X_{4},
\end{align}
where $z_{(0)}^{i\nu}$ has a branch cut along $[0,+\infty)$, such that $z_{(0)}^{i\nu} = |z|^{i\nu}e^{-\nu  \arg_{0}(z)}$, $\arg_{0}(z) \in (0,2\pi)$. Then the RH problem 
\begin{enumerate}[$(a)$]
\item $m^{X}(\cdot) = m^{X}(q, \cdot) : \C \setminus X \to \mathbb{C}^{3 \times 3}$ is analytic;

\item on $X \setminus \{0\}$, the boundary values of $m^{X}$ exist, are continuous, and satisfy $m_+^{X} =  m_-^{X} v^{X}$;

\item $m^{X}(z) = I + O(z^{-1})$ as $z \to \infty$, and $m^{X}(z) = O(1)$ as $z \to 0$;
\end{enumerate}
has a unique solution $m^{X}(q,z)$. This solution satisfies
\begin{align}\label{mXasymptotics}
m^{X}(z) = I + \frac{m_{1}^{X}}{z} + O\biggl(\frac{1}{z^2}\biggr), \quad z \to \infty,  \quad m_{1}^{X}:=\begin{pmatrix} 
0 & \beta_{12} & 0 \\ 
\beta_{21} & 0 & 0 \\
0 & 0 & 0 \end{pmatrix},
\end{align}  
where the error term is uniform for $\arg z \in [-\pi, \pi]$ and $q$ in compact subsets of $\mathbb{C}$, and
\begin{align}\label{betaXdef}
& \beta_{12} := \frac{\sqrt{2\pi}e^{\frac{\pi i}{4}}e^{\frac{3\pi \nu}{2}}}{q \Gamma(i\nu)}, \qquad \beta_{21} := \frac{\sqrt{2\pi}e^{-\frac{\pi i}{4}}e^{-\frac{5\pi \nu}{2}}}{-\bar{q} \Gamma(-i\nu)}.
\end{align}
(Note that $\beta_{12}\beta_{21} = \nu$ because $|\Gamma(i\nu)| = \frac{\sqrt{2\pi}}{\sqrt{-\nu}\sqrt{e^{-\pi\nu}-e^{\pi\nu}}}=\frac{\sqrt{2\pi}}{\sqrt{-\nu}e^{\frac{\pi\nu}{2}}|q|}$.)
\end{lemma}

\subsection*{Acknowledgements}
Support is acknowledged from the Novo Nordisk Fonden Project, Grant 0064428, the European Research Council, Grant Agreement No. 682537, the Swedish Research Council, Grant No. 2015-05430, Grant No. 2021-04626, and Grant No. 2021-03877, the G\"oran Gustafsson Foundation, and the Ruth and Nils-Erik Stenb\"ack Foundation.

\subsection*{Declarations of interest} None.

\bibliographystyle{plain}
\bibliography{is}

\end{document}